\documentclass[10pt,twoside,b5paper]{article}
\newcommand{\hidetag}[1]{}\usepackage{graphics}
\usepackage{graphicx}
\usepackage{color}
\usepackage{ulem}
\usepackage{mathrsfs}
\usepackage[utf8]{inputenc}
\usepackage[nottoc,numbib]{tocbibind}

\usepackage{amsfonts}

\usepackage{amsmath}
\usepackage{amstext}
\usepackage{amsopn}
\usepackage{amsbsy}
\usepackage{amscd}
\usepackage{amsxtra}
\usepackage{amsthm}
\usepackage{amssymb}
\usepackage{esint}

\usepackage{epsfig}
\usepackage{pstricks,multido,pst-node}

\usepackage{bibspacing}
\usepackage{fancyhdr}

\numberwithin{equation}{section}

\title{Thin layer analysis of a non-local model for the double layer structure}

\author{Chiun-Chang Lee \thanks{Institute for Computational and Modeling Science \& Department of Applied Mathematics, National Tsing Hua University (Nanda Campus), Hsinchu 30014, Taiwan ({\tt chlee@mail.nd.nthu.edu.tw}).}}

\date{}

\begin{document}
\allowdisplaybreaks
\maketitle

\newtheorem{theorem}{\large T\small HEOREM\large}[section]
\newtheorem{lemma}[theorem]{\large L\small EMMA\large}
\newtheorem{remark}{\large R\small EMARK\large}
\newtheorem{corollary}[theorem]{\large C\small OROLLARY\large}
\newtheorem{proposition}[theorem]{\large P\small ROPOSITION\large}
\newtheorem{definition}{\large D\small EFINITION\large}

\theoremstyle{theorem}
\newtheorem*{thma}{\large T\small HEOREM \large A1}
\theoremstyle{case}
\newtheorem{case}{\large C\small ASE\large}[subsection]
\newtheorem*{casea}{\large C\small ASE\large\ 1}
\newtheorem*{caseb}{\large C\small ASE\large\ 2}

\newcommand{\phie}{\phi_\epsilon}
\newcommand{\psie}{\psi_\epsilon}
\newcommand{\sumk}{\sum_{k=1}^{N_1}}
\newcommand{\suml}{\sum_{l=1}^{N_2}}
\newcommand{\eapx}{e^{a_k\phie(x)}}
\newcommand{\ebpx}{e^{-b_l\phie(x)}}
\newcommand{\eapy}{e^{a_k\phie(y)}dy}
\newcommand{\ebpy}{e^{-b_l\phie(y)}dy}
\newcommand{\intt}{\int_{-1}^1}
\newcommand{\etae}{\eta_\epsilon}
\newcommand{\ds}{\displaystyle}
\newcommand{\xint}{\fint_{\Omega}}
\newcommand{\Xint}{\fint_{\Omega}}
\newcommand{\dx}{\,\mathrm{d}x}
\newcommand{\dy}{\,\mathrm{d}y}
\newcommand{\dr}{\,\mathrm{d}r}
\newcommand{\dss}{\,\mathrm{d}s}
\newcommand{\dt}{\,\mathrm{d}t}
\newcommand{\dsx}{\,\mathrm{dS}_x}

\begin{abstract}
For the structure of the thin electrical double layer~(EDL) and the property related to the EDL capacitance, we analyze boundary layer solutions (corresponding to the electrostatic potential) of a non-local elliptic equation which is a steady-state Poisson--Nernst--Planck equation with a singular perturbation parameter related to the small Debye screening length. Theoretically, the boundary layer solutions describe that those ions exactly approach neutrality in the bulk, and the extra charges are accumulated near the charged surface. Hence, the non-neutral phenomenon merely occurs near the charged surface. To investigate such phenomena, we develop new analysis techniques to investigate thin boundary layer structures. A series of fine estimates combining the Poho\v{z}aev's identity, the inverse H\"{o}lder type estimates and some technical comparison arguments are developed in arbitrary bounded domains. Moreover, we focus on the physical domain being a ball with the simplest geometry and gain a clear picture on the effect of the curvature on the boundary layer solutions. The content involves three contributions. The first one focuses mainly on the boundary concentration phenomena. We show that the net charge density behaves exactly as Dirac delta measures concentrated at boundary points. The second one is devoted to pointwise descriptions with curvature effects for the thin boundary layer. An interesting outcome shows that the significant curvature effect merely occurs in the part of the boundary layer close to the boundary, and this part is extremely thinner than the whole boundary layer. The third contribution gives a connection to the EDL capacitance. We provide a theoretical way to support that the EDL has higher capacitance in a quite thin region near the charged surface, not in the whole EDL. In particular,  for the cylindrical electrode, our result has a same analogous measurement as the specific capacitance of the well-known Helmholtz double layer. 
\end{abstract}

\footnotesize
\textbf{Mathematics Subject Classification.}   34E05 $\cdot$ 35J67 $\cdot$ 35Q92 $\cdot$	35R09.\\ 

\textbf{Keywords.}  Non-local electrostatic model, Boundary concentration phenomenon,
Pointwise description, Boundary curvature effect, Electrical double layer 
 capacitance.  \\ 

\normalsize

\noindent 

\tableofcontents

\section{Introduction} 
The electrical double layer (EDL) is an important electrochemical phenomenon involving the transfer of electrons/ions in semiconductors, electrokinetic fluids and electrolyte solutions~(cf. \cite{BCB,BS-C-M,Eb1,Eb2,HN,MB}). In particular, in electrolyte solutions the EDL merely has the thickness of nanometer and behaves as a novel energy storage device obeying a high-capacity electrochemical capacitor~(usually called the supercapacitor~\cite{SG}). A boundary layer problem for the Poisson--Nernst--Planck (PNP) model with small Debye screening length~\cite{ZBL} arises in order to deal with the micro-phenomenon of the EDL, and an important approach for investigating the structure of the thin EDL is played by the \textit{boundary layer solution} (corresponding to the electrostatic potential) of a steady-state (time independence) PNP equation~(cf. \cite{BJ,HKL,Lw,SGNE}); that is the charge-conserving Poisson--Boltzmann (CCPB) equation~\cite{Sh1,WXLS}. Here boundary layer solutions mean that the profiles of solutions form boundary layers near boundary points and become flat in the interior domain (cf. \cite{LHLL2,ss1992,s2003,s2004}).


The aim of this work is threefold --~(i)~the pointwise description of the 
net charge density (introduced in Section~\ref{sec1-1dpb}) near the charged surface; (ii)
the effect of the charged surface curvature on thin EDL structures; (iii)~a physical connection to the curvature effects on the EDL capacitance.
To study (i)--(iii),
we consider the dimensionless model with a singular perturbation parameter related to a small Debye screening length and develop rigorous analysis techniques.
In Section~\ref{secMP},
we shall start with the motivation and the basic issue treated in this work.
It is worth stressing that the boundary layer solution of the dimensionless model is a quite crucial way to 
study the curvature effect on the thin EDL structure. 
In order for readers to gain a clear picture of this work,
in Section~\ref{sec1-1dpb} we go back to
the dimensionless formulation of the model and
the related background information such as the net charge density
and the EDL capacitance.
 In Section~\ref{sec1-3da}, as a summary of the main theorems,
we shall review previous related works~\cite{L2,LHLL1,LHLL2,RLW}
and summarize the main contributions and technical analysis for the present work.
The main theorems 
and their corresponding proofs will be stated in Sections~\ref{sec0416addsec}--\ref{sec-phys-app}.

\subsection{Background and problem formulation}\label{secMP}
 The concept of the EDL was proposed firstly 
by Helmholtz in 1879, and has been fully treated theoretically
by Grahame~(see, e.g., \cite{Gdc2}) in the 1950s.
The EDL is a structure of charge accumulation/separation formed at the
charged surface~(metal charged surfaces, calcium silicates, electrode surfaces, or other
solid materials at the surface~\cite{BF}).
Due to the property of the charged surface and the effect of the electric field on ion transport, 
the EDL attaching to the charged surface usually
presences in a nanoscale region~\cite{KPLM}. In particular, 
the charged surface shape (geometry structure)
plays a crucial role on the behavior of the electrostatic potential in the EDL~(cf. \cite{DKU,FJC,FQHDSM}).
For more than half a century, some electrostatic models such as
 Poisson--Boltzmann~(PB) type equations
are widely developed for exploring the behavior of the electrostatic potential
in the nanometer-scale EDL; see, e.g., \cite{FG,GP,LHLL1,LHLL2,LE,NS,R}
and references therein.
Due to their reasonable success for describing the behavior of the EDL,
 a crucial issue on the EDL structure arises concerning
 the boundary curvature effect on boundary layer solutions of these models.

Indeed, in order to develop theoretical frameworks for the EDL,
 most PB type equations 
have been investigated via numerical simulations~\cite{BAO,GaPr,HEL,Sh1,WP,WXLS,ZWL}, 
multiple-scale method~\cite{ST}
and the method of matched asymptotic expansions~\cite{SGNE}.
Such frameworks focus mainly on the linear PB theory (Debye--H\"{u}ckel approximation~\cite{DH1,DH2}).
Some fundamental works describing the EDL structure
in ``symmetric electrolytes" via \textit{nonlinear} PB type models can also be found in~\cite{FG,L2,LHLL1,LHLL2,RLW}.
Despite the importance of these works,
 the investigation for the influence of the mean curvature of the charged surface on the EDL structure
at a microscopic level
 is rarely to be acquired in literatures. 
Foremost among reasons is that in the electrolyte 
environment, the width of the thin EDL
is sufficiently small compared with the characteristic length of the physical domain.
Hence, under the physical dimensionless formulation,
these PB type equations become \textit{singularly perturbed models}
with a singular perturbation parameter associated with the ratio of the Debye screening length
(the characteristic length of the EDL)
over the characteristic length of the physical domain~\cite{BCB,GP,L3,Lw}.
Consequently, the related issue 
is exactly a boundary layer problem 
for boundary layer solutions of those singularly perturbed PB models.

To describe clearly the curvature effects on the thin EDL,
it suffices to establish \textit{pointwise descriptions} for boundary layer solutions near the boundary.
Based upon \cite{LHLL1,LHLL2,RLW,Sh1,WXLS} for boundary layer solutions of the CCPB equations,
in this work we focus mainly on an electrolyte environment involving the mixture of 
one anion species with the charge valence $-pe_0$ and
one cation species with the charge valence $+qe_0$ ($p,q>0$ and $e_0$ is the elementary charge).
 Let $\Omega$ be a bounded domain in $\mathbb{R}^N$ ($N\geq2$),
where the boundary~$\partial\Omega$ associated with the charged surface is smooth.
The model can be represented as 
\begin{align}
\epsilon^2\nabla\cdot\left(g\nabla{u_\epsilon}\right)=\frac{Ae^{pu_\epsilon}}{\Xint\,e^{pu_\epsilon(y)}\dy}-
\frac{Be^{-qu_\epsilon}}{\Xint\,e^{-qu_\epsilon(y)}\dy}\quad\mathrm{in}\,\,\Omega,&\label{eq1}
\end{align}
and $u_\epsilon$ is imposed by the Neumann boundary condition
\begin{align}
\epsilon^2\partial_{\nu}u_\epsilon=\mathrm{C}_{\mathrm{bd}}\quad\mathrm{on}\,\,\partial\Omega.&\label{bd1}
\end{align}
Here, $0<\epsilon\ll1$ is a singular perturbation parameter; 
the unknown variable~$u_\epsilon\equiv{u}_\epsilon(x)$ stands for the electrostatic potential;
$g\equiv{g}(x)\in C^\infty(\overline{\Omega})$ is a positive smooth function independent of $\epsilon$;
$\nabla\cdot\left(g\nabla{\,\,}\right):=\sum_{i=1}^N\frac{\partial}{\partial{x_i}}
\left(g\frac{\partial{\,\,}}{\partial{x_i}}\right)$,
$\partial_{\nu}:=\sum_{i=1}^{N}\nu_i\frac{\partial}{\partial{x_i}}$ denotes the normal derivative at the boundary~$\partial\Omega$ in the outward direction $\nu=(\nu_1,...,\nu_N)$,
and $\xint:=\frac{1}{|\Omega|}\int_{\Omega}$ denotes the average integral over $\Omega$
(where $|\Omega|$ stands for the Lebesgue measure of $\Omega$ in $\mathbb{R}^N$);
$A$, $B$, $p$ and $q$ are positive constants independent of $\epsilon$.
For simplicity, the surface charge $\mathrm{C}_{\mathrm{bd}}$ is assumed a constant 
which is uniquely determined by $A$, $B$ and $g$ as follows:
\begin{align}\label{ce}
\mathrm{C}_{\mathrm{bd}}=\frac{|\Omega|\left(A-B\right)}{\int_{\partial\Omega}g\dsx}.
\end{align}
Indeed,
integrating equation~(\ref{eq1}) over $\Omega$, one formally obtains
$\epsilon^2\int_{\partial\Omega}g\partial_{\nu}u_\epsilon{\dsx}=|\Omega|(A-B)$.
Along with (\ref{bd1}) yields (\ref{ce}).

Note also that equation~(\ref{eq1}) with the Neumann boundary condition~(\ref{bd1})
satisfies the shift-invariance.
Without loss of generality, we would make a constraint
\begin{align}\label{1129-510p}
\int_{\Omega}u_\epsilon(x)\dx=0\quad\mathrm{for\,\,the\,\,Neumann\,\,boundary\,\,condition~(\ref{bd1})}. 
\end{align}

On the other hand,  
we also study the equation~(\ref{eq1}) with a Robin boundary condition
\begin{align}\label{rb-bdy}
u_\epsilon+\eta_{\epsilon}\partial_{\nu}u_\epsilon=0\quad\mathrm{on}\,\,\partial\Omega,
\end{align}
where $\eta_{\epsilon}$ is a non-negative parameter depending on $\epsilon$. Note that
(\ref{rb-bdy}) is equivalent to $u_\epsilon+\eta_{\epsilon}\partial_{\nu}u_\epsilon=k$
due to a shift-invariance of (\ref{eq1}) under the replacement
 of $u_\epsilon$ by $u_\epsilon+k$ for any constant~$k$.

The physical background and the dimensionless formulation of (\ref{eq1}), 
the Neumann boundary condition~(\ref{bd1}) and
the Robin boundary condition~(\ref{rb-bdy})  
will be introduced in Section~\ref{sec1-1dpb}.
In summary,
we classify these two problems as follows:
\begin{itemize}
\item[\textbf{(N).}] \textbf{The Neumann problem:} The equation~(\ref{eq1}) with the Neumann boundary condition~(\ref{bd1})
and constraints~(\ref{ce}) and (\ref{1129-510p}).
\item[\textbf{(R).}] \textbf{The Robin problem:} The equation~(\ref{eq1}) with the Robin boundary condition~(\ref{rb-bdy}).
\end{itemize}

Existence and uniqueness for the problem~(R)
has been proven in Theorem~1.1 of \cite{L1}.
Following the same argument,
one can show that the problem~(N)
has a unique solution~$u_\epsilon\in{C^\infty(\Omega)\cap{C^1(\overline{\Omega})}}$.
In particular,
when $A=B$, we have $\mathrm{C}_{\mathrm{bd}}=0$ due to (\ref{ce}),
and the uniqueness immediately implies that (\ref{eq1})--(\ref{bd1})
only has a trivial solution~$u_\epsilon\equiv0$.
The same situation also holds for the problem~(R).
Hence, in order to avoid trivial cases and to study asymptotics of non-trivial solutions of this model
as $\epsilon$ approaches zero, 
we assume
\begin{align}\label{aneqb0902}
A\neq{B}. 
\end{align}
Condition~(\ref{aneqb0902}) means \textit{non-electroneutrality}~\cite{LC}; 
 that is,
the total concentration of anion species is not equal to
that of cation species in the whole electrolyte system.
Investigating problems (N) and (R)
is of quite importance in physics and of great value in mathematical researches because in electrolyte solutions
 non-electroneutral phenomena may occur near the charged surface
 when systems involve empirical tests of
boundary conditions~(\ref{bd1}) and (\ref{rb-bdy}) at charged surfaces to take into account of electrochemical reactions. 
For example, 
the Faradaic current is driven by redox reactions occurring
at the surface of electrodes where the non-electroneutrality phenomenon appears naturally (cf. \cite{BF,LC,RWL}). 
Due to the previous works~\cite{L1,L2,LHLL1,LHLL2,RLW} on the non-electroneutral case, 
 a corresponding mathematical problem is to investigate
boundary concentration phenomena for solutions of problems (N) and (R) as $\epsilon$ approaches zero.

Physically,
the thin EDL is a high-capacity electrochemical capacitor (usually called 
 the \textit{supercapacitor}) which can be viewed as a novel energy storage device~\cite{BBZ,RZSMB,SG}.
The corresponding voltage-dependent capacitance is influenced by a number of factors
including the permittivity and the charged surface curvature~\cite{DKU,FJC}.
It is expected that for the EDL, the curvature effect on the capacitance in the region near the charged surface
is more strongly than that near the bulk.
However, such a theoretical description seems unclear, which motivates us to study the curvature effects on these two distinct regions within the thin EDL.
Investigating the effect of the mean curvature of the charged surface on boundary layer solutions of these models
seems to be of a great challenge.

Hence, the main goal of this paper is to develop rigorous analysis for  
the structure of boundary layer solutions in problems~(N) and (R) under the constraint~(\ref{aneqb0902}). Moreover,
according to these mathematical results
we may provide theoretical applications
on calculating the electrostatic potential, the net charge density~$\rho_{\epsilon}$
and the capacitance~$\mathscr{C}_{\epsilon}$
including the curvature effect
in the thin EDL region (these physical quantities will be introduced in Section~\ref{sec1-1dpb}). 

Equation (\ref{eq1}) is a non-local one because the appearance of 
those terms $\int_{\Omega}e^{pu_\epsilon(y)}\dy$ and $\int_{\Omega}e^{-qu_\epsilon(y)}\dy$'s	
 imply that (\ref{eq1}) is not a
pointwise identity. This causes some mathematical difficulties making the study of (\ref{eq1}) particularly
interesting. In order to better communicate our ideas presented in Section~\ref{sec1-3da},
 we shall state some difference between this model and other singularly perturbed models
and point out some difficulties in treating this problem.
Firstly, we compare model~(\ref{eq1})--(\ref{bd1}) with standard singularly perturbed semilinear elliptic models
$\epsilon^2\Delta\phi_\epsilon=f(\phi_\epsilon)$
in a bounded domain ($\Delta$ stands for the usual Laplace operator in $\mathbb{R}^N$, and $\phi_\epsilon$ is imposed by zero Dirichlet boundary conditions).
 We refer the reader to some pioneering works~\cite{Fpc,Hfa,LN}.
These investigations focus mainly on the existence of positive solutions
under some specific assumptions of $f$; e.g., $f$ satisfies the logistic-type nonlinearity. Furthermore,
they established various arguments to show that $\phi_\epsilon$ is uniformly bounded to $\epsilon$, and
the outward normal derivatives $\partial_{\nu}\phi_\epsilon$ at the boundary points are bounded by the quantity $\epsilon^{-1}$
as $0<\epsilon\ll1$.
However, for model~(\ref{eq1}) with the boundary condition~(\ref{bd1}) and constraints~(\ref{ce}) and (\ref{aneqb0902}),
those non-local coefficients depend on the unknown solution, and $\partial_{\nu}u_{\epsilon}\sim\epsilon^{-2}$
is strongly singular on $\partial\Omega$ as $\epsilon\downarrow0$.
Hence, this model is totally different from the standard singularly perturbed elliptic models,
and those arguments are not to be applied to problems (N) and (R).

We also want to point out again that the non-local coefficients of (\ref{eq1})
can be regarded as extra parameters varying with respect to the major parameter~$\epsilon$.
As $\epsilon$ approaches zero, 
the asymptotic behavior of these non-local coefficients
seem not easy to be obtained in an intuitive way, and
 traditional asymptotic analysis technique and the multiple-scale method
are also difficult to capture the asymptotic blow-up behavior of boundary layer solutions of this model.


Let us review some related works.
For the CCPB model~(\ref{eq1}),
the previous works merely investigated the asymptotic behaviors of boundary layer solutions  
in the case of $p=q$, i.e., the case of symmetric electrolyte solutions. 
In \cite{L2,LHLL1,LHLL2},
the one-dimensional (1D) solutions of equation~(\ref{eq1}) subject to Robin boundary conditions
have been studied.
When (\ref{aneqb0902}) holds, we showed that the 1D solution 
 asymptotically blows up (for $\epsilon\downarrow0$) at boundary points,
 and established the exact leading order terms for the boundary asymptotic behavior~(see Theorem~1.6 of \cite{LHLL1}
and Theorem~1.5 of \cite{LHLL2}).
From the physical point of view, these results support
a phenomenon that the ionic distribution approaches electroneutrality
in the bulk, and the extra charges (associated with $|A-B|$) are accumulated near the charged surface
and develops boundary concentration phenomena~(see Theorem~1.5 of \cite{LHLL1}).
As a consequence, non-electroneutral phenomenon merely occurs near the charged surface.

Although 1D solutions for the case $p=q$ provide basic understanding on 
the behavior of ion transport and ion--ion interaction in the EDL, 
when $p\neq{q}$, their arguments merely give different lower and upper bounds for those non-local coefficients
and cannot determine the exact leading order terms for asymptotic expansions of boundary layer solutions.
In~\cite{L2}, for the high-dimensional CCPB model~(\ref{eq1}) with $p=q$,
we established a lower-bound estimate for the boundary blow-up behavior of
 solutions (cf. Theorem~1.1 of \cite{L2}). 
The main argument is based on the corresponding Poho\v{z}aev's identity~(cf. Lemma~4.2 of \cite{L2})
and technical comparison theorems. As an extension of \cite{L2}, in this work we also apply these arguments to problems~(N) and (R)
under the constraint~$(p-{q})(A-B)\geq0$; see Theorem~\ref{thm1} 
in Section~\ref{sec1-3da}. We stress that  
this result makes an improvement for Theorem~1.1 of \cite{L2}.
However, because of the limitation of these comparison arguments,
the approaches established in \cite{L2,LHLL1,LHLL2}
seem difficult to deal with solutions of problems (N) and (R) in the situation~$(p-{q})(A-B)<0$
(cf. Remark~\ref{20170101rk} in Section~\ref{gentherm-sec}). 
Another difficulty for studying this model comes from 
these non-local coefficients,
because as $\epsilon$ approaches zero, the limits of  
these non-local coefficients of (\ref{eq1}) seem unpredictable.
In the present work, we shall deal with the case of $p\neq{q}$ and establish an \textit{inverse H\"{o}lder 
type estimate}  (cf. Lemma~\ref{lem1}) for those non-local coefficients to study the asymptotics 
of boundary layer solutions~$u_\epsilon$. As will be clear from our main results presented below,
we can precisely depict the first-two order terms for
 asymptotic expansions of $u_\epsilon$ near the boundary as $\epsilon\downarrow0$.

We shall stress that in recent related works,
the boundary asymptotic behavior of boundary layer solutions $u_\epsilon$ 
is not completely clear, much less
the curvature effect on the boundary layers.
As the first step on such an issue,
we concentrate ourselves to the domain~$\Omega$
being a ball with the simplest geometry, which is also of practical significance in the field of electrochemistry~\cite{NS,ST}.
This work fully addresses some prospective issues as follows:
\begin{itemize}
\item[\textbf{(i)}]\,\textbf{Boundary concentration phenomenon.}
 Continuing the work~\cite{LHLL1,LHLL2}
related to the 1D solutions (i.e., $N=1$) of (\ref{eq1}) with non-electroneutral constraint~(\ref{aneqb0902}), 
we shall refine our estimates for high-dimensional solutions (i.e., $N\geq2$) and illustrate boundary concentration phenomena with curvature effects for $\epsilon^2|\nabla{u}_\epsilon|^2$ (associated with the electric energy) 
and the net charged density~$\rho_{\epsilon}$.
 Such phenomena occur mainly due to the concentration difference~$|A-B|$, and can be described via Dirac delta functions concentrated 
at boundary points.

\item[\textbf{(ii)}]\,\textbf{Structure of the thin boundary layer.}
In order to investigate the thin boundary layer structure, we establish pointwise descriptions for thin boundary layers and the influences of surface dielectric constant,
the charged surface curvature and the concentration difference
on the boundary layer.
Mathematically, for $x\in\Omega$ sufficiently close to the boundary,
we establish the exact first two order terms for the asymptotic expansion of ${u}_\epsilon(x)$ with respect to $\epsilon$,
and analyze the curvature effects on ${u}_\epsilon(x)$.
As $\epsilon$ approaches zero,
the leading order term of ${u}_\epsilon(x)$ asymptotically blows up,
the second order term of ${u}_\epsilon(x)$ keeps bounded,
and the other terms of ${u}_\epsilon(x)$ tends to zero.
Moreover,
the influences of the boundary curvature and the concentration difference on ${u}_\epsilon(x_\epsilon)$ 
are strongly in the region where $x_\epsilon$
satisfies $\displaystyle\lim_{\epsilon\downarrow0}\epsilon^{-2}\mathrm{dist}(x_\epsilon,\partial\Omega)<\infty$.
In this region, such effects exactly appear in the second order term of ${u}_\epsilon(x_\epsilon)$ as $0<\epsilon\ll1$.
However, when $x_\epsilon$ is located
outside of this region, such effects become quite weak
because the curvature and the concentration difference appear in the other order terms tending to zero as
$\epsilon$ approaches zero.
As a consequence, our result supports that the significant
 curvature effect merely occurs in the part of the EDL sufficiently close to the charged surface,
but not in the whole EDL.

\item[\textbf{(iii)}]\,\textbf{Connection to the physical application.}
As an application, we shall explore the high-capacity of the EDL
via a theoretical way (rigorous mathematical analysis).
We use pointwise estimates of the boundary layer to provide formulas
for calculating the net charge density
and the capacitance of the EDL in quite thin regions. 
We also compare it with the capacitance of the
Helmholtz double layer for cylindrical electrode of radius $R$~(cf. \cite{HN,WP}).
\end{itemize}
Our main results for these issues will be contained in Theorem~\ref{thm1} (Section~\ref{sec1-3da}), Theorem~\ref{thm3} 
(Section~\ref{sec-bcp}),
Theorems~\ref{thm2}--\ref{thm3-new1106} (Section~\ref{compl-proof}) and Theorem~\ref{cor-20170206} (Section~\ref{sec-phys-app}); see also, Section~\ref{sec1-3da} for
summary of main results.

We stress that the CCPB equation~(\ref{eq1})--(\ref{bd1}), regarding the ion as a point-like particle, is an electrostatic model without the ion size effect. Our results reveal that its solution (the electrostatic potential) asymptotically blows up near the boundary. Moreover, we obtain high concentrations having behavior as \textit{Dirac delta functions} and a quite steep boundary layer with the slope of the order $\epsilon^{-2}$ in a thin region attaching to the boundary (see Theorems~\ref{thm3} and \ref{thm2}(II-c) for the details). Such a blow-up behavior more or less points out the importance of ion size effects. Indeed, ion size plays critical roles in the Stern layer formulation which is an extremely important improvement of layer structures in the Gouy--Chapman model \cite{s1924} and agrees well with experimental results. In recent years, there are lots of topics related to the ion size effects in electrolytes; see, e.g., \cite{BAO,Ga2018,GaPr,HEL,HKL,LE,ZWL} and references therein. Theoretically, the finite size effects on the structure of the EDL can be understood via the pointwise asymptotic behavior of the boundary layer solutions of these models. A natural question arises, which consists in asking what would be the main difference between boundary layer solutions of these models. This is also our ongoing project.

We use the following notations in the sequel.
\begin{itemize}
\item\,\,$o_{\epsilon}(1)$ always denotes a small quantity tending towards
zero as $\epsilon$ approaches zero. 
\item\,\,$O(1)$ always denotes a ``non-zero" bounded quantity independent of $\epsilon$.
\item\,\,For those estimates presented in this paper, we will frequently
abbreviate $\leq{C}$ to $\lesssim$, where $C>0$ is a generic constant independent of $\epsilon$.
\end{itemize}

\subsection{Some physical quantities}\label{sec1-1dpb}
The CCPB equation~(\ref{eq1})
is a dimensionless model of the Poisson equation
\begin{align}
-\nabla\cdot\left(\epsilon_D\nabla\phi\right)=\rho(\phi)\label{1127-nimod}
\end{align}
with a position-dependence dielectric coefficient $\epsilon_D$~(cf. \cite{L2}) and the net charge density~$\rho(\phi)$ represented as
\begin{align}
\rho(\phi)=-\Bigg(\underbrace{\frac{Ae_0}{\xint\,e^{\frac{pe_0}{k_BT}\phi}\dy}}_{\displaystyle\mathop{}^{\mathrm{reference\,\, concentration}}_{\mathrm{of\,\,the\,\,anion\,\,species}}}{\!\!\!\!\!\!}e^{\frac{pe_0}{k_BT}\phi}-
\underbrace{\frac{Be_0}{\xint\,e^{-\frac{qe_0}{k_BT}\phi}\dy}}_{\displaystyle\mathop{}^{\mathrm{reference\,\, concentration}}_{\mathrm{of\,\,the}\,\,\mathrm{cation\,\,species}}}{\!\!\!\!\!\!}e^{-\frac{qe_0}{k_BT}\phi}\Bigg),\label{ancat-0923-pipipi}
\end{align}
where $k_B$ is the Boltzmann constant, $T$ is the absolute temperature
and $e_0$ is the elementary charge.
On the other hand, (\ref{1127-nimod})--(\ref{ancat-0923-pipipi}) is 
also called the ion-conserving Poisson--Boltzmann (IC-PB) equation~\cite{Sh1}. 
Each term in (\ref{ancat-0923-pipipi})
corresponds to the Boltzmann distribution of each ion species,
where the reference concentration is defined as 
the bulk concentration at the zero potential~\cite{Sh1}. 
Besides,
 $A$ and $B$ are total number concentrations of anion and cation species in the electrolyte solution, respectively.
Hence, the constraint~(\ref{aneqb0902}) means non-electroneutrality,
as was mentioned previously.

To explain $0<\epsilon\ll1$ which appears in (\ref{eq1}),
we apply the standard dimensionless analysis (cf. Section 2.2 of \cite{L3}) 
to (\ref{1127-nimod})--(\ref{ancat-0923-pipipi}). We use
the symbol $[P]$ as the basic dimension of the physical quantity $P$, and
$P_1\sim{P_2}$ means that $P_1$ and $P_2$ have the same physical dimension
and $\frac{P_1}{P_2}$ and $\frac{P_2}{P_1}$ both are bounded.
According to the Debye--H\"{u}ckel approximation~\cite{DH1,DH2},
the Debye screening length~$\lambda_D$ measuring the thin EDL satisfies
\begin{align}\label{debye-d}
\lambda_D\sim\sqrt{\frac{k_BT\epsilon_D}{\sum_{j}q_j^2e_0^2c_j}},
\end{align}
where $q_je_0$ and $c_j$ correspond to the charge valence and the \textit{number concentration} of the $j$th ion species, respectively. For the sake of carefulness, 
we double check (via (\ref{debye-d})) that $[\lambda_D]=\left[\frac{\sqrt{k_BT\epsilon_D}}{e_0}\right]=L$ is the physical dimension of the length since $q_j$ and $c_j$ are dimensionless.
Due to the microscopic phenomenon of the thin EDL that
is extremely small compared with the diameter $\mathrm{diam}(\Omega)$ of the physical domain $\Omega$,
 we may assume 
$\lambda_D\sim\epsilon\mathrm{diam}(\Omega)$ and $0<\epsilon\ll1$.
As a consequence, by (\ref{debye-d}) we have
\begin{align}\label{dielec-d}
\epsilon_D\sim\frac{e_0^2\lambda_D^2}{k_BT}
\sim\frac{e_0^2}{k_BT}g(x)\left(\epsilon\mathrm{diam}(\Omega)\right)^2, 
\end{align}
where $g$ is a dimensionless variable associated with the dielectric coefficient~$\epsilon_D$.
It is worth mentioning that
the physical dimension of average integrals in (\ref{ancat-0923-pipipi}) are quite important
because $|\Omega|$ has physical dimension $L^{N}$, and number concentrations $A$ and $B$
have physical dimensions
\begin{align}\label{ab0412}
[A]=[B]=L^{-N}.
\end{align}
Accordingly, following the same argument as in Section~2.2 of \cite{L3} gets
\begin{align}\label{ab0412-j}
[\nabla\cdot\left(\epsilon_D\nabla\phi\right)]=&\left[{\left(Ae_0\right)e^{\frac{pe_0}{k_BT}\phi}}{\left(\Xint\,e^{\frac{pe_0}{k_BT}\phi}\dy\right)^{-1}}\right]\notag\\[-0.7em]
&\\[-0.7em]
=&\left[{\left(Be_0\right)e^{-\frac{qe_0}{k_BT}\phi}}{\left(\Xint\,e^{-\frac{qe_0}{k_BT}\phi}\dy\right)^{-1}}\right]=[e_0]L^{-N},\notag
\end{align} 
i.e., all $\nabla\cdot\left(\epsilon_D\nabla\phi\right)$,
${\left(Ae_0\right)e^{\frac{pe_0}{k_BT}\phi}}{\left(\Xint\,e^{\frac{pe_0}{k_BT}\phi}\dy\right)^{-1}}$
and ${\left(Be_0\right)e^{-\frac{qe_0}{k_BT}\phi}}{\left(\Xint\,e^{-\frac{qe_0}{k_BT}\phi}\dy\right)^{-1}}$
have the same physical dimensions. Note also that
by (\ref{ancat-0923-pipipi}),
$\left[\int_{\Omega}\rho(\phi)\dx\right]=[|\Omega|(A-B)e_0]=[e_0]$
due to the fact that $|\Omega|(A-B)$ is dimensionless.
Consequently, setting $u_{\epsilon}=\frac{e_0}{k_BT}\phi$ (which is a dimensionless variable) and using (\ref{dielec-d})--(\ref{ab0412-j}) under a suitable length scale associated with $\mathrm{diam}(\Omega)$, 
we may transform~(\ref{1127-nimod})--(\ref{ancat-0923-pipipi})
into (\ref{eq1}).

For the charged surface~$\partial\Omega$,
the setting of boundary conditions is important in a wide range of applications and provides
various effects on the electrostatic potential and the net charge density in the bulk of electrolyte solutions.
The Neumann boundary condition~(\ref{bd1}) is considered for a given surface charge distribution~$\mathrm{C}_{\mathrm{bd}}$
at the charged surface;
the Robin boundary condition~(\ref{rb-bdy}) is related to the capacitance effect of the EDL,
where $\eta_{\epsilon}$ is related to the thickness of the Stern layer~\cite{BCB}.
We refer the reader to \cite{L1,R,RWL,Sh1,WXLS} for the more details of the physical background information of the model~(\ref{1127-nimod})--(\ref{ancat-0923-pipipi})
and boundary conditions~(\ref{bd1}) and~(\ref{rb-bdy}).

Traditionally, for the EDL being formed as a parallel striped or an annular structures, such as
the planar EDL, the cylindrical EDL and the spherical EDL, the EDL capacitance can be predicted via PB type models for
 dividing the total charge in the EDL
by the respective potential difference; see, e.g., \cite{WP} for the definition and the derivation.
To the best of the author's knowledge,
for the general domain~$\Omega$
the EDL capacitance seems difficult to calculate
because the concept of ``respective potential difference"
in the region of the EDL
is ambiguously. For the sake of generalization,
for a physical region $K_0$ is contained in $\overline{\Omega}$,
we define a quantity~$\mathscr{C}^+(\phi;K_0)$ via the model~(\ref{1127-nimod})--(\ref{ancat-0923-pipipi}) in $K_0$
 as 
\begin{align}\label{capa-formula}
\mathscr{C}^+(\phi;K_0)=\frac{\left|\int_{K_0}\rho(\phi(z))dz\right|}{\displaystyle\max_{x,y\in\overline{K_0}}|\phi(x)-\phi(y)|}.
\end{align}
We stress that
for one-dimensional and radially symmetric cases, 
formula~(\ref{capa-formula})  (at most preceded by a minus sign) is equivalent to the capacitance of the planar EDL,
the cylindrical EDL and the spherical EDL
provided that $\phi$ is monotonic near the boundary.
On the other hand, an analogous measurement called the differential capacitance
is defined as the derivative of the surface charge with respect to the electric surface potential,
which describes the rate of the change of the surface charge to
 that of the electric surface potential.
In the general situation, these two approaches
are directly proportional to the voltage, making the similar concept on the EDL capacitance~\cite{WP}.

\begin{figure}[htp]
\centering{%
\begin{tabular}{@{\hspace{-0pc}}c@{\hspace{-0pc}}c}
 \psfig{figure=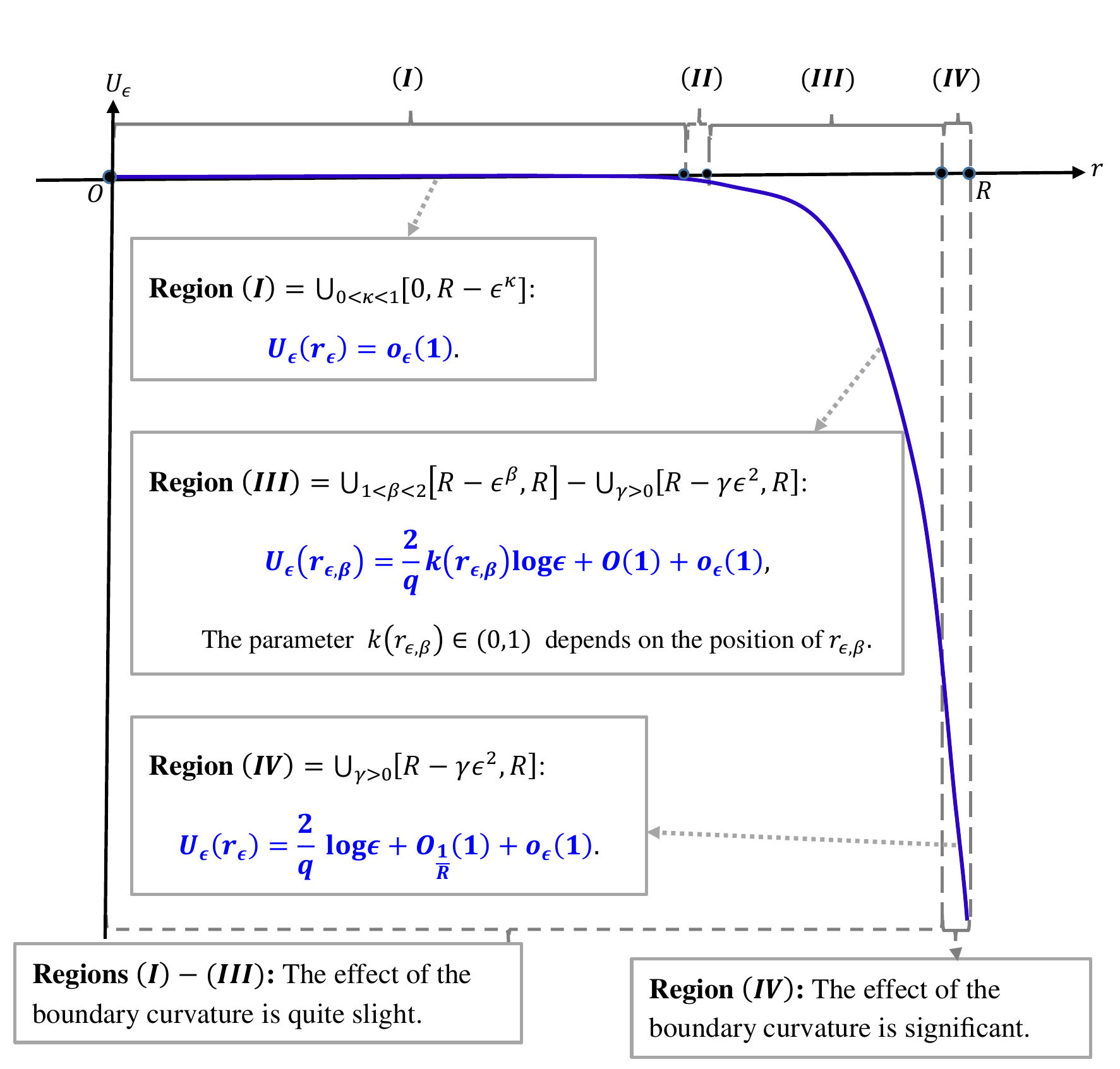, width=13cm}
  \end{tabular}
	}
	\caption{\small A schematic representation for the limiting behavior (as $\epsilon\downarrow0$) of ${u_{\epsilon}(x)}\equiv{U}_{\epsilon}(r)$
for  $r=|x|\in[0,R]$ in the case of $0<A<B$ (cf. Theorems~\ref{thm2} and \ref{thm3-new1106}).
Note that region~$(I)\subset\{r_{\epsilon}\in[0,R]:\lim_{\epsilon\downarrow0}\frac{R-r_{\epsilon}}{\epsilon}=\infty\}$,
region~$(III)\subset\{r_{\epsilon}\in[0,R]:\lim_{\epsilon\downarrow0}\frac{R-r_{\epsilon}}{\epsilon^2}=\infty\,\,\mathrm{and}\,\,\lim_{\epsilon\downarrow0}\frac{R-r_{\epsilon}}{\epsilon}=0\}$, and
region~$(IV)\subset\{r_{\epsilon}\in[0,R]:\limsup_{\epsilon\downarrow0}\frac{R-r_{\epsilon}}{\epsilon^2}<\infty\}$.
As $\epsilon\downarrow0$, we have $U_{\epsilon}\to0$ in region~$(I)$, and
$U_{\epsilon}$ asymptotically blows up in regions~$(III)$ and $(IV)$. 
Moreover, the curvature~$\frac{1}{R}$ exactly appears in the second order term~$O_{\frac{1}{R}}(1)$ of
the asymptotic expansion of $U_{\epsilon}(r_{\epsilon})$
for $r_{\epsilon}$ lying in region $(IV)$. However, outside of this region, $\frac{1}{R}$ appears in the 
term tending to zero as $\epsilon\downarrow0$.
We also establish complicated asymptotic expansions for $U_{\epsilon}$ at some points in region~$(II)$; see Theorem~\ref{thm3-new1106}.\label{fig-ccpb}}
\end{figure}

\textbf{Organization of the paper.}
The remainder of this paper goes as follows.
In Section~\ref{sec1-3da}, we summarize the main contributions for asymptotics of solutions $u_\epsilon$ in an arbitrary bounded domain (i.e., the problems~(N) and (R)) and the radially symmetric solution in a ball.
We shall stress that in order to investigate the asymptotic behavior of $u_\epsilon$, we require a series of technical estimates 
for those non-local coefficients as $0<\epsilon\ll1$.
In Section~\ref{04222017sec1}, for problems~(N) and (R)
 we establish the inverse H\"{o}lder type estimate (cf. Lemma~\ref{lem1}).
In addition, under a constraint~$A\neq{B}$ and $(A-B)(p-q)\geq0$, we establish a lower bound for the boundary blow-up behavior of $u_\epsilon$ as $\epsilon$ approaches zero (cf. Theorem~\ref{thm1}).
In Section~\ref{04222017sec2}, for radially symmetric solutions 
we summarize the main results containing boundary concentration phenomenon~(cf. Theorem~\ref{thm3} in Section~\ref{sec-bcp}), 
the exact pointwise description and the curvature effect for the thin boundary layer~(cf. Theorems~\ref{thm2} and \ref{thm3-new1106}
in Section~\ref{compl-proof}).
We also provide theoretical analysis and applications 
for the EDL capacitance~(cf. Theorem~\ref{cor-20170206} in Section~\ref{sec-phys-app}).
It is worth mentioning that
we use different approaches on studying asymptotics of radially symmetric solutions
 so that we do not require the constraint $(A-B)(p-q)\geq0$.
In Section~\ref{0422-diff2017sec}, we state main difficulty and analysis techniques for radially symmetric solutions.

All rigorous proofs are given in Sections~\ref{sec0416addsec}--\ref{sec-phys-app}.
In Section~\ref{gentherm-sec}, we prove Lemma~\ref{lem1} and Theorem~\ref{thm1}.
As the preliminaries of the main theorems stated in Sections~\ref{sec-bcp}--\ref{sec-phys-app},
 we first establish basic estimates including the Poho\v{z}aev's identities and gradient estimates 
for the radially symmetric solutions in Section~\ref{prelimnstar}.
Using these arguments, we also establish the exact first two order terms with curvature effects
for boundary asymptotic expansions. After the preparation,
we give the proof of Theorem~\ref{thm3} (cf. Section~\ref{sec-bcp}). 
We state the proof of  Theorems~\ref{thm2} and \ref{thm3-new1106} in Sections~\ref{sec6thm2} and \ref{sec6thm3-new1106},
respectively.
Finally, the proof of Theorem~\ref{cor-20170206} is given in Section~\ref{sec-phys-app}.



\section{Overview of the main results and applications}\label{sec1-3da}
In this section, we provide a guideline for readers in gaining clear pictures of the present work,
and state our main ideas for the rigorous proofs.

\subsection{Asymptotic blow-up analysis in arbitrary bounded domains}\label{04222017sec1}
Let $\Omega$ be an arbitrary bounded domain. As mentioned previously, problems~(N) and~(R) have unique solutions~$u_\epsilon\in{C^\infty(\Omega)\cap{C^1(\overline{\Omega})}}$. For the asymptotic behavior of $u_\epsilon$,
we establish the \textit{inverse H\"{o}lder type estimate} for those non-local coefficients
and some estimates of the net charge density~(\ref{ancat-0923-pipipi}) 
as follows:

\begin{lemma}\label{lem1}
Assume $A$, $B$, $p$ and $q$
are positive constants independent of $\epsilon$,
and $\Omega$ is a bounded smooth domain in $\mathbb{R}^N$ ($N\geq2$).
Let $u_\epsilon\in{C^\infty(\Omega)\cap{C^1(\overline{\Omega})}}$ be the unique solution of problem~(N)
for $\epsilon>0$. Then
\begin{itemize}
\item[(I)] 
\begin{itemize}
\item[(i)] If $0<A<B$, then 
 \begin{align}\label{id-lem-0}
A-B\leq\max_{\overline{\Omega}}\left(\frac{Ae^{pu_\epsilon}}{\xint\,e^{pu_\epsilon(y)}\dy}-
\frac{Be^{-qu_\epsilon}}{\xint\,e^{-qu_\epsilon(y)}\dy}\right)\leq0.
\end{align}
\item[(ii)] If $0<B<A$, then 
 \begin{align}\label{id-lem-0-add0405}
0\leq\min_{\overline{\Omega}}\left(\frac{Ae^{pu_\epsilon}}{\xint\,e^{pu_\epsilon(y)}\dy}-
\frac{Be^{-qu_\epsilon}}{\xint\,e^{-qu_\epsilon(y)}\dy}\right)\leq{A-B}.
\end{align}
\item[(iii)] (Inverse H\"{o}lder type estimate for $e^{u_{\epsilon}}$ and $e^{-u_{\epsilon}}$)
The following estimate holds:
\begin{align}\label{id-lem-2}
\big\|e^{u_{\epsilon}}{\big\|}_{L^{p}(\Omega)}
{\big\|}e^{-u_{\epsilon}}{\big\|}_{L^{q}(\Omega)}
\leq|\Omega|^{\frac{1}{p}+\frac{1}{q}}\max\left\{\left(\frac{B}{A}\right)^{\frac{1}{q}},\left(\frac{A}{B}\right)^{\frac{1}{p}}\right\}.
\end{align}
\end{itemize}
\item[(II)] In addition, we assume
\begin{align}\label{aneqb}
(A-B)(p-q)\geq0\quad{and}\quad{A}\neq{B}.
\end{align}
Then
\begin{itemize}
\item[(i)] If $0<A<B$ and $0<p\leq{q}$, then 
\begin{align}\label{id-lem-1}
\min_{\overline{\Omega}}\left(\frac{Ae^{pu_\epsilon}}{\xint\,e^{pu_\epsilon(y)}\dy}-
\frac{Be^{-qu_\epsilon}}{\xint\,e^{-qu_\epsilon(y)}\dy}\right)\leq\frac{A-B}{\alpha\epsilon}.
\end{align}
\item[(ii)] If $0<B<A$ and $0<q\leq{p}$, then 
\begin{align}\label{id-lem-1-pulun}
\max_{\overline{\Omega}}\left(\frac{Ae^{pu_\epsilon}}{\xint\,e^{pu_\epsilon(y)}\dy}-
\frac{Be^{-qu_\epsilon}}{\xint\,e^{-qu_\epsilon(y)}\dy}\right)\geq\frac{A-B}{\alpha\epsilon}.
\end{align}
Here $\alpha$ is a positive constant independent of $\epsilon$. 
\end{itemize}
\end{itemize}
\end{lemma}

We stress that the inverse H\"{o}lder type estimate~(\ref{id-lem-2}) does not require the constraint~(\ref{aneqb}).
Such an estimate is novel and never appears in \cite{L2,LHLL1,LHLL2}.

Now we assume that (\ref{aneqb}) holds.
Using Lemma~\ref{lem1} we may generalize the method used in~\cite{LHLL1,LHLL2} and follow the comparison argument established in~Theorem~1.1 of \cite{L2} to establish the following:

\begin{theorem}\label{thm1}
 Under the same hypotheses as in Lemma~\ref{lem1}, in addition, we 
assume that (\ref{aneqb}) holds.
Then the maximum difference satisfying
$\displaystyle\max_{x,y\in\overline{\Omega}}\left|u_\epsilon(x)-u_\epsilon(y)\right|
\geq\frac{1}{\max\{p,q\}}\log\frac{1}{\epsilon}+O(1)$ asymptotically blows up as $\epsilon\downarrow0$. 
Moreover, we have
\begin{itemize}
\item[(I)] If $0<A<B$ and $0<p\leq{q}$, then $u_\epsilon$
attains its maximum value at an interior point and minimum value at a boundary point, and 
\begin{align}
 0<&\max_{\overline{\Omega}}u_\epsilon\leq\frac{1}{q}\log\frac{B}{A},\quad\forall\,\,\epsilon>0,\label{thm1-id1}\\
\min_{\partial{\Omega}}u_\epsilon\leq&-\frac{1}{q}\log\frac{1}{\epsilon}+O(1),\quad
{for}\quad0<\epsilon\ll1.\label{thm1-id2}
\end{align}
\item[(II)] If $A>B>0$ and $p\geq{q}>0$, then $u_\epsilon$
attains its minimum value at an interior point and maximum value at a boundary point, and 
\begin{align}
 -\frac{1}{p}\log\frac{A}{B}\leq\min_{\overline{\Omega}}u_\epsilon<0&,\quad\forall\,\,\epsilon>0,\label{thm1-id3}\\
\max_{\partial{\Omega}}u_\epsilon\geq\frac{1}{p}\log\frac{1}{\epsilon}+O(1),&\quad
{for}\quad0<\epsilon\ll1.\label{thm1-id4}
\end{align}
\item[(III)] For any compact subset $K$ of $\Omega$, 
there exists positive constants $C_K$ and $M_K$ independent of $\epsilon$
such that
\begin{align}\label{thm1-id5}
\frac{\displaystyle\max_{x,y\in{K}}|u_\epsilon(x)-u_\epsilon(y)|}{\displaystyle\max_{x,y\in{\overline{\Omega}}}|u_\epsilon(x)-u_\epsilon(y)|}\leq\,C_Ke^{-\frac{M_K}{\epsilon}},
\end{align}
which exponentially decays to zero as $\epsilon$ goes to zero.
\end{itemize}
\end{theorem}

The proof of Lemma~\ref{lem1} and Theorem~\ref{thm1} are given in Section~\ref{gentherm-sec}.

Theorem~\ref{thm1} shows that if (\ref{aneqb}) holds, then
as $\epsilon$ approaches zero $u_\epsilon$ asymptotically blows up at a boundary point, and
(\ref{thm1-id5}) shows that the potential difference in the bulk
is quite small compared with that over the whole domain~$\overline{\Omega}$.
In Section~\ref{04222017sec2},
we will see that for the situation $\Omega$ a ball and $u_\epsilon$ radially symmetric,
 $\displaystyle\max_{x,y\in{\overline{\Omega}}}|u_\epsilon(x)-u_\epsilon(y)|e^{-\frac{M_K}{\epsilon}}$
also exponentially decays to zero as $\epsilon$ goes to zero.

\subsection{Fine asymptotics with the curvature effect in a ball}\label{04222017sec2}
We consider $\Omega=\mathbb{B}_R\equiv\{x\in\mathbb{R}^N:\,|x|<R\}$ for $R>0$
with the simplest geometry in order to describe precisely the curvature effect on the boundary layer.  
Due to the uniqueness, $u_\epsilon(x)=U_\epsilon(|x|)$ is radially symmetric, and 
$U_\epsilon$ satisfies
\begin{align}
\epsilon^2g(r)\Bigg[U_{\epsilon}''(r)+&\left(\frac{N-1}{r}
+\frac{g'(r)}{g(r)}\right)U_{\epsilon}'(r)\Bigg]\notag\\[-0.2em]
&\label{bigu1}\\[-0.7em]
=&\frac{R^N}{N}\left(\frac{Ae^{pU_{\epsilon}(r)}}{\int_0^{R}s^{N-1}e^{pU_{\epsilon}(s)}\dss}-\frac{Be^{-qU_{\epsilon}(r)}}{\int_0^{R}s^{N-1}e^{-qU_{\epsilon}(s)}\dss}\right),\quad r\in(0,R)\notag
\end{align}
(for convenience, here we set the surface area of the unit $N$-dimensional sphere is equal to $1$),
and
\begin{align}\label{chdensity-1204}
\rho_{\epsilon}(r)=-\frac{R^N}{N}\left(\frac{Ae^{pU_{\epsilon}(r)}}{\int_0^{R}s^{N-1}e^{pU_{\epsilon}(s)}\dss}-\frac{Be^{-qU_{\epsilon}(r)}}{\int_0^{R}s^{N-1}e^{-qU_{\epsilon}(s)}\dss}\right),
\end{align}
where $\rho_{\epsilon}$ is the net charge density defined in (\ref{ancat-0923-pipipi}).

Similar to problems~(N) and (R), we shall study the following problems for $U_{\epsilon}$:
\begin{itemize}
\item[\textbf{(N*).}] \textbf{Neumann problem for} $\boldsymbol{U_\epsilon}$:
The equation~(\ref{bigu1}) with the constraint
\begin{align}
\int_{0}^Rs^{N-1}U_{\epsilon}(s)\dss=0,\quad\quad&\label{bigu2}
\end{align}
and the boundary condition
\begin{align}
U_{\epsilon}'(0)=0,\quad\quad\mathrm{and}\quad\quad{U}_{\epsilon}'(R)=&\frac{R(A-B)}{\epsilon^2Ng(R)}.\label{bigu3}
\end{align}
\item[ \textbf{(R*).}] \textbf{Robin problem for} $\boldsymbol{U_\epsilon}$: The equation~(\ref{bigu1}) with the boundary condition
\begin{align}
U_{\epsilon}'(0)=0,\quad\quad\mathrm{and}\quad\quad{U}_{\epsilon}(R)+\eta_{\epsilon}{U}_{\epsilon}'(R)=0.\label{bigurandb}
\end{align}
\end{itemize}

Multiplying (\ref{bigu1}) by $r^{N-1}$, integrating the result over $(0,R)$,
and using (\ref{bigurandb}), we find that the boundary condition~(\ref{bigurandb}) at $r=R$ 
can be decoupled into 
\begin{align}\label{haha1211-2016}
{U}_{\epsilon}'(R)=\frac{R(A-B)}{\epsilon^2Ng(R)}\quad\mathrm{and}\quad{U}_{\epsilon}(R)=\frac{\eta_\epsilon{R}(B-A)}{\epsilon^2Ng(R)}.
\end{align}
Note that the first one of (\ref{haha1211-2016}) has the same form as (\ref{bigu3}).
Let $\boldsymbol{U_{\mathrm{R}^*,\epsilon}}$ be the unique solution for problem (R*), 
and let $\boldsymbol{U_{\mathrm{N}^*,\epsilon}}$ be the
unique solution for 
problem (N*).
Then by (\ref{bigu2})--(\ref{bigurandb}) and the uniqueness, one immediately obtains
\begin{align}\label{urnstar}
 \boldsymbol{U_{\mathrm{R}^*,\epsilon}}-\frac{N}{R^N}\int_0^Rs^{N-1}\boldsymbol{U_{\mathrm{R}^*,\epsilon}}(s)\dss=\boldsymbol{U_{\mathrm{N}^*,\epsilon}}.
\end{align}
On the other hand, by (\ref{haha1211-2016}) and (\ref{urnstar}), we have $ \boldsymbol{U_{\mathrm{R}^*,\epsilon}}(R)=\frac{\eta_\epsilon{R}(B-A)}{\epsilon^2Ng(R)}$ and $\boldsymbol{U_{\mathrm{R}^*,\epsilon}}(R)-\boldsymbol{U_{\mathrm{R}^*,\epsilon}}(r)=\boldsymbol{U_{\mathrm{N}^*,\epsilon}}(R)-\boldsymbol{U_{\mathrm{N}^*,\epsilon}}(r)$, for $r\in[0,R)$. Consequently,
\begin{align}\label{urnstar-0208}
\boldsymbol{U_{\mathrm{R}^*,\epsilon}}(r)=
\frac{\eta_\epsilon{R}(B-A)}{\epsilon^2Ng(R)}+\boldsymbol{U_{\mathrm{N}^*,\epsilon}}(r)-\boldsymbol{U_{\mathrm{N}^*,\epsilon}}(R).
\end{align}
This shows that the asymptotic behavior of $\boldsymbol{U_{\mathrm{R}^*,\epsilon}}(r)$
can be directly obtained from $\boldsymbol{U_{\mathrm{N}^*,\epsilon}}(r)$ and (\ref{urnstar-0208}). 
Hence, we focus on the asymptotics of solutions~$U_\epsilon$ of the problem~(N*) for $\epsilon\downarrow0$.
Without loss of generality, we may assume $0<A<B$.

We stress that in problem~(N*),
we do not need the constraint~(\ref{aneqb})
because we can use different approaches such as the Poho\v{z}aev's identity (see (\ref{sec3id1})) and (pointwise) gradient estimates to get much more fine estimates for the radially symmetric solutions.
(Such approaches cannot work for general solutions in arbitrary domains.)
Note also that the inverse H\"{o}lder type estimate for the solution $u_\epsilon$ of the problem~(N)
can be directly applied to the solutions~$U_\epsilon$ of the problem~(N*).

For (N*), we show that as $\epsilon$ approaches zero,
 $U_\epsilon$ asymptotically blows up near the boundary
 and develops boundary concentration phenomenon~(cf. Theorem~\ref{thm3}). Moreover,
 we establish new appraoches to obtain the exact first two terms asymptotic expansions of $U_\epsilon$
 and analyze the curvature effect on $U_\epsilon$
near the boundary (cf. Theorems~\ref{thm2} and \ref{thm3-new1106}).
Besides,
as an application, we use those pointwise asymptotic expansions to calculate~(\ref{capa-formula})
for $\phi=U_\epsilon$,
which is related to the EDL capacitance (cf. Theorem~\ref{cor-20170206}).

The main results for the asymptotic behavior of $U_\epsilon$ require lengthy statements.
For the reader's convenience, we are in a position to make brief summaries for those results.
(The related theorems and their corresponding proofs will be contained in Sections~\ref{sec0416addsec}--\ref{sec-phys-app}.) 
Some results for problem~(N*) are stated as follows:\\ \\ 
\textbf{(A). Boundary concentration phenomena (cf. Theorem~\ref{thm3}).}\\

 To illustrate the boundary concentration phenomenon in problem~(N*),
we introduce a Dirac delta function~$\delta_R$ concentrated 
at the boundary point $r=R$ as follows.
\begin{definition}\label{def1}
It is said that functions $f_\epsilon\rightharpoonup\mathcal{C}\delta_R$
 weakly in 
$C([0,R];\mathbb{R})$ with a weight $\mathcal{C}\neq0$ as $\epsilon\downarrow0$ if there holds
\begin{align}\notag
\lim_{\epsilon\downarrow0}\int_0^Rh(r)f_\epsilon(r)\dr=\mathcal{C}h(R)
\end{align}
for any continuous function $h:[0,R]\to\mathbb{R}$ independent of $\epsilon$.
\end{definition}
 Note that $U_\epsilon$ is the electrostatic potential, and $({\epsilon}U_\epsilon')^2$
is related to the electric energy. As $\epsilon$ approaches zero,
 $({\epsilon}U_\epsilon')^2$ and
the net charge density $\rho_{\epsilon}$ 
develop boundary concentration phenomena in the following senses:
\begin{itemize}
\item[\textbf{(a1).}]\,\,As $\epsilon\downarrow0$, $\left(\epsilon\,U_\epsilon'\right)^2$
and $\left(\epsilon\,U_\epsilon'\right)^{\theta}$ for $\theta\in(0,2)$
have totally different asymptotic behaviors in $[0,R]$:
\begin{align}
\begin{cases}
\displaystyle\left(\epsilon\,U_\epsilon'\right)^2\rightharpoonup\frac{2R(B-A)}{qNg(R)}\delta_R\,\,\mathrm{weakly\,\,in}\,\,
C([0,R];\mathbb{R}),\\
\displaystyle||\epsilon\,U_\epsilon'||_{L^{\theta}((0,R))}\lesssim\epsilon^{\min\{1,\frac{2}{\theta}-1\}}\left(\log\frac{1}{\epsilon}\right)^{\chi(\theta)}\,\,\mathrm{as}\,\,0<\epsilon\ll1,
\end{cases}
\end{align} 
 where $\delta_R$ is a Dirac delta function concentrated 
at the boundary point $r=R$, $\chi(\theta)=1$ if $\theta\in(0,1]$;
$\chi(\theta)=0$ if $\theta\in(1,2)$. In other words,
$\left(\epsilon\,U_\epsilon'\right)^2$ has boundary concentration phenomenon
but $\epsilon\,U_\epsilon'\rightarrow0$
strongly in $L^\theta((0,R))$.
\item[\textbf{(a2).}]\,\,As $\epsilon\downarrow0$, $\rho_{\epsilon}$ develops a thin layer
and has the concentration phenomenon near the boundary. More precisely, there holds
\begin{equation*}
 \rho_{\epsilon}\rightharpoonup\frac{R(B-A)}{N}\delta_R\,\,\mathrm{weakly\,\,in}\,\,C([0,R];\mathbb{R}).
\end{equation*}
However, for $\widetilde{R}\in(0,R)$ independent of $\epsilon$,
\begin{equation*}
\max_{[0,\widetilde{R}]}|\rho_{\epsilon}|\longrightarrow0
\end{equation*}
exponentially as $\epsilon\downarrow0$.
\end{itemize}

Here we give a connection between (a2) with a physical phenomenon of the ion distributions
in the bulk and near the charged surface.
The net charge density~$\rho_{\epsilon}$ defined in (\ref{chdensity-1204})
can be regarded as the concentration difference at each interior point.
(a2) presents that $\rho_{\epsilon}$ possesses the charge neutrality (zero net charge)
in the bulk of $[0,R)$. Our result indicates that
the extra charges are accumulated near the boundary, and hence
 the non-electroneutral phenomenon only occurs near the boundary.
 It is worth mentioning that the boundary concentration phenomenon occurs
mainly due to the difference $|A-B|$ between the total concentrations of cation and anion species.

The boundary concentration phenomenon of $\left({\epsilon}U_{\epsilon}'\right)^2$ 
 seems a novel phenomenon, which occurs mainly due to
 the boundary condition with strongly singular perturbation parameter~$\epsilon^{-2}$.
To the best of our knowledge,
such a result is not found in other related literatures.\\ \\
\textbf{(B). Pointwise descriptions for the boundary layers 
(cf. Figure~\ref{fig-ccpb} and Theorems~\ref{thm2} and~\ref{thm3-new1106}).}

 We analyze the effects
of the boundary curvature~$\frac{1}{R}$,
the boundary dielectric constant~$g(R)$ and the concentration difference~$|A-B|$
on the structure of the boundary layer as follows:

\begin{itemize}
\item[\textbf{(b1).}] Let $\kappa\in(0,1)$. Then $\displaystyle\max_{r_{\epsilon}\in[0,R-\epsilon^\kappa]}|U_\epsilon(r_{\epsilon})|\lesssim\epsilon^{\kappa}\log\frac{1}{\epsilon}$
and $\displaystyle\max_{r_{\epsilon}\in[0,R-\epsilon^\kappa]}|r_{\epsilon}^{N-1}U_\epsilon'(r_{\epsilon})|\leq{e}^{-\frac{C}{{\epsilon}^{1-\kappa}}}$
 as $0<\epsilon\ll1$,
where $C$ is a positive constant independent of $\epsilon$.
Hence, in the region~$[0,R-\epsilon^\kappa]$, the net charge density~$\rho_{\epsilon}$ tends to electroneutrality, 
and the effects of the curvature~$\frac{1}{R}$, the boundary dielectric constant~$g(R)$,
 and the concentration difference~$|A-B|$
on the electrostatic potential~$U_\epsilon$
are quite slight.  
\end{itemize}

 On the other hand, if  $r_{\epsilon}$ is sufficiently close to the boundary such that
 $r_{\epsilon}\in\displaystyle\bigcup_{\beta>1}[R-\epsilon^{\beta},R]$ as $0<\epsilon\ll1$, then $U_\epsilon(r_{\epsilon})$
asymptotically blows up as $\epsilon$ approaches zero.
We stress that
the leading-orders of asymptotic expansions of $U_\epsilon(r_{\epsilon})$
have various blow-up rates depending on the position of $r_\epsilon$.
  A variety of  
asymptotic blow-up behavior of $U_\epsilon(r_{\epsilon})$
and pointwise estimates of $U_\epsilon'(r_{\epsilon})$ and $\rho_\epsilon(r_{\epsilon})$
are clearly stated in Theorems~\ref{thm2} and \ref{thm3-new1106}.
Moreover, we show that as $r_{\epsilon}$
locates at the region
\begin{align}
\left\{r_{\epsilon}\in[0,R):\limsup_{\epsilon\downarrow0}\frac{R-r_{\epsilon}}{\epsilon^{2}}<\infty\right\},
\end{align} 
the curvature effect on the boundary layer of $U_{\epsilon}(r_{\epsilon})$ is quite strong.
Outside of this region, such effects become quite weak; see more details in Theorems~\ref{thm2} and~\ref{thm3-new1106}.
For the sake of simplicity,
in following (b2) and (b3) we merely summarize the effects of the boundary curvature~$\frac{1}{R}$ and the concentration difference~$|A-B|$ on $U_\epsilon(r_{\epsilon,\beta})$
in the case that $r_{\epsilon,\beta}$ satisfies $\displaystyle\lim_{\epsilon\downarrow0}\frac{R-r_{\epsilon,\beta}}{\epsilon^\beta}=\gamma_\beta$ for $\beta>1$~(cf. (\ref{1105-2016-ab1})):
\begin{itemize} 
\item[\textbf{(b2).}] \textbf{The part of the boundary layer where the effects of the curvature 
and the concentration difference are quite slight.}
For $\beta\in(1,2)$ and $\gamma_\beta\in(0,\infty)$, we have
$$U_{\epsilon}(r_{\epsilon,\beta})=\underbrace{-\frac{2}{q}(\beta-1)\log\frac{1}{\epsilon}+\frac{2}{q}\left(\log\sqrt{\frac{A}{2qg(R)}}+\log(\gamma_{\beta}q)\right)}_{\mathrm{without\,\,the\,\,effect\,\,of\,\,the\,\,curvature}~R^{-1}}+o_{\epsilon}(1).$$ 
Consequently, 
the influences of the curvature $\frac{1}{R}$ and the concentration difference $|A-B|$ seem quite slight 
because these quantities appear in the terms tending to zero as $\epsilon\downarrow0$.

\item[\textbf{(b3).}] \textbf{The part of the boundary layer where the effects of the curvature and the concentration difference are significant.}
For $\beta\geq2$, the curvature $\frac{1}{R}$, the surface dielectric constant $g(R)$,
and the concentration difference $|A-B|$ exactly appear in \textit{the second order term}~$O(1)$ of
$U_{\epsilon}(r_{\epsilon,\beta})$, which is stated as follows:
\begin{itemize}
 \item[\textbf{(b3-1).}] If $\beta=2$ and $\gamma_\beta>0$, then
\begin{align}
U_{\epsilon}(r_{\epsilon,\beta})=-\frac{2}{q}\log\frac{1}{\epsilon}+\frac{2}{q}\Bigg[\log\sqrt{\frac{A}{2qg(R)}}+\log\Bigg(\gamma_{\beta}q+\underbrace{\frac{1}{R}}_{(\mathrm{e}_1)}\cdot\underbrace{\frac{2Ng(R)}{B-A}}_{(\mathrm{e}_2)}\Bigg)\Bigg]+o_{\epsilon}(1),\notag
\end{align}
where $(\mathrm{e}_1)$ presents the effect of the boundary (corresponding to the charged surface) curvature, and 
$(\mathrm{e}_2)$ is about the effect of the difference between the total concentrations of cation and anion species.
 \item[\textbf{(b3-2).}] If $\beta>2$, then 
\begin{align}
U_{\epsilon}(r_{\epsilon,\beta})
=-\frac{2}{q}\log\frac{1}{\epsilon}+\frac{2}{q}\left[\log\sqrt{\frac{A}{2qg(R)}}+\log\left(\frac{1}{R}\cdot\frac{2Ng(R)}{B-A}\right)\right]
+o_{\epsilon}(1).\notag
\end{align}
Note that the first two order terms are independent of $\gamma_\beta$.
\end{itemize}	
\end{itemize}

Similar phenomena are also obtained for the case of $0<B<A$.
On the other hand, in Theorem~\ref{thm3-new1106} we provide some novel asymptotic behaviors,
such as that of $U_{\epsilon}(R-\gamma\epsilon^\beta(\log\frac{1}{\epsilon})^l)$ and $U_{\epsilon}(R-\gamma\epsilon^{\beta(\epsilon)})$
which are not the cases of Theorem~\ref{thm2},
where $\gamma>0$ and $l\neq0$ are independent of $\epsilon$,
 and $\beta(\epsilon)>1$ satisfies $\lim_{\epsilon\downarrow0}\beta(\epsilon)=1$.
For more detailed results,
such as the influences of the curvature, the boundary dielectric constant,
and the concentration difference on the slope of the boundary layer, 
 we refer the reader to the main theorems in Section~\ref{compl-proof}.\\ \\
\textbf{(C). Applications to calculating the EDL capacitance (cf. Theorem~\ref{cor-20170206}).} 

 Due to the monotonicity of solutions~$U_{\epsilon}$ of (\ref{bigu1})--(\ref{bigu3})
(cf. Theorem~\ref{thm2}), the quantity~(\ref{capa-formula})
related to the EDL capacitance is equivalent to
\begin{align}\label{capa-formula-new}
\mathscr{C}^+(U_{\epsilon};[r_{\epsilon},R]):=\left|\frac{\int_{r_{\epsilon}}^R\rho_{\epsilon}(r)r^{N-1}\dr}{U_{\epsilon}(R)-U_{\epsilon}(r_{\epsilon})}\right|,
\end{align}
where $r_{\epsilon}\in(0,R)$ and the net charge density~$\rho_{\epsilon}$ is defined in (\ref{chdensity-1204}).
We stress that (\ref{capa-formula-new}) corresponds to the EDL capacitance in an annular region~$\{x\in\mathbb{R}^N:|x|=r\in[r_{\epsilon},R]\}$ attaching to the charged surface.

We provide a theoretical point of view to see the influences 
of the surface dielectric constant~$g(R)$, the surface curvature~$\frac{1}{R}$,
and the concentration difference~$|A-B|$ on the EDL capacitance as follows:
\begin{itemize}
\item[\textbf{(c1).}]\,\,If $\displaystyle\lim_{\epsilon\downarrow0}\frac{R-r_{\epsilon}^{\infty}}{\epsilon^2}=\infty$, then
$\mathscr{C}^+(U_{\epsilon};[r_{\epsilon}^{\infty},R])=o_{\epsilon}(1)$ as $\epsilon\downarrow0$.
\item[\textbf{(c2).}]\,\,If $\displaystyle\lim_{\epsilon\downarrow0}\frac{R-r_{\epsilon}^{\gamma}}{\epsilon^2}=\gamma\in(0,\infty)$,
then 
\begin{align*}
\mathscr{C}^+(U_{\epsilon};[r_{\epsilon}^{\gamma},R])=
\frac{C^bq}{2N\left(1+\frac{2NR^{N-1}g(R)}{C^b\gamma{q}}\right)\log\left(1+\frac{C^b
\gamma{q}}{2NR^{N-1}g(R)}\right)}+o_{\epsilon}(1),
\end{align*} 
where $C^b=R^N|A-B|$ represents
 the extra charge near the charged surface. 
(Note that $\left|[r_{\epsilon}^{\infty},R]\right|=R-r_{\epsilon}^{\infty}\gg{R}-r_{\epsilon}^{\gamma}=\left|[r_{\epsilon}^{\gamma},R]\right|$.)
For a physical case that the region $[r_{\epsilon}^{\gamma},R]$ lying in the Stern layer is quite thinner, i.e., $0<\gamma\ll1$, we further lead a formal approximation 
\begin{align}\notag
\frac{1}{\mathscr{C}^+(U_{\epsilon};[R-\gamma{\epsilon}^{2},R])}\approx\frac{1}{\mathscr{C}_1}+\frac{1}{\mathscr{C}_2},
\end{align}
where $\mathscr{C}_1=\frac{R^{N-1}g(R)}{\gamma}=\frac{R^{N-1}{\epsilon}^{2}g(R)}{\gamma{\epsilon}^{2}}$ and 
$\mathscr{C}_2=\frac{C^bq}{2N}$ (cf. Remark~\ref{rk0207}).
In particular, $\mathscr{C}_1=\frac{R^{N-1}g(R)}{\gamma}=\frac{R^{N-1}g(R){\epsilon}^{2}}{\gamma{\epsilon}^{2}}$
can be regarded as the capacitance of a capacitor constructed of two parallel plates 
separated by a distance $\gamma{\epsilon}^{2}$ (the thickness of $[R-\gamma{\epsilon}^{2},R]$).
Moreover, $\mathscr{C}^+(U_{\epsilon};[r_{\epsilon}^{\gamma},R])$ has the same analogous measurement as the 
specific capacitance of the Helmholtz double layer for the cylindrical electrode of radius~$R$ (cf. \cite{WP}). 
\end{itemize}

An interesting outcome shows that $\mathscr{C}^+(U_{\epsilon};[r_{\epsilon}^{\gamma},R])$ is getting higher
as the thickness~($\sim\gamma\epsilon^2$) of the region $[r_{\epsilon}^{\gamma},R]$
becomes thinner (i.e., $\gamma$ becomes smaller). This provides a theoretical way to support that
the EDL has higher capacitance in a quite thin region near the charged surface, not in the whole EDL. 
Those results are stated in Theorem~\ref{cor-20170206}.

\subsection{Main difficulty and analysis techniques}\label{0422-diff2017sec}
As pointed out previously, the major challenges for analysis of the problem~(N*)
come from its non-local dependence on the unknown variable~$U_{\epsilon}$ and the Neumann boundary condition~(\ref{bigu3}) having strongly singular perturbation parameter~$\frac{1}{\epsilon^2}$.
Indeed, the exact asymptotics of those non-local coefficients (as $\epsilon\downarrow0$)
seem difficult to be captured via an intuitive way before we have any estimate for $U_{\epsilon}$,
which may increase an extra difficulty for investigating the asymptotic behavior of solutions to the problem~(N*). 
On the other hand, because of the shift-invariance of (\ref{bigu1}), it is not easy to guess the exact leading order term of
the boundary value $U_{\epsilon}(R)$ from (\ref{bigu2}) and (\ref{bigu3}).
Based upon the analysis techniques in \cite{L2,LHLL1,LHLL2,LR,RLW},
we derive the corresponding Poho\v{z}aev's identities~(\ref{sec3id1})
which presents a relationship between those non-local coefficients and the boundary value~$U_{\epsilon}(R)$.
Moreover, using the inverse H\"{o}lder type estimate (cf. Lemma~\ref{lem1}(I-iii)), the standard Jensen's inequality 
and the constraint~(\ref{bigu2}) asserts that both $\int_0^Rs^{N-1}e^{pU_{\epsilon}(s)}\dss$
and $\int_0^Rs^{N-1}e^{-qU_{\epsilon}(s)}\dss$ have positive lower and upper bounds independent of $\epsilon$
(cf.  Lemma~\ref{lem-aug22-1}).
Combining this estimate with the Poho\v{z}aev's identities and some comparison principles, we reach the exact leading order terms of those non-local coefficients~(cf. Lemma~\ref{lem-0910-new})
\begin{align}
\int_0^{R}s^{N-1}e^{pU_{\epsilon}(s)}\dss=\frac{R^N}{N}+o_{\epsilon}(1),\quad\,\,\int_0^{R}s^{N-1}e^{-qU_{\epsilon}(s)}\dss=\frac{BR^N}{AN}+o_{\epsilon}(1)\,\,(0<A<B),\label{0818-nig1}\\
\int_0^{R}s^{N-1}e^{pU_{\epsilon}(s)}\dss=\frac{AR^N}{BN}+o_{\epsilon}(1),\quad\int_0^{R}s^{N-1}e^{-qU_{\epsilon}(s)}\dss=\frac{R^N}{N}+o_{\epsilon}(1)\,\,(0<B<A),\label{0818-nig2}
\end{align}
  for any $p,q>0$,
and establish the exact first two order terms of $U_{\epsilon}(R)$
with respect to $\epsilon$ (cf. (\ref{0916-thm2-1})). 
We stress that in \cite{L2,LHLL1,LHLL2,RLW},
 the authors can merely deal with the non-neutral CCPB equation under the case $p=q$
because the analysis techniques is not enough to establish the inverse H\"{o}lder type estimate for $p\neq{q}$. 
We state these arguments in Section~\ref{prelimnstar}.

We emphasize that for a general case without the constraint~(\ref{bigu2}),
$\int_0^Rs^{N-1}{U_{\epsilon}(s)}\dss$ indeed affects asymptotics of all non-local coefficients for small $\epsilon>0$. Their exact leading order terms are depicted as follows:
\begin{align}
\begin{cases}
\displaystyle\int_0^{R}s^{N-1}e^{pU_{\epsilon}(s)}\dss&=\left(\frac{\displaystyle{R^N}}{\displaystyle{N}}+o_{\epsilon}(1)\right)\exp\left(p\int_0^Rs^{N-1}{U_{\epsilon}(s)}\dss\right),\vspace{5pt}\\
\displaystyle\int_0^{R}s^{N-1}e^{-qU_{\epsilon}(s)}\dss&=\frac{\displaystyle{B}}{\displaystyle{A}}\left(\frac{\displaystyle{R^N}}{\displaystyle{N}}+o_{\epsilon}(1)\right)\exp\left(-q\int_0^Rs^{N-1}{U_{\epsilon}(s)}\dss\right)
\end{cases}\,\,(0<A<B),\label{0818-nig1-ad}
\end{align}
and
\begin{align}
\begin{cases}
\displaystyle\int_0^{R}s^{N-1}e^{pU_{\epsilon}(s)}\dss&=\frac{\displaystyle{A}}{\displaystyle{B}}\left(\frac{\displaystyle{R^N}}{\displaystyle{N}}+o_{\epsilon}(1)\right)\exp\left(p\int_0^Rs^{N-1}{U_{\epsilon}(s)}\dss\right),\vspace{5pt}\\
\displaystyle\int_0^{R}s^{N-1}e^{-qU_{\epsilon}(s)}\dss&=\left(\frac{\displaystyle{R^N}}{\displaystyle{N}}+o_{\epsilon}(1)\right)\exp\left(-q\int_0^Rs^{N-1}{U_{\epsilon}(s)}\dss\right)
\end{cases}
\,\,(0<B<A),\label{0818-nig2-ad}
\end{align}
respectively.

Moreover, for the case of $0<A<B$,
by (\ref{0818-nig1}) and (\ref{0818-nig2}) we establish the following estimates for $U_{\epsilon}$:
\begin{itemize}
\item\,\,as $0<\epsilon\ll1$, there holds
\begin{align*}
 \epsilon^2U_{\epsilon}'(r)=-\sqrt{\frac{2A}{qg(r)}}\exp\left({-\frac{q}{2}\left(U_{\epsilon}(r)-\frac{2}{q}\log\epsilon\right)}\right)+o_{\epsilon}(1)
\end{align*}
 for $r\in[\epsilon^\kappa,R]$ and $0<\kappa<1$ (cf. (\ref{2016-1010-guochin}) and (\ref{etae-t}));
\item\,\,if $R-r_{\epsilon,\beta}=(\gamma_\beta+o_{\epsilon}(1))\epsilon^\beta$ as $0<\epsilon\ll1$, then 
\begin{align*}
\exp\left({-\frac{q}{2}\left(U_{\epsilon}(r_{\epsilon,\beta})-\frac{2}{q}\log\epsilon\right)}\right)=&\left(\frac{N}{R(B-A)}\sqrt{\frac{2Ag(R)}{q}}+o_{\epsilon}(1)\right)\\
&+\epsilon^{\beta-2}\left(\gamma_{\beta}\sqrt{\frac{qA}{2g(R)}}+o_{\epsilon}(1)\right),
\end{align*}
where $\beta$ and $\gamma_{\beta}$ are constants independent of $\epsilon$ (cf. (\ref{psi-2016-7})).
Undoubtedly, the sign of $\beta-2$ strongly affects the asymptotics of $U_{\epsilon}(r_{\epsilon,\beta})$.
\end{itemize}
These estimates play crucial roles in dealing with the 
boundary concentration phenomena of $\left({\epsilon}U_{\epsilon}'\right)^2$
and $\rho_{\epsilon}$ and the pointwise descriptions for the boundary asymptotic behavior of $U_{\epsilon}$.

\section{Preliminary estimates}\label{sec0416addsec}
\subsection{Technical estimates for problems~(N) and (R)}\label{gentherm-sec}
\subsubsection{The inverse H\"{o}lder type estimate}\label{sec3thm31}
When $0<A<B$, by (\ref{bd1}) and (\ref{ce}), we have
$\partial_{\nu}u_\epsilon<0$ on $\partial\Omega$. Hence, 
 $u_\epsilon$ must attain its maximum value at interior points of $\Omega$.
On the other hand, 
 by following the same argument as in Theorem~3.1(i) of \cite{L1},
we can obtain that $u_{\epsilon}$ attains its minimum value at boundary points.
Similarly, when $A>B>0$, 
 $u_\epsilon$ must attain its minimum value at interior points of $\Omega$
and attains its maximum value at boundary points.

Now we want to prove Lemma~\ref{lem1}(I). Assume $0<A<B$.
For each $u_{\epsilon}$ we define
\begin{align}\label{sec-thm-id1}
w_{\epsilon}(x)=\frac{Ae^{pu_\epsilon(x)}}{\xint\,e^{pu_\epsilon(y)}\dy}-
\frac{Be^{-qu_\epsilon(x)}}{\xint\,e^{-qu_\epsilon(y)}\dy}\quad\mathrm{for}\,\,x\in\overline{\Omega}.
\end{align}
Since the maximum value of $u_\epsilon$ is located at an interior point~$x_0\in\Omega$,
by (\ref{eq1}) and (\ref{sec-thm-id1})
one finds $w_{\epsilon}(x_0)={\epsilon}^2\nabla\cdot(g\nabla{u_\epsilon})(x_0)\leq0$.
This immediately implies
\begin{align}\label{add-w}
w_{\epsilon}(x)\leq{w}_{\epsilon}(x_0)\leq0\quad \mathrm{for}\,\, x\in\overline{\Omega}
\end{align}
due to the fact that $w_{\epsilon}$ is monotonically increasing with respect to $u_{\epsilon}$.
As a consequence, we obtain 
\begin{align}\label{sec-thm-id1-0814}
e^{(p+q)u_\epsilon(x)}\leq
\frac{B}{A}{{\big|}{\big|}e^{u_{\epsilon}}{\big|}{\big|}_{L^{p}(\Omega)}^{p}}{{\big|}{\big|}e^{-u_{\epsilon}}{\big|}{\big|}_{L^{q}(\Omega)}^{-q}},\quad\mathrm{for}\,\,x\in\overline{\Omega}.
\end{align} 
On the other hand, 
\begin{align}
&\max_{\overline{\Omega}}\left(\frac{Ae^{pu_\epsilon}}{\xint\,e^{pu_\epsilon(y)}\dy}-
\frac{Be^{-qu_\epsilon}}{\xint\,e^{-qu_\epsilon(y)}\dy}\right)\notag\\
=&\frac{A}{\xint\,e^{p(u_\epsilon(y)-u_\epsilon(x_0))}\dy}-
\frac{B}{\xint\,e^{-q(u_\epsilon(y)-u_\epsilon(x_0))}\dy}\notag\\
\geq&\,{A-B}.\notag
\end{align}
Along with (\ref{add-w}) gives (\ref{id-lem-0}).
Similarly, for the case of $0<B<A$, we can prove (\ref{id-lem-0-add0405}).

Now we shall prove (\ref{id-lem-2}).
Integrating (\ref{sec-thm-id1-0814}) over $\Omega$ and applying the H\"{o}lder inequality
to the expansion,
one may check that
$\left(|\Omega|^{-\frac{1}{p}}{\big|}{\big|}e^{u_{\epsilon}}{\big|}{\big|}_{L^p(\Omega)}\right)^{p+q}
\leq|\Omega|^{-1}{\big|}{\big|}e^{u_{\epsilon}}{\big|}{\big|}_{L^{p+q}(\Omega)}^{p+q}
\leq\frac{B}{A}{{\big|}{\big|}e^{u_{\epsilon}}{\big|}{\big|}_{L^{p}(\Omega)}^{p}}{{\big|}{\big|}e^{-u_{\epsilon}}{\big|}{\big|}_{L^{q}(\Omega)}^{-q}}$
which immediately gives 
\begin{align*}
\big\|e^{u_{\epsilon}}{\big\|}_{L^{p}(\Omega)}
{\big\|}e^{-u_{\epsilon}}{\big\|}_{L^{q}(\Omega)}
\leq|\Omega|^{\frac{1}{p}+\frac{1}{q}}\left(\frac{B}{A}\right)^{\frac{1}{q}}.
\end{align*}
Applying the H\"{o}lder inequality again, we find
\begin{align*}
&{\big|}{\big|}e^{u_{\epsilon}}{\big|}{\big|}_{L^{p}(\Omega)}
{\big|}{\big|}e^{-u_{\epsilon}}{\big|}{\big|}_{L^{q}(\Omega)}\\
=&
\left({\Big|}{\Big|}e^{\frac{pq}{p+q}u_{\epsilon}}{\Big|}{\Big|}_{L^{\frac{p+q}{q}}(\Omega)}
{\Big|}{\Big|}e^{-\frac{pq}{p+q}u_{\epsilon}}{\Big|}{\Big|}_{L^{\frac{p+q}{p}}(\Omega)}\right)^{\frac{1}{p}+\frac{1}{q}}\\
\geq&|\Omega|^{\frac{1}{p}+\frac{1}{q}}.\notag
\end{align*}
Hence, we obtain (\ref{id-lem-2}) for the case $0<A<B$.

Similarly, for the case $0<B<A$, we can prove
\begin{align}
|\Omega|^{\frac{1}{p}+\frac{1}{q}}\leq\big\|e^{u_{\epsilon}}{\big\|}_{L^{p}(\Omega)}
{\big\|}e^{-u_{\epsilon}}{\big\|}_{L^{q}(\Omega)}
\leq|\Omega|^{\frac{1}{p}+\frac{1}{q}}\left(\frac{A}{B}\right)^{\frac{1}{p}}.\notag
\end{align}
Therefore, we get (\ref{id-lem-2}) without the constraint~(\ref{aneqb})
and complete the proof of Lemma~\ref{lem1}(I).

Now we deal with (\ref{id-lem-1}) under the constraint~(\ref{aneqb}).
Firstly, we establish a lower bound of $\frac{pAe^{pu_\epsilon}}{\int_{\Omega}e^{pu_\epsilon(y)}\dy}+
\frac{qBe^{-qu_\epsilon}}{\int_{\Omega}e^{-qu_\epsilon(y)}\dy}$ on $\overline{\Omega}$.
By (\ref{id-lem-2}) and the Young's inequality, one obtains the estimate
\begin{align}\label{cc-0814-1147am}
&\,\frac{pAe^{pu_\epsilon}}{\int_{\Omega}e^{pu_\epsilon(y)}\dy}+
\frac{qBe^{-qu_\epsilon}}{\int_{\Omega}e^{-qu_\epsilon(y)}\dy}\notag\\
=&\,pA\left(\frac{e^{u_{\epsilon}}}{\big\|e^{u_{\epsilon}}{\big\|}_{L^{p}(\Omega)}}\right)^p
+qB\left(\frac{e^{-u_{\epsilon}}}{\big\|e^{-u_{\epsilon}}{\big\|}_{L^{q}(\Omega)}}\right)^q\notag\\
\geq&\,(p+q)
\left(\frac{\left(\frac{pA}{q}\right)^{\frac{1}{p}}\left(\frac{qB}{p}\right)^\frac{1}{q}}{\big\|e^{u_{\epsilon}}{\big\|}_{L^{p}(\Omega)}
{\big\|}e^{-u_{\epsilon}}{\big\|}_{L^{q}(\Omega)}}\right)^{\frac{pq}{p+q}}\\
\geq&\frac{A(p+q)}{|\Omega|}\left(\frac{q}{p}\right)^{\frac{p-q}{p+q}}.\notag
\end{align}
Hence, by (\ref{eq1}), (\ref{id-lem-0}),
 (\ref{sec-thm-id1}), (\ref{add-w}) and (\ref{cc-0814-1147am}), we have
 \begin{align}\label{cc-0814-1}
\epsilon^2\nabla\cdot\left(g\nabla{w_\epsilon}\right)
=&\left(\frac{pAe^{pu_\epsilon}}{\xint\,e^{pu_\epsilon(y)}\dy}+
\frac{qBe^{-qu_\epsilon}}{\xint\,e^{-qu_\epsilon(y)}\dy}\right)w_\epsilon\notag\\
&+\epsilon^2\left(\frac{p^2Ae^{pu_\epsilon}}{\xint\,e^{pu_\epsilon(y)}\dy}-
\frac{q^2Be^{-qu_\epsilon}}{\xint\,e^{-qu_\epsilon(y)}\dy}\right)g|\nabla{u}_\epsilon|^2\\
\leq&A(p+q)\left(\frac{q}{p}\right)^{\frac{p-q}{p+q}}w_\epsilon.\notag
\end{align}
Here we have used~(\ref{id-lem-0}), (\ref{add-w}), (\ref{cc-0814-1147am})
and the condition~$0<p\leq{q}$ to verify
\begin{align}\label{01012017}
\epsilon^2\left(\frac{p^2Ae^{pu_\epsilon}}{\xint\,e^{pu_\epsilon(y)}\dy}-
\frac{q^2Be^{-qu_\epsilon}}{\xint\,e^{-qu_\epsilon(y)}\dy}\right)g|\nabla{u}_\epsilon|^2\leq0,
\end{align}
which gives the final estimate in (\ref{cc-0814-1}).

In order to estimate $w_\epsilon$, we consider a linear equation as follows:
\begin{align}\label{id-lem-3}
\epsilon^2\nabla\cdot\left(g\nabla{v_\epsilon}\right)
={A(p+q)}\left(\frac{q}{p}\right)^{\frac{p-q}{p+q}}v_\epsilon\quad\mathrm{in}\,\,\Omega,\quad{v_\epsilon}=1\quad\mathrm{on}\,\,\partial\Omega,
\end{align}
which has a unique solution satisfying $v_{\epsilon}\geq0$.
Since ${A(p+q)}\left(\frac{q}{p}\right)^{\frac{p-q}{p+q}}$ 
is a positive constant independent of $\epsilon$, $v_{\epsilon}$ possesses the following property:
\begin{lemma}\label{lem2}(cf. Proposition~2 of \cite{L2})
\begin{itemize}
\item[(i)] There exists a positive constant $\alpha$ independent of $\epsilon$
such that for $\epsilon>0$,
\begin{align}\label{id-lem-4}
\Xint\,v_\epsilon\,\dx\leq\alpha\epsilon.
\end{align}
\item[(ii)] For any compact subset $K$ of $\Omega$,
there exist positive constants $\beta_{K}$ and $M_K$ independent of $\epsilon$ such that
\begin{align}
\max_{K}v_{\epsilon}\leq\beta_{K}e^{-\frac{M_K}{\epsilon}}.
\end{align}
\end{itemize}
\end{lemma}
Now we want to prove (\ref{id-lem-1}).
Note that $\displaystyle\left(\min_{\overline{\Omega}}w_\epsilon\right)v_\epsilon=\min_{\overline{\Omega}}w_\epsilon\leq{w_\epsilon}$ on $\partial\Omega$.
Hence, by (\ref{cc-0814-1}) and (\ref{id-lem-3}), 
$\displaystyle\left(\min_{\overline{\Omega}}w_\epsilon\right)v_\epsilon$
is a subsolution of $w_\epsilon$, i.e.,
\begin{align}\label{1040-0814}
\left(\min_{\overline{\Omega}}w_\epsilon\right)v_\epsilon\leq{w}_\epsilon\quad\mathrm{in}\,\,\overline{\Omega}.
\end{align}
Note also that $\xint\,w_\epsilon\,\dx=A-B<0$. Hence,
by the combinations of (\ref{sec-thm-id1}), (\ref{id-lem-4}) and (\ref{1040-0814}), it yields
\begin{align}\label{0102-id-fd}
\min_{\overline{\Omega}}w_\epsilon\leq\frac{\xint\,w_\epsilon\,\dx}{\xint\,v_\epsilon\,\dx}\leq\frac{A-B}{\alpha\epsilon}.
\end{align}
By (\ref{sec-thm-id1}) and (\ref{0102-id-fd}) we get (\ref{id-lem-1}) and complete the proof of Lemma~\ref{lem1}(II-i).
Similarly, we can prove Lemma~\ref{lem1}(II-ii).
Therefore, we complete the proof of Lemma~\ref{lem1}.

\subsubsection{A boundary blow-up estimate}. 
In this section we give the proof of Theorem~\ref{thm1}. Assume $0<A<B$ and $0<p\leq{q}$.
We now deal with (\ref{thm1-id1}).
By (\ref{id-lem-2}) and (\ref{sec-thm-id1-0814}), we have
\begin{align}\label{0815-cu}
e^{(p+q)u_\epsilon(x)}
\leq&\frac{B\big\|e^{u_{\epsilon}}{\big\|}_{L^{p}(\Omega)}^p}
{A\big\|e^{-u_{\epsilon}}{\big\|}_{L^{q}(\Omega)}^{q}}\notag\\
=&\frac{B}{A}\left(\frac{\big\|e^{u_{\epsilon}}{\big\|}_{L^{p}(\Omega)}\big\|e^{-u_{\epsilon}}{\big\|}_{L^{q}(\Omega)}}
{\big\|e^{-u_{\epsilon}}{\big\|}_{L^{q}(\Omega)}^{1+\frac{q}{p}}}\right)^{p}\notag\\
\leq&\left(\frac{B|\Omega|}{A}\right)^{1+\frac{p}{q}}\big\|e^{-u_{\epsilon}}{\big\|}_{L^{q}(\Omega)}^{-\left(p+q\right)}\\
=&\left(\frac{B}{A}\right)^{1+\frac{p}{q}}e^{-\left(1+\frac{p}{q}\right)
\log\left(|\Omega|^{-1}\int_{\Omega}e^{-qu_{\epsilon}}\dy\right)}\notag\\
\leq&\left(\frac{B}{A}\right)^{1+\frac{p}{q}},\quad\mathrm{for}\,\,x\in\overline{\Omega}.\notag
\end{align}
Here we have used the Jensen's inequality and the constraint (\ref{1129-510p})
to get $\log\left(\xint\,e^{-qu_\epsilon}\dy\right)\geq-q\xint\,u_\epsilon\,\dy=0$, which justifies the last term of (\ref{0815-cu}).
Hence, we arrive at
$u_\epsilon(x)\leq\frac{1}{q}\log\frac{B}{A}$ which gives the right-hand inequality of (\ref{thm1-id1}).
On the other hand, since $u_\epsilon$ is a smooth nontrivial function satisfying $\int_{\Omega}u_\epsilon\,\dx=0$,
we immediately get $\displaystyle\max_{\overline{\Omega}}u_\epsilon>0$. 
This completes the proof of  (\ref{thm1-id1}). 

Next, we deal with $\displaystyle\min_{\overline{\Omega}}u_\epsilon$.
By (\ref{id-lem-1}), one may check that 
\begin{align}\notag
\frac{A-B}{\alpha\epsilon}\geq-
\frac{Be^{-q\min_{\overline{\Omega}}u_\epsilon}}{\xint\,e^{-qu_\epsilon}\dy}\geq
-{B}e^{q\left(\max_{\overline{\Omega}}u_\epsilon-\min_{\overline{\Omega}}u_\epsilon\right)}.
\end{align}
Along with (\ref{thm1-id1}) gives
\begin{align}
\min_{\overline{\Omega}}u_\epsilon\leq&\max_{\overline{\Omega}}u_\epsilon
-\frac{1}{q}\log\frac{B-A}{B\alpha\epsilon}\notag\\
\leq&-\frac{1}{q}\log\frac{1}{\epsilon}+\frac{1}{q}\log\frac{B^2\alpha}{A(B-A)}.\notag
\end{align}
Therefore, we get (\ref{thm1-id2}) and complete the proof of Theorem~\ref{thm1}(i).
For the case $A>B>0$, we may follow the same argument to prove Theorem~\ref{thm1}(ii).

It remains to prove (\ref{thm1-id5}).
Recall that $u_\epsilon$ attains its maximum value at an interior point~$x_0\in\Omega$.
Let us consider the difference $\mathcal{U}_{\epsilon}(x)\equiv{u}_\epsilon(x_0)-u_\epsilon(x)$ for $x\in\Omega$. Then 
\begin{align}
\mathcal{U}_{\epsilon}\geq0\quad\mathrm{in}\,\,\Omega, 
\end{align}
and by (\ref{eq1}) and (\ref{add-w}), one may check that
\begin{align}\label{0815-1129ni}
\epsilon^2\nabla\cdot(g\nabla{\mathcal{U}_{\epsilon}})(x)=-w_{\epsilon}(x)\geq{w}_{\epsilon}(x_0)-w_{\epsilon}(x).
\end{align}
We shall claim:

\begin{lemma}\label{lem-0815}
 There holds
\begin{align}\label{0815-1129-ni2}
{w}_{\epsilon}(x_0)-w_{\epsilon}(x)\geq{A(p+q)}\left(\frac{q}{p}\right)^{\frac{p-q}{p+q}}\mathcal{U}_{\epsilon}(x),\quad\forall\,\,x\in\Omega.
\end{align}
\end{lemma}
\begin{proof}
If $\mathcal{U}_{\epsilon}(\hat{x})=0$ at a point $\hat{x}\in\Omega$, then (\ref{0815-1129-ni2}) holds
at $\hat{x}$ trivially due to (\ref{add-w}). Hence, without loss of generality, we
may assume $\mathcal{U}_{\epsilon}(x)>0$ for $x\in\Omega$. 
Using the inequality $(e^a-e^b)(a-b)\geq{e}^{\frac{a+b}{2}}(a-b)^2$ ($a,\,b\in\mathbb{R}$)
and following the similar argument as in (\ref{cc-0814-1147am}),
one may check that
\begin{align}
&\left({w}_{\epsilon}(x_0)-w_{\epsilon}(x)\right)\mathcal{U}_{\epsilon}(x)\notag\\
&\quad\quad\quad=\left[\frac{A\left(e^{pu_\epsilon(x_0)}-e^{pu_\epsilon(x)}\right)}{\int_{\Omega}e^{pu_\epsilon(y)}\dy}-
\frac{B\left(e^{-qu_\epsilon(x_0)}-e^{-qu_\epsilon(x)}\right)}{\int_{\Omega}e^{-qu_\epsilon(y)}\dy}\right](u_\epsilon(x_0)-u_\epsilon(x))\notag\\
&\quad\quad\quad\geq\left(\frac{pAe^{\frac{p}{2}(u_\epsilon(x_0)+u_\epsilon(x))}}{\int_{\Omega}e^{pu_\epsilon(y)}\dy}+
\frac{qBe^{-\frac{q}{2}(u_\epsilon(x_0)+u_\epsilon(x))}}{\int_{\Omega}e^{-qu_\epsilon(y)}\dy}\right)(u_\epsilon(x_0)-u_\epsilon(x))^2\notag\\
&\quad\quad\quad\geq{A(p+q)}\left(\frac{q}{p}\right)^{\frac{p-q}{p+q}}\mathcal{U}_{\epsilon}^2(x).\notag
\end{align}
Since $\mathcal{U}_{\epsilon}(x)>0$, the above estimate gives (\ref{0815-1129-ni2}). 
This completes the proof of Lemma~\ref{lem-0815}. 
\end{proof}

By (\ref{0815-1129ni}) and (\ref{0815-1129-ni2}), we have
\begin{align}\label{0816-habith}
\epsilon^2\nabla\cdot(g\nabla{\mathcal{U}_{\epsilon}})(x)\geq{A(p+q)}\left(\frac{q}{p}\right)^{\frac{p-q}{p+q}}\mathcal{U}_{\epsilon}(x),
\end{align}
for $x\in\Omega$. Note that $\displaystyle\mathcal{U}_{\epsilon}\leq\left(\max_{x\in\overline{\Omega}}|u_\epsilon(x_0)-u_\epsilon(x)|\right)v_\epsilon$ on $\partial\Omega$, where $v_\epsilon$ is defined in (\ref{id-lem-3}).
As a consequence, by (\ref{id-lem-3}), (\ref{0816-habith}) and Lemma~\ref{lem2}(ii),
we obtain that for any compact subset $K$ of $\Omega$, there holds
\begin{align}
\max_{x\in{K}}\left(u_\epsilon(x_0)-u_\epsilon(x)\right)
=\max_{x\in{K}}\mathcal{U}_{\epsilon}(x)\leq&\left(\max_{x\in\overline{\Omega}}|u_\epsilon(x_0)-u_\epsilon(x)|\right)\max_{x\in{K}}v_\epsilon(x)\notag\\
\leq&\left(\max_{x\in\overline{\Omega}}|u_\epsilon(x_0)-u_\epsilon(x)|\right)\beta_Ke^{-\frac{M_K}{\epsilon}}.\notag
\end{align}
Since $\displaystyle\max_{\overline{\Omega}}u_\epsilon=u_\epsilon(x_0)$,
the above estimate concludes 
\begin{align}\notag
\frac{\displaystyle\max_{x,y\in{K}}|u_\epsilon(x)-u_\epsilon(y)|}{\displaystyle\max_{x,y\in\overline{\Omega}}|u_\epsilon(x)-u_\epsilon(y)|}\leq\frac{\displaystyle\max_{x\in{K}}\left(u_\epsilon(x_0)-u_\epsilon(x)\right)}{\displaystyle\max_{x\in\overline{\Omega}}(u_\epsilon(x_0)-u_\epsilon(x))}\leq\beta_Ke^{-\frac{M_K}{\epsilon}}.
\end{align}
Hence, we prove (\ref{id-lem-2}) for the case $A<B$.

The proof of (\ref{id-lem-2}) for the case $A>B$ is similar to that for the case $A<B$.
Therefore, we complete the proof of Theorem~\ref{thm1}.

\begin{corollary}\label{cor1205-2016}
Under the same hypotheses as in Theorem~\ref{thm1}, for $A\neq{B}$,  we have
\begin{align}\label{pneq-16}
1\leq\Xint\,e^{pu_\epsilon}\dx\leq\max\left\{\frac{A}{B},\left(\frac{B}{A}\right)^{\frac{p}{q}}\right\},\quad
1\leq\Xint\,e^{-qu_\epsilon}\dx\leq\max\left\{\frac{B}{A},\left(\frac{A}{B}\right)^{\frac{q}{p}}\right\}.
\end{align}
\end{corollary}
\begin{proof}
Following the same argument as for the third line of (\ref{0815-cu}),
we can use the Jensen's inequality and (\ref{1129-510p}) to obtain
the left-hand sides of both two inequalities in (\ref{pneq-16}).
Along with (\ref{id-lem-2}), 
we immediately get the right-hand sides of two inequalities in (\ref{pneq-16}). 
This completes the proof of (\ref{pneq-16}).
\end{proof}

\begin{remark}\label{rk-thm1}
It is worth mentioning that,
by (\ref{pneq-16}),
both coefficients of (\ref{eq1})
have a finite upper bound and a positive lower bound with respect to $\epsilon$.
Thus, equation~(\ref{eq1}) can be regarded as a standard Poisson--Boltzmann equation with bounded coefficients
$\frac{A}{\int_{\Omega}e^{pu_\epsilon(y)}\dy}$ and $\frac{B}{\int_{\Omega}e^{-qu_\epsilon(y)}\dy}$.
 However, an interesting consequence of Theorem~\ref{thm1} is that
 the solution of (\ref{eq1})--(\ref{bd1}) still asymptotically blows up at boundary points
as $\epsilon$ approaches zero.
Such behaviors cannot be found in the standard solutions of PB equations.
\end{remark}
\begin{remark}\label{20170101rk}
We stress that
in the proof of Theorem~\ref{thm1},
the constraint~(\ref{aneqb}) cannot be removed
because in (\ref{cc-0814-1}), we need the estimate (\ref{01012017}) which is obtained from
 (\ref{aneqb}), (\ref{sec-thm-id1}) and (\ref{add-w}). Without the constraint~(\ref{aneqb}),
(\ref{01012017}) seems difficult to be estimated.
\end{remark}
{\large \textbf{Direct extension of Theorem~\ref{thm1} to problem~(R).}}

 Note that (\ref{eq1}) satisfies the shift-invariance. 
In problem~(R), we may set $\widehat{u}_{\epsilon}={u}_{\epsilon}-\xint{u}_{\epsilon}(z)dz$
which gives $\xint\widehat{u}_{\epsilon}(x)\dx=0$
and $\displaystyle\max_{x,y\in\overline{\Omega}}|\widehat{u}_{\epsilon}(x)-\widehat{u}_{\epsilon}(y)|=\max_{x,y\in\overline{\Omega}}|{u}_{\epsilon}(x)-{u}_{\epsilon}(y)|$. Since,
 $\widehat{u}_{\epsilon}$ satisfies a Robin boundary condition,
we can follow the same arguments as in Theorem~1.1 of \cite{L2}, Theorem~\ref{thm1}
and Corollary~\ref{cor1205-2016} to obtain the following theorem:
\begin{theorem}
Under the same hypotheses as in Theorem~\ref{thm1},
let $u_\epsilon\in{C^\infty(\Omega)\cap{C^1(\overline{\Omega})}}$ be the unique solution of problem~(R)
for $\epsilon>0$. 
Then the maximum difference 
$\displaystyle\max_{x,y\in\overline{\Omega}}\left|u_\epsilon(x)-u_\epsilon(y)\right|
\geq\frac{1}{\max\{p,q\}}\log\frac{1}{\epsilon}+O(1)$ asymptotically blows up as $\epsilon\downarrow0$, 
and (\ref{thm1-id5}) also holds true. Moreover, we have
\begin{align}\label{r-pneq-16}
1\leq\frac{\xint\,e^{pu_\epsilon}\dx}{e^{p-\hspace{-6pt}\int_{\Omega}u_\epsilon\,\dx}}\leq\max\left\{\frac{A}{B},\left(\frac{B}{A}\right)^{\frac{p}{q}}\right\},\quad
1\leq\frac{\xint\,e^{-qu_\epsilon}\dx}{e^{-q-\hspace{-6pt}\int_{\Omega}u_\epsilon\,\dx}}\leq\max\left\{\frac{B}{A},\left(\frac{A}{B}\right)^{\frac{q}{p}}\right\}.
\end{align}
Hence, both $\xint\,e^{pu_\epsilon}\dx$ and $\xint\,e^{-qu_\epsilon}\dx$
are uniformly bounded to $\epsilon$ if and only if  $\xint\,u_\epsilon\dx$
is uniformly bounded to $\epsilon$.
\end{theorem}

\subsection{Technical estimates for problem~(N*)}\label{prelimnstar}
In the whole section, we state crucial estimates for solutions $U_{\epsilon}$ to problem~(N*).
A difficulty for studying problem~(N*) comes mainly from the
 coefficients $\int_0^{R}s^{N-1}e^{pU_{\epsilon}(s)}\dss$ and $\int_0^{R}s^{N-1}e^{-qU_{\epsilon}(s)}\dss$,
which depending on the unknown solution~$U_{\epsilon}$.
Understandably, the asymptotics of these unknown coefficients and the asymptotic behaviors of $U_{\epsilon}(0)$ and $U_{\epsilon}(R)$
are closely affected each other. To proceed our study,
we shall develop a series of estimates for 
those unknown coefficients  
 directly from the structure of the equation~(\ref{bigu1})--(\ref{bigu3}),
and establish boundary and interior estimates of ${U}_{\epsilon}$ for $0<\epsilon\ll1$.
In accordance with these estimates, 
 we will investigate the exact leading order terms of those 
unknown coefficients as $\epsilon$ approaches zero.


Assume $0<A<B$. 
Note that $u_\epsilon(x)=U_\epsilon(r)$ for $r=|x|\leq{R}$. Thus,
by Lemma~\ref{lem1}(I-i) and (\ref{bigu1}), one finds
\begin{align}\label{1110822-1-2016}
\left(g(r)r^{N-1}U_{\epsilon}'(r)\right)'={g}(r)r^{N-1}\left[U_{\epsilon}''(r)+\left(\frac{N-1}{r}
+\frac{g'(r)}{g(r)}\right)U_{\epsilon}'(r)\right]\leq0,\quad\mathrm{for}\,\,r\in(0,R).
\end{align}
Since $g(r)>0$ on $[0,R]$
and $U_{\epsilon}'(0)=0$, by (\ref{1110822-1-2016}) we concludes
\begin{align}\label{780822-1-2016}
0<A<B\,\,\Longrightarrow\,\,{U}_{\epsilon}'\leq0\quad\mathrm{on}\,\,[0,R].
\end{align}
Similarly, we have
\begin{align}\label{780822-1-2016-00}
0<B<A\,\,\Longrightarrow\,\,{U}_{\epsilon}'\geq0\quad\mathrm{on}\,\,[0,R].
\end{align}
Applying ${u}_{\epsilon}(x)={U}_{\epsilon}(|x|)$ to Theorem~\ref{thm1}
and using (\ref{780822-1-2016})--(\ref{780822-1-2016-00}),
one concludes the followings:
\begin{align}
0<A<B\,\,\,\quad\quad\quad\quad\quad\quad\Longrightarrow&\quad\quad\quad\quad0<\max_{[0,R]}{U}_{\epsilon}={U}_{\epsilon}(0)\leq\frac{1}{q}\log\frac{B}{A},\label{0828-id1}\\
0<A<B\,\,\mathrm{and}\,\,0<p\leq{q}\,\Longrightarrow&\quad\quad\quad\quad\quad\,\,\,\,\,\,\min_{[0,R]}U_\epsilon={U}_{\epsilon}(R)\leq-\frac{1}{q}\log\frac{1}{\epsilon}+O(1),\label{0828-id1-0}\\
0<B<A\,\,\,\,\quad\quad\quad\quad\quad\quad\Longrightarrow&\,\,\,-\frac{1}{p}\log\frac{A}{B}\leq\min_{[0,R]}{U}_{\epsilon}={U}_{\epsilon}(0)<0,\label{0828-id2}\\
0<B<A\,\,\mathrm{and}\,\,0<q\leq{p}\,\Longrightarrow&\quad\quad\quad\quad\quad\,\,\,\,\,\,\max_{[0,R]}U_\epsilon={U}_{\epsilon}(R)\geq\frac{1}{p}\log\frac{1}{\epsilon}+O(1),\label{0828-id2-0}
\end{align}
as $0<\epsilon\ll1$.

For non-local coefficients of (\ref{bigu1}), we have the following estimates:
\begin{lemma}\label{lem-aug22-1}
Assume that $A,\,B,\,p$, and $q$ are positive constants independent of $\epsilon$ and satisfy $A\neq{B}$.
For $\epsilon>0$, 
let $U_\epsilon\in{C}^1([0,R])\cap{C}^{\infty}((0,R))$
be the unique solution of problem~(N*). Then there hold
\begin{align}
&\frac{R^N}{N}\leq\int_0^{R}s^{N-1}e^{pU_{\epsilon}(s)}\dss\leq\frac{R^N}{N}\max\left\{\frac{A}{B},\left(\frac{B}{A}\right)^{\frac{p}{q}}\right\},\label{holderrd1}\\
 &\frac{R^N}{N}\leq\int_0^{R}s^{N-1}e^{-qU_{\epsilon}(s)}\dss\leq\frac{R^N}{N}\max\left\{\frac{B}{A},\left(\frac{A}{B}\right)^{\frac{q}{p}}\right\},\label{holderrd2}
\end{align}
and
\begin{align}\label{08250611}
\frac{pAe^{pU_{\epsilon}(r)}}{\int_0^{R}s^{N-1}e^{pU_{\epsilon}(s)}\dss}+&\frac{qBe^{-qU_{\epsilon}(r)}}{\int_0^{R}s^{N-1}e^{-qU_{\epsilon}(s)}\dss}\notag\\[-0.7em]
&\\[-0.7em]
&\geq\frac{N(p+q)\min\{A,B\}}{R^{N}}\left(\frac{q}{p}\right)^{\frac{p-q}{p+q}},\quad\forall\,r\in[0,R].\notag
\end{align}
\end{lemma}
\begin{proof}
Since $u_{\epsilon}(x)=U_{\epsilon}(|x|)$,
(\ref{holderrd1}) and (\ref{holderrd2})
are direct results of (\ref{pneq-16}).
Using the argument of (\ref{id-lem-2}) 
 and following the derivation of (\ref{cc-0814-1147am}), we can prove (\ref{08250611}).
Here we omit the detailed proof. 
\end{proof}

We emphasize again that
boundary asymptotic blow-up estimates (\ref{0828-id1-0})
and (\ref{0828-id2-0}) of $U_\epsilon$ are established under an extra constraint~(\ref{aneqb}),
but (\ref{holderrd1})--(\ref{08250611}) hold true without the constraint~(\ref{aneqb}), 
which are direct results from Theorem~\ref{thm1}.

\subsubsection{The Poho\v{z}aev's identity and interior estimates}\label{sec31-pre}
 Let $\kappa\in(0,1)$ be a parameter independent of $\epsilon$.
For $t\in(0,R]$, multiplying equation~(\ref{bigu1}) by $U_{\epsilon}'(r)$ and
integrating the expression over the interval~$[\epsilon^\kappa,t]$ (or $[t,\epsilon^\kappa]$) gives
\begin{align}
&\frac{R^N}{N}\left(\frac{A\left(e^{pU_{\epsilon}(t)}-e^{pU_{\epsilon}(\epsilon^\kappa)}\right)}{p\int_0^{R}s^{N-1}e^{pU_{\epsilon}(s)}\dss}
+\frac{B\left(e^{-qU_{\epsilon}(t)}-e^{-qU_{\epsilon}(\epsilon^\kappa)}\right)}
{q\int_0^{R}s^{N-1}e^{-qU_{\epsilon}(s)}\dss}\right)\notag\\[-0.7em]
&\label{bigoneadd-as}\\[-0.7em]
&\quad\quad\quad=\frac{\epsilon^2}{2}\left(g(t)U_{\epsilon}'^2(t)
-g(\epsilon^\kappa)U_{\epsilon}'^2(\epsilon^\kappa)\right)+\frac{\epsilon^2}{2}\int_{\epsilon^\kappa}^t\left(\frac{2(N-1)g(r)}{r}+g'(r)\right)U_{\epsilon}'^2(r)\dr.\notag
\end{align}
Here we have used integration by parts twice to examine directly
\begin{align}
&\epsilon^2\int_{\epsilon^{\kappa}}^tg(r)\left[U_{\epsilon}''(r)+\left(\frac{N-1}{r}
+\frac{g'(r)}{g(r)}\right)U_{\epsilon}'(r)\right]U_{\epsilon}'(r)\dr\notag\\
&\quad\quad\quad=\frac{\epsilon^2}{2}\left(g(t)U_{\epsilon}'^2(t)-g(\epsilon^{\kappa})U_{\epsilon}'^2(\epsilon^{\kappa})\right)+
\frac{\epsilon^2}{2}\int_{\epsilon^\kappa}^t\left(\frac{2(N-1)g(r)}{r}+g'(r)\right)U_{\epsilon}'^2(r)\dr,\notag
\end{align}
which verifies the right-hand side of (\ref{bigoneadd-as}).

In the following Lemmas~\ref{lem-aug-21} and \ref{lem-aug22-2}, we state a series of identities and estimates
for ${U}_{\epsilon}$, where Lemma~\ref{lem-aug-21} is obtained 
from the standard derivation of the Poho\v{z}aev's identity.

\begin{lemma}[The Poho\v{z}aev's identity]\label{lem-aug-21}
Under the same hypotheses as in Lemma~\ref{lem-aug22-1}, we have
\begin{align}\label{sec3id1}
&\frac{R^N}{N}\left(\frac{Ae^{pU_{\epsilon}(R)}}{p\int_0^{R}s^{N-1}e^{pU_{\epsilon}(s)}\dss}
+\frac{Be^{-qU_{\epsilon}(R)}}
{q\int_0^{R}s^{N-1}e^{-qU_{\epsilon}(s)}\dss}\right)\notag\\[-0.7em]
&\\[-0.7em]
&\quad\quad\quad\quad\quad\quad\quad\quad\quad\quad\quad=\frac{R^2(A-B)^2}{2N^2\epsilon^2g(R)}+\frac{A}{p}+\frac{B}{q}+\Lambda_{1,\epsilon}(R),\notag
\end{align}
and
\begin{align}\label{sec3id2}
&\frac{R^N}{N}\left(\frac{Ae^{pU_{\epsilon}(\epsilon^{\kappa})}}{p\int_0^{R}s^{N-1}e^{pU_{\epsilon}(s)}\dss}
+\frac{Be^{-qU_{\epsilon}(\epsilon^{\kappa})}}
{q\int_0^{R}s^{N-1}e^{-qU_{\epsilon}(s)}\dss}\right)\notag\\[-0.7em]
&\\[-0.7em]
&\quad\quad\quad\quad\quad\quad\quad\quad\quad\quad\quad\quad=\frac{A}{p}+\frac{B}{q}+\Lambda_{1,\epsilon}(R)-\Lambda_{2,\epsilon}(R),\notag
\end{align}
where 
\begin{align}
&\Lambda_{1,\epsilon}(R)=\frac{\epsilon^2}{2}\int_0^R\frac{(N-2)g(s)+sg'(s)}{R^N}s^{N-1}U_{\epsilon}'^2(s)\dss,\label{lam1e}\\
\Lambda_{2,\epsilon}(R)&=-\frac{\epsilon^2}{2}g(\epsilon^\kappa)U_{\epsilon}'^2(\epsilon^\kappa)+\frac{\epsilon^2}{2}\int_{\epsilon^{\kappa}}^R\frac{2(N-1)g(s)+sg'(s)}{s}U_{\epsilon}'^2(s)\dss.\label{lam2e}
\end{align}
\end{lemma}
\begin{proof}
Firstly, we multiply equation~(\ref{bigu1}) by $r^NU_{\epsilon}'(r)$ to obtain
\begin{align}\label{ph1}
\epsilon^2\left(g(r)r^{N-1}U_{\epsilon}'(r)\right)'rU_{\epsilon}'(r)
=\frac{R^N}{N}\left(\frac{Ae^{pU_{\epsilon}(r)}}{\int_0^{R}s^{N-1}e^{pU_{\epsilon}(s)}\dss}
-\frac{Be^{-qU_{\epsilon}(r)}}
{\int_0^{R}s^{N-1}e^{-qU_{\epsilon}(s)}\dss}\right)r^NU_{\epsilon}'(r).
\end{align} 
Integrating (\ref{ph1}) over the interval $(0,R)$ and using the integration by parts,
one may check via simple calculations that
\begin{align}\label{ph2}
&\epsilon^2\int_0^R\left(g(r)r^{N-1}U_{\epsilon}'(r)\right)'rU_{\epsilon}'(r)\dr\notag\\[-0.7em]
&\\[-0.7em]
&\quad\quad\quad=\frac{R^2(A-B)^2}{2N^2\epsilon^2g(R)}+\frac{\epsilon^2}{2}\int_0^R\left[(N-2)g(r)+rg'(r)\right]r^{N-1}U_{\epsilon}'^2(r)\dr,\notag
\end{align}
and
\begin{align}\label{ph3}
&\int_0^R\left(\frac{Ae^{pU_{\epsilon}(r)}}{\int_0^{R}s^{N-1}e^{pU_{\epsilon}(s)}\dss}
-\frac{Be^{-qU_{\epsilon}(r)}}
{\int_0^{R}s^{N-1}e^{-qU_{\epsilon}(s)}\dss}\right)r^NU_{\epsilon}'(r)\dr\notag\\[-0.7em]
&\\[-0.7em]
&\quad\quad\quad=R^N\left(\frac{Ae^{pU_{\epsilon}(R)}}{p\int_0^{R}s^{N-1}e^{pU_{\epsilon}(s)}\dss}
+\frac{Be^{-qU_{\epsilon}(R)}}
{q\int_0^{R}s^{N-1}e^{-qU_{\epsilon}(s)}\dss}\right)-N\left(\frac{A}{p}+\frac{B}{q}\right).\notag
\end{align}
Here we have used the boundary condition~(\ref{bigu3}) to verify~(\ref{ph2}).
Consequently, by (\ref{ph1})--(\ref{ph3}) it follows
\begin{align}\label{sec3id1-a}
&\frac{R^N}{N}\left(\frac{Ae^{pU_{\epsilon}(R)}}{p\int_0^{R}s^{N-1}e^{pU_{\epsilon}(s)}\dss}
+\frac{Be^{-qU_{\epsilon}(R)}}
{q\int_0^{R}s^{N-1}e^{-qU_{\epsilon}(s)}\dss}\right)\notag\\[-0.7em]
&\\[-0.7em]
&\quad\quad\quad=\frac{R^2(A-B)^2}{2N^2\epsilon^2g(R)}+\frac{A}{p}+\frac{B}{q}+\frac{\epsilon^2}{2}\int_0^R\frac{(N-2)g(r)+rg'(r)}{R^N}r^{N-1}U_{\epsilon}'^2(r)\dr,\notag
\end{align} 
which gives (\ref{sec3id1}).

Setting $t=R$ in (\ref{bigoneadd-as}) and using the boundary condition~(\ref{bigu3}),
we obtain 
\begin{align}
&\frac{R^N}{N}\left(\frac{A\left(e^{pU_{\epsilon}(R)}-e^{pU_{\epsilon}(\epsilon^\kappa)}\right)}{p\int_0^{R}s^{N-1}e^{pU_{\epsilon}(s)}\dss}
+\frac{B\left(e^{-qU_{\epsilon}(R)}-e^{-qU_{\epsilon}(\epsilon^\kappa)}\right)}
{q\int_0^{R}s^{N-1}e^{-qU_{\epsilon}(s)}\dss}\right)\notag\\[-0.7em]
&\label{bigoneadd}\\[-0.7em]
&\quad\quad\quad=\frac{R^2(A-B)^2}{2N^2\epsilon^2g(R)}
-\frac{\epsilon^2}{2}g(\epsilon^\kappa)U_{\epsilon}'^2(\epsilon^\kappa)+\frac{\epsilon^2}{2}\int_{\epsilon^\kappa}^R\left(\frac{2(N-1)g(r)}{r}+g'(r)\right)U_{\epsilon}'^2(r)\dr.\notag
\end{align}
By (\ref{sec3id1-a}) and (\ref{bigoneadd}), we immediately get (\ref{sec3id2})
and complete the proof of Lemma~\ref{lem-aug-21}.
\end{proof}

\begin{lemma}\label{lem-aug22-2}
(Interior gradient estimates)
Let $\delta\in(0,R)$ be a constant independent of $\epsilon$.
Under the same hypotheses as in Lemma~\ref{lem-aug22-1}, as $\epsilon>0$ is sufficiently small, we have
\begin{align}\label{0823-99}
\left|U_{\epsilon}'(r)\right|\leq{C}_1{e}^{-\frac{\widetilde{M}}{\epsilon}\min\left\{\frac{\delta}{8},\frac{R-\delta}{4}\right\}}\,\,\,{uniformly}\,\,\,{in}\,\,[\frac{\sqrt{N-1}}{4\widetilde{M}}\epsilon,\delta),
\end{align}
and
\begin{align}\label{0823-99-eoei}
|U_{\epsilon}'(r)|\leq{C}_2\left(\epsilon^{-2}e^{-\frac{\widetilde{M}(R-r)}{2\epsilon}}+e^{-\frac{\widetilde{M}\delta}{8\epsilon}}\right),\quad{for}\,\,{r}\in[\delta,R),
\end{align}
where
\begin{align}\label{ccddot}
\widetilde{M}=\sqrt{\frac{A(p+q)}{\displaystyle\max_{[0,R]}g}\left(\frac{q}{p}\right)^{\frac{p-q}{p+q}}},\,\,C_1=\frac{R^N|A-B|}{N\left(\frac{\sqrt{N-1}}{4\widetilde{M}}\right)^{N-1}\displaystyle\min_{[0,R]}g},\,\,
C_2=\frac{R^N|A-B|}{N\delta^{N-1}\displaystyle\min_{[0,R]}g}
\end{align}
are positive constants independent of $\epsilon$.
\end{lemma}
\begin{proof}
Multiplying equation~(\ref{bigu1}) by $r^{N-1}$ and differentiating
the expression with respect to $r$, we have
\begin{align}\label{nighaug-923}
\epsilon^2V_{\epsilon}''(r)=&
\frac{R^N}{N}\left(\frac{pAe^{pU_{\epsilon}(r)}}{\int_0^{R}s^{N-1}e^{pU_{\epsilon}(s)}\dss}+\frac{qBe^{-qU_{\epsilon}(r)}}{\int_0^{R}s^{N-1}e^{-qU_{\epsilon}(s)}\dss}\right)\frac{V_{\epsilon}(r)}{g(r)}\notag\\[-0.7em]
&\\[-0.7em]
&+\frac{N-1}{r}\epsilon^2V_{\epsilon}'(r),\notag
\end{align}
for $r\in(0,R)$, where
\begin{align}\label{bigV}
V_{\epsilon}(r)\equiv{g}(r){r}^{N-1}U_{\epsilon}'(r).
\end{align}
Multiplying (\ref{nighaug-923}) by $V_{\epsilon}(r)$, one may check that
\begin{align}\label{0823gine}
\epsilon^2\left(V_{\epsilon}^2(r)\right)''=&2\epsilon^2\left(V_{\epsilon}''(r)V_{\epsilon}(r)+V_{\epsilon}'^2(r)\right)\notag\\
=&\frac{2R^N}{N}\left(\frac{pAe^{pU_{\epsilon}(r)}}{\int_0^{R}s^{N-1}e^{pU_{\epsilon}(s)}\dss}+\frac{qBe^{-qU_{\epsilon}(r)}}{\int_0^{R}s^{N-1}e^{-qU_{\epsilon}(s)}\dss}\right)\frac{V_{\epsilon}^2(r)}{g(r)}\notag\\
&+\epsilon^2\left(\frac{N-1}{r}V_{\epsilon}'(r)V_{\epsilon}(r)+2V_{\epsilon}'^2(r)\right)\\
\geq&\left[\frac{2A(p+q)}{\displaystyle\max_{[0,R]}g}\left(\frac{q}{p}\right)^{\frac{p-q}{p+q}}-\frac{\epsilon^2(N-1)}{8r^2}\right]V_{\epsilon}^2(r).\notag
\end{align}
Here we have used (\ref{08250611})
 and a basic estimate $\frac{N-1}{r}V_{\epsilon}'(r)V_{\epsilon}(r)+2V_{\epsilon}'^2(r)\geq-\frac{N-1}{8r^2}V_{\epsilon}^2(r)$ to verify the last line of (\ref{0823gine}).
Hence, one may check from (\ref{0823gine}) that 
 for $0<\epsilon<\frac{4\widetilde{M}R}{\sqrt{N-1}}$ and $r\in[\frac{\sqrt{N-1}}{4\widetilde{M}}\epsilon,R]$,
there holds
\begin{align}\label{0823-11223}
\epsilon^2\left(V_{\epsilon}^2(r)\right)''\geq\widetilde{M}^2V_{\epsilon}^2(r),
\end{align}
where $\widetilde{M}$ is defined by (\ref{ccddot}).
As a consequence,
by (\ref{bigu3}), (\ref{780822-1-2016}), (\ref{780822-1-2016-00}), (\ref{bigV}) and (\ref{0823-11223}), we arrive at
following estimates:\\

\textbf{(a)} For $r\in[\frac{\delta}{2},R]$, there holds
\begin{align}\label{bdb-1}
{g}^2(r){r}^{2(N-1)}U_{\epsilon}'^2(r)
\leq&{V}_{\epsilon}^2(R)\left(e^{-\frac{\widetilde{M}\left(r-\frac{\delta}{2}\right)}{\epsilon}}+e^{-\frac{\widetilde{M}(R-r)}{\epsilon}}\right)\notag\\[-0.7em]
&\\[-0.7em]
=&\frac{R^{2N}(A-B)^2}{N^2\epsilon^4}\left(e^{-\frac{\widetilde{M}\left(r-\frac{\delta}{2}\right)}{\epsilon}}+e^{-\frac{\widetilde{M}(R-r)}{\epsilon}}\right).\notag
\end{align}

\textbf{(b)} For $r\in[\frac{\sqrt{N-1}}{4\widetilde{M}}\epsilon,\delta]$, there holds
\begin{align}\label{bdb-2}
{g}^2(r){r}^{2(N-1)}U_{\epsilon}'^2(r)
\leq&{V}_{\epsilon}^2(\delta)\left(e^{-\frac{\widetilde{M}\left(r-\frac{\sqrt{N-1}}{4\widetilde{M}}\epsilon\right)}{\epsilon}}+e^{-\frac{\widetilde{M}(\delta-r)}{\epsilon}}\right).
\end{align}
Hence, for $r\in[\delta,R)$,
by (\ref{bdb-1}), we have
\begin{align}
|U_{\epsilon}'(r)|\leq&\frac{R^{N}|A-B|}{N\epsilon^2\delta^{N-1}\min_{[0,R]}g}\left(e^{-\frac{\widetilde{M}(R-r)}{2\epsilon}}+e^{-\frac{\widetilde{M}\delta}{4\epsilon}}\right)\notag\\
\leq&\frac{R^{N}|A-B|}{N\delta^{N-1}\min_{[0,R]}g}\left(\epsilon^{-2}e^{-\frac{\widetilde{M}(R-r)}{2\epsilon}}+e^{-\frac{\widetilde{M}\delta}{8\epsilon}}\right),\quad\quad\quad\quad\mathrm{as}\,\,0<\epsilon\ll1,\notag
\end{align}
which gives (\ref{0823-99-eoei}).

On the other hand, notice that (\ref{bdb-1}) implies ${V}_{\epsilon}^2(\delta)\leq\frac{R^{2N}(A-B)^2}{N^2\epsilon^4}\left(e^{-\frac{\widetilde{M}\delta}{2\epsilon}}+e^{-\frac{\widetilde{M}(R-\delta)}{\epsilon}}\right)$.
Along with (\ref{bdb-2}) yields that for ${r}\in[\frac{\sqrt{N-1}}{4\widetilde{M}}\epsilon,\delta)$, 
\begin{align}
\left|U_{\epsilon}'(r)\right|\leq&\frac{R^N|A-B|\left(e^{-\frac{\widetilde{M}\delta}{4\epsilon}}+e^{-\frac{\widetilde{M}(R-\delta)}{2\epsilon}}\right)}{N\left(\frac{\sqrt{N-1}}{4\widetilde{M}}\right)^{N-1}\epsilon^{N+1}\min_{[0,R]}g}\left(e^{-\frac{\widetilde{M}\left(r-\frac{\sqrt{N-1}}{4\widetilde{M}}\epsilon\right)}{2\epsilon}}+e^{-\frac{\widetilde{M}(\delta-r)}{2\epsilon}}\right)\notag\\
\leq&\frac{R^N|A-B|e^{-\frac{\widetilde{M}}{\epsilon}\min\left\{\frac{\delta}{8},\frac{R-\delta}{4}\right\}}}{N\left(\frac{\sqrt{N-1}}{4\widetilde{M}}\right)^{N-1}\min_{[0,R]}g},\quad\quad\quad\quad\quad\quad\quad\quad\quad\quad\mathrm{as}\,\,0<\epsilon\ll1.\notag
\end{align}
The last inequality holds trivially due to the fact $\frac{2}{\epsilon^{N+1}}\left(e^{-\frac{\widetilde{M}\delta}{4\epsilon}}+e^{-\frac{\widetilde{M}(R-\delta)}{2\epsilon}}\right)\leq{e}^{-\frac{\widetilde{M}}{\epsilon}\min\left\{\frac{\delta}{8},\frac{R-\delta}{4}\right\}}$ as $0<\epsilon\ll1$. (Note that this inequality is independent of the variable $r\in(0,R]$.)
Therefore, we get (\ref{0823-99}) and complete the proof of Lemma~\ref{lem-aug22-2}.
\end{proof}

\subsubsection{The boundary asymptotics with exact first two order terms}\label{sec-ccthmpf}
   In what follows we denote $\widetilde{C}_i$'s as positive constants independent of $\epsilon$.

Assume $0<A<B$.
We shall use (\ref{sec3id1}) and (\ref{sec3id2}) to deal with the exact leading order terms of ${U}_{\epsilon}(0)$ and ${U}_{\epsilon}(R)$.  We start with the estimates of $\Lambda_{1,\epsilon}(R)$ and $\Lambda_{2,\epsilon}(R)$ (see (\ref{lam1e})
and (\ref{lam2e})) for $0<\epsilon\ll1$ as follows.

\begin{lemma}\label{lem-new0830}
Let $\kappa\in(0,1)$ be independent of $\epsilon$. Then,
as $0<\epsilon\ll1$, there hold
\begin{align}
\int_{\epsilon^{\kappa}}^{R-\epsilon^{\kappa}}\frac{1}{s}{U}_{\epsilon}'^2(s)\dss\leq&\widetilde{C}_1e^{-\frac{\widetilde{M}}{2\epsilon^{1-\kappa}}},\label{0901-id1}\\
\int_{R-\epsilon^{\kappa}}^R{U}_{\epsilon}'^2(s)\dss\leq&\frac{\widetilde{C}_2}{\widetilde{M}\epsilon^{3}}.\label{0901-id2}
\end{align}
Hence, we have
\begin{align}\label{0827ch1}
\left|\Lambda_{i,\epsilon}(R)\right|\leq\frac{\widetilde{C}_3}{\epsilon},\quad{i=1,2}.
\end{align}
\end{lemma}
\begin{proof}
We shall use gradient estimates proposed in Lemma~\ref{lem-aug22-2}
to prove (\ref{0901-id1}) and (\ref{0901-id2}).
Let $\delta\in(0,R)$ be a fixed constant independent of $\epsilon$.
Due to $0<\kappa<1$, we have $0<\frac{\sqrt{N-1}}{4\widetilde{M}}\epsilon<\epsilon^{\kappa}<\delta<R-\epsilon^{\kappa}$ as $0<\epsilon\ll1$.
Thus, by (\ref{0823-99}) and (\ref{0823-99-eoei}), one may check that
\begin{align}\label{hence}
\int_{\epsilon^\kappa}^{R-\epsilon^\kappa}\frac{1}{s}{U}_{\epsilon}'^2(s)\dss
=&\int_{\epsilon^\kappa}^{\delta}\frac{1}{s}{U}_{\epsilon}'^2(s)\dss+\int_{\delta}^{R-\epsilon^\kappa}\frac{1}{s}{U}_{\epsilon}'^2(s)\dss\notag\\
\leq&C_1^2{e}^{-\frac{\widetilde{M}}{\epsilon}\min\left\{\frac{\delta}{4},\frac{R-\delta}{2}\right\}}\log\frac{\delta}{\epsilon^{\kappa}}+\frac{2C_2^2}{\delta}\left(\frac{1}{\widetilde{M}\epsilon^{3}}e^{-\frac{\widetilde{M}}{\epsilon^{1-\kappa}}}+Re^{-\frac{\widetilde{M}\delta}{4\epsilon}}\right)\\
\leq&\widetilde{C}_1e^{-\frac{\widetilde{M}}{2\epsilon^{1-\kappa}}},\quad\quad\quad\quad\quad\quad\quad\quad\quad\quad\quad\quad\quad\mathrm{as}\,\,0<\epsilon\ll1,\notag
\end{align}
and
\begin{align}
\int_{R-\epsilon^{\kappa}}^R{U}_{\epsilon}'^2(s)\dss\leq2C_2^2\left(\frac{1}{\widetilde{M}\epsilon^{3}}
+\epsilon^{\kappa}e^{-\frac{\widetilde{M}\delta}{4\epsilon}}\right)\leq\frac{\widetilde{C}_2}{\widetilde{M}\epsilon^{3}},\quad\mathrm{as}\,\,0<\epsilon\ll1,
\end{align}
which give (\ref{0901-id1}) and (\ref{0901-id2}).

By (\ref{780822-1-2016}), (\ref{lam1e}), (\ref{lam2e}), (\ref{0823-99}) and (\ref{0901-id2}), 
there exists a positive constant~$\widetilde{C}_3$ independent of $\epsilon$ such that
\begin{align}
\left|\Lambda_{1,\epsilon}(R)\right|\leq&
\,\frac{\epsilon^2}{2R}\left((N-2)\max_{[0,R]}g+R\max_{[0,R]}|g'|\right)\left\{\int_0^{\epsilon^{\kappa}}+\int_{\epsilon^{\kappa}}^{R-\epsilon^{\kappa}}+\int_{R-\epsilon^{\kappa}}^R\right\}U_{\epsilon}'^2(s)\dss\notag\\
\leq&\frac{\epsilon^2}{2R}\left((N-2)\max_{[0,R]}g+R\max_{[0,R]}|g'|\right)\label{0902-he}\\
&\,\times\left({C}_1^2\epsilon^{\kappa}{e}^{-\frac{\widetilde{M}}{\epsilon}\min\left\{\frac{\delta}{4},\frac{R-\delta}{2}\right\}}+R\widetilde{C}_1e^{-\frac{\widetilde{M}}{2\epsilon^{1-\kappa}}}+\frac{\widetilde{C}_2}{\widetilde{M}\epsilon^{3}}\right)\leq\frac{\widetilde{C}_3}{\epsilon},\notag
\end{align}
and
\begin{align}
\left|\Lambda_{2,\epsilon}(R)\right|\leq&\frac{\epsilon^2}{2}\left(\max_{[0,R]}g\right){C}_1^2{e}^{-\frac{\widetilde{M}}{\epsilon}\min\left\{\frac{\delta}{4},\frac{R-\delta}{2}\right\}}\notag\\
&+\frac{\epsilon^2}{2}\left(2(N-1)\max_{[0,R]}g+R\max_{[0,R]}|g'|\right)\left\{\int_{\epsilon^{\kappa}}^{R-\epsilon^{\kappa}}+\int_{R-\epsilon^{\kappa}}^R\right\}\frac{1}{s}U_{\epsilon}'^2(s)\dss\label{0902-hetwo}\\
\leq&\frac{\widetilde{C}_3}{\epsilon}.\notag
\end{align}
Therefore, we get (\ref{0827ch1}) and complete the proof of Lemma~\ref{lem-new0830}.
\end{proof}

By (\ref{0828-id1}), (\ref{sec3id1}), 
(\ref{holderrd1})--(\ref{holderrd2}), 
and (\ref{0827ch1}), we have
\begin{align}\label{min-maxu1}
e^{-qU_{\epsilon}(R)}\leq&\frac{qN}{BR^N}\left(\frac{R^2(A-B)^2}{2N^2\epsilon^2g(R)}+\frac{A}{p}+\frac{B}{q}+\frac{\widetilde{C}_3}{\epsilon}\right)\int_0^{R}s^{N-1}e^{-qU_{\epsilon}(s)}\dss\notag\\[-0.7em]
&\\[-0.7em]
\leq&\frac{\widetilde{C}_{4}}{\epsilon^2},\notag
\end{align}
and 
\begin{align}\label{min-maxu2}
e^{-qU_{\epsilon}(R)}=&\frac{qN}{BR^N}\Bigg(\frac{R^2(A-B)^2}{2N^2\epsilon^2g(R)}+\frac{A}{p}+\frac{B}{q}+\Lambda_{1,\epsilon}(R)\notag\\
&\quad\quad\quad\quad\quad-\frac{AR^Ne^{pU_{\epsilon}(R)}}{pN\int_0^{R}s^{N-1}e^{pU_{\epsilon}(s)}\dss}\Bigg)\int_0^{R}s^{N-1}e^{-qU_{\epsilon}(s)}\dss\\
\geq&\frac{q}{B}\left(\frac{R^2(A-B)^2}{2N^2\epsilon^2g(R)}+\frac{A}{p}+\frac{B}{q}-\frac{\widetilde{C}_3}{\epsilon}-\frac{AN}{p}O(1)\epsilon^{\frac{p}{q}}\right)\notag\\
\geq&\frac{\widetilde{C}_{5}}{\epsilon^2},\notag
\end{align}
as $0<\epsilon\ll1$.
Here we have used (\ref{0828-id1-0}) and (\ref{holderrd1}) to get
$\frac{AR^Ne^{pU_{\epsilon}(R)}}{pN\int_0^{R}s^{N-1}e^{pU_{\epsilon}(s)}\dss}\leq\frac{AN}{p}O(1)\epsilon^{\frac{p}{q}}$. Along with (\ref{0828-id2}) and (\ref{0827ch1}),
we verify (\ref{min-maxu1}) and (\ref{min-maxu2}).
Consequently, by (\ref{min-maxu1}) and (\ref{min-maxu2}), there holds
\begin{align}\label{leadingur}
U_{\epsilon}(R)=-\frac{2}{q}\log\frac{1}{\epsilon}+O(1),\quad\mathrm{as}\,\,0<\epsilon\ll1.
\end{align}

Now we deal with the exact leading order term of $U_{\epsilon}(0)$.
Notice ${u}_{\epsilon}(x)\equiv{U}_{\epsilon}(|x|)$.
In (\ref{thm1-id5}), 
we may set $K=\{x\in\mathbb{R}^N: |x|\leq\delta\}$ and $\Omega=\{x\in\mathbb{R}^N: |x|<R\}$, 
where $\delta\in(0,R)$ is independent of $\epsilon$.
Then by (\ref{780822-1-2016}),
we have $\displaystyle\max_{x,y\in{K}}|{u}_{\epsilon}(x)-{u}_{\epsilon}(y)|=|{U}_{\epsilon}(\delta)-{U}_{\epsilon}(0)|$
and $\displaystyle\max_{x,y\in{\overline{\Omega}}}|{u}_{\epsilon}(x)-{u}_{\epsilon}(y)|=|{U}_{\epsilon}(R)-{U}_{\epsilon}(0)|$. Consequently, by (\ref{thm1-id5}), (\ref{0828-id1}) and (\ref{leadingur}) one finds
\begin{align}\label{888and9}
|U_{\epsilon}(\delta)-U_{\epsilon}(0)|\leq{C}_{\delta}e^{-\frac{M_{\delta}}{\epsilon}},\quad
\mathrm{as}\,\,0<\epsilon\ll1,
\end{align}
where
$C_{\delta}$ and $M_{\delta}$ are positive constants
independent of $\epsilon$. On the other hand, integrating 
(\ref{0823-99-eoei}) over the interval $[\delta,R-\epsilon^{\kappa}]$ gives
 \begin{align}\label{0830-1}
|U_{\epsilon}(R-\epsilon^{\kappa})-U_{\epsilon}(\delta)|\leq{C}_2\left(\frac{2}{\widetilde{M}\epsilon}e^{-\frac{\widetilde{M}}{2\epsilon^{1-\kappa}}}+e^{-\frac{\widetilde{M}\delta}{8\epsilon}}(R-\delta)\right),\quad
\mathrm{as}\,\,0<\epsilon\ll1,
\end{align}
where $\kappa\in(0,1)$ is independent of $\epsilon$.
By (\ref{780822-1-2016}),  (\ref{888and9}) and (\ref{0830-1}), we conclude
\begin{align}\label{0830-cchaha}
\max_{r\in[0,R-\epsilon^{\kappa}]}|U_{\epsilon}(r)-U_{\epsilon}(0)|\leq\widetilde{C}_{6}e^{-\frac{\widetilde{M}}{4\epsilon^{1-\kappa}}},\quad
\mathrm{as}\,\,0<\epsilon\ll1.
\end{align}

Note that $\int_0^{R}U_{\epsilon}(r)r^{N-1}\dr=0$. Hence by (\ref{0828-id1}), (\ref{leadingur}) and (\ref{0830-cchaha}), we have
\begin{align}
0<U_{\epsilon}(0)=&\frac{N}{R^N}\int_0^R\left(U_{\epsilon}(0)-U_{\epsilon}(r)\right)r^{N-1}\dr\notag\\
=&\frac{N}{R^N}\left\{\int_0^{R-\epsilon^{\kappa}}+\int_{R-\epsilon^{\kappa}}^R\right\}\left(U_{\epsilon}(0)-U_{\epsilon}(r)\right)r^{N-1}\dr\\
\leq&\widetilde{C}_{7}e^{-\frac{\widetilde{M}}{4\epsilon^{1-\kappa}}}
+\frac{R^N-(R-\epsilon^{\kappa})^N}{R^N}\left(\frac{2}{q}\log\frac{1}{\epsilon}+O(1)\right).\notag
\end{align}
 Along with~(\ref{0830-cchaha}) gives
\begin{align}\label{0909-1}
\max_{[0,R-\epsilon^{\kappa}]}\left|U_{\epsilon}\right|\leq\widetilde{C}_{8}\epsilon^{\kappa}\log\frac{1}{\epsilon},\quad\mathrm{as}\,\,0<\epsilon\ll1.
\end{align}

\textbf{The exact second order term of} $\boldsymbol{U_{\epsilon}(R)}$. 
Due to (\ref{leadingur}), we let
\begin{align}\label{0909-2}
\xi(\epsilon)=U_{\epsilon}(R)+\frac{2}{q}\log\frac{1}{\epsilon}.
\end{align}
Then 
\begin{align}\label{0909-3}
\sup_{0<\epsilon\ll1}|\xi(\epsilon)|<\infty.
\end{align}
To get the exact second order term of $U_{\epsilon}(R)$, 
we shall deal with the leading order term of $\xi(\epsilon)$.
Firstly, by (\ref{sec3id1}) and (\ref{0909-2}), we have
\begin{align}\label{0909-4}
&\frac{R^N}{N}\left(\frac{A\epsilon^{2+\frac{2p}{q}}e^{p\xi(\epsilon)}}{p\int_0^{R}s^{N-1}e^{pU_{\epsilon}(s)}\dss}
+\frac{Be^{-q\xi(\epsilon)}}
{q\int_0^{R}s^{N-1}e^{-qU_{\epsilon}(s)}\dss}\right)\notag\\[-0.7em]
&\\[-0.7em]
=&\frac{R^2(A-B)^2}{2N^2g(R)}+\epsilon^2\left(\frac{A}{p}+\frac{B}{q}+\Lambda_{1,\epsilon}(R)\right).\notag
\end{align}
Combining (\ref{holderrd1}), (\ref{holderrd2}), (\ref{0827ch1}), (\ref{0909-3}) and (\ref{0909-4}),
one may check that
\begin{align}\label{0909-5}
\left|e^{-q\xi(\epsilon)}-\frac{q(A-B)^2}{2BNg(R)R^{N-2}}\int_0^{R}s^{N-1}e^{-qU_{\epsilon}(s)}\dss\right|\leq\widetilde{C}_{9}\epsilon.
\end{align}
We need the following lemma:
\begin{lemma}\label{lem-0910-new}
 As $0<\epsilon\ll1$,
we have
\begin{align}\label{0909-6}
\left|\int_0^{R}s^{N-1}e^{pU_{\epsilon}(s)}\dss-\frac{R^N}{N}\right|+\left|\int_0^{R}s^{N-1}e^{-qU_{\epsilon}(s)}ds-\frac{BR^N}{AN}\right|\leq\widetilde{C}_{10}\epsilon^{\kappa}\log\frac{1}{\epsilon},
\end{align}
for any $\kappa\in(0,1)$,
where $\widetilde{C}_{10}>0$ denotes a generic constant independent of $\epsilon$.
\end{lemma}
\begin{proof}
Note that $\left|\frac{e^{pt}-1}{t}\right|\leq\frac{e^{p|t|}-1}{|t|}$ 
and $\frac{e^{p|t|}-1}{|t|}$ is increasing to $|t|$ for $t\neq0$. By virtue of
(\ref{0828-id1}), (\ref{0830-cchaha}) and (\ref{0909-1}), we get the inequality
\begin{align}\label{chiun0102}
|e^{pU_{\epsilon}(s)}-1|\leq&\frac{\max_{[0,R-\epsilon^{\kappa}]}\left|U_{\epsilon}\right|}{U_{\epsilon}(0)+\widetilde{C}_{6}e^{-\frac{\widetilde{M}}{4\epsilon^{1-\kappa}}}}\left(e^{{p}\left(U_{\epsilon}(0)+\widetilde{C}_{6}e^{-\frac{\widetilde{M}}{4\epsilon^{1-\kappa}}}\right)}-1\right)\notag\\[-0.7em]
&\\[-0.7em]
\leq&\frac{\widetilde{C}_{8}q}{\log\frac{B}{A}}\left[\left(\frac{B}{A}\right)^{\frac{p}{q}}-1\right]\epsilon^{\kappa}\log\frac{1}{\epsilon},\quad\mathrm{for}\,\,s\in[0,R-\epsilon^{\kappa}].\notag
\end{align}
Along with (\ref{0828-id1}), 
one verifies 
\begin{align}\label{0910-1}
\left|\int_0^{R}s^{N-1}e^{pU_{\epsilon}(s)}\dss-\frac{R^N}{N}\right|\leq&\int_0^{R}s^{N-1}\left|e^{pU_{\epsilon}(s)}-1\right|\dss\notag\\
=&\left\{\int_0^{R-\epsilon^\kappa}+\int_{R-\epsilon^\kappa}^R\right\}s^{N-1}\left|e^{pU_{\epsilon}(s)}-1\right|\dss\notag\\[-0.7em]
&\\[-0.7em]
\leq&\frac{\widetilde{C}_{8}R^Nq}{N\log\frac{B}{A}}\left[\left(\frac{B}{A}\right)^{\frac{p}{q}}-1\right]\epsilon^{\kappa}\log\frac{1}{\epsilon}\notag\\
&\quad\quad\quad\quad\quad+\frac{R^N-(R-\epsilon^\kappa)^N}{N}
\left[\left(\frac{B}{A}\right)^{\frac{p}{q}}+1\right]\notag\\
\leq&\widetilde{C}_{10}\epsilon^{\kappa}\log\frac{1}{\epsilon},\quad\quad\quad\quad\quad\quad\quad\quad\quad\mathrm{for}\,\,0<\kappa<1\,\,\mathrm{and}\,\,0<\epsilon\ll1.\notag
\end{align}

Multiplying (\ref{bigu1}) by $r^{N-1}$,
integrating the expansion over the interval $(0,R-\epsilon^{\kappa})$
and using the boundary condition~(\ref{bigu3}),
we have
\begin{align}\label{bigu-0910-1}
&\epsilon^2g(R-\epsilon^\kappa)(R-\epsilon^\kappa)^{N-1}U_{\epsilon}'(R-\epsilon^\kappa)\notag\\[-0.7em]
&\\[-0.7em]
&\quad\quad\quad=\frac{R^N}{N}\left(\frac{A\int_{0}^{R-\epsilon^\kappa}r^{N-1}e^{pU_{\epsilon}(r)}\dr}{\int_0^{R}s^{N-1}e^{pU_{\epsilon}(s)}\dss}-\frac{B\int_{0}^{R-\epsilon^\kappa}r^{N-1}e^{-qU_{\epsilon}(r)}\dr}{\int_0^{R}s^{N-1}e^{-qU_{\epsilon}(s)}\dss}\right).\notag
\end{align}
By (\ref{0823-99-eoei}), we can deal with the left-hand side of (\ref{bigu-0910-1}):
\begin{align}\label{1010-2016-553pm}
\epsilon^2g(R-\epsilon^\kappa)(R-\epsilon^\kappa)^{N-1}\left|U_{\epsilon}'(R-\epsilon^\kappa)\right|\leq2\left(\max_{[0,R]}g\right)R^{N-1}C_2e^{-\frac{\widetilde{M}}{2\epsilon^{1-\kappa}}}.
\end{align}
For the right-hand side of (\ref{bigu-0910-1}), 
 we may use (\ref{0909-1}) and follow the same argument as (\ref{0910-1}) to get
\begin{align}\label{0910-2}
\left|\int_{0}^{R-\epsilon^\kappa}r^{N-1}e^{pU_{\epsilon}(r)}\dr-\frac{R^N}{N}\right|+\left|\int_{0}^{R-\epsilon^\kappa}r^{N-1}e^{-qU_{\epsilon}(r)}\dr-\frac{R^N}{N}\right|\leq\widetilde{C}_{10}\epsilon^{\kappa}\log\frac{1}{\epsilon}.
\end{align} 
Combining (\ref{0910-1})--(\ref{0910-2}) and passing through simple calculations, 
we immediately get 
\begin{align}\label{0909-6-1010}
\left|\int_0^{R}s^{N-1}e^{-qU_{\epsilon}(s)}\dss-\frac{BR^N}{AN}\right|\leq\widetilde{C}_{10}\epsilon^{\kappa}\log\frac{1}{\epsilon}.
\end{align}
Therefore, by (\ref{0910-1}) and (\ref{0909-6-1010}) we
get (\ref{0909-6}) and complete the proof of Lemma~\ref{lem-0910-new}.
\end{proof}

By (\ref{0909-5}) and (\ref{0909-6}), it follows
\begin{align*}
\xi(\epsilon)=\frac{1}{q}\log\frac{2AN^2g(R)}{qR^2(A-B)^2}+o_{\epsilon}(1).
\end{align*}
Along with (\ref{leadingur}), we conclude
\begin{align}\label{0916-thm2-1}
U_{\epsilon}(R)=-\frac{2}{q}\log\frac{1}{\epsilon}+\frac{1}{q}\log\frac{2AN^2g(R)}{qR^2(A-B)^2}+o_{\epsilon}(1),
\end{align}
as $0<\epsilon\ll1$.

We stress that the second order term of (\ref{0916-thm2-1}) is a bounded quantity with respect to small $\epsilon$
and includes the surface dielectric constant $g(R)$, the concentration difference $|A-B|$ between cations and anions
and the curvature $\frac{1}{R}$.

\section{Concentration phenomenon described by Dirac delta functions}\label{sec-bcp}
For the case of $A\neq{B}$, we show that
 the net charge density $\rho_\epsilon$  (defined in (\ref{ancat-0923-pipipi}); see also (\ref{chdensity-1204}))
and  $\left(\epsilon\,U_\epsilon'\right)^2$  (related to the electric energy)
 have concentration phenomena as $\epsilon\downarrow0$,
which can be described by Dirac delta functions concentrated 
at the boundary point $r=R$ (see Definition~\ref{def1}). Such results are stated as follows:

\begin{theorem}\label{thm3} 
Assume that $A,\,B,\,p$, and $q$ are positive constants independent of $\epsilon$ and satisfy $A\neq{B}$.
For $\epsilon>0$, 
let $U_\epsilon\in{C}^1([0,R])\cap{C}^{\infty}((0,R))$
be the unique solution of problem~(N*). Then
 $\rho_{\epsilon}$
and $\left(\epsilon\,U_\epsilon'\right)^2$
have boundary concentration phenomena
exhibiting in the limit
Dirac delta functions supported on the boundary with suitable weights,
 which can be depicted as 
\begin{align}\label{1123-2016-1113pm}
\rho_{\epsilon}:&=\,-\frac{R^N}{N}\left(\frac{Ae^{pU_{\epsilon}}}{\int_0^{R}s^{N-1}e^{pU_{\epsilon}(s)}\dss}-\frac{Be^{-qU_{\epsilon}}}{\int_0^{R}s^{N-1}e^{-qU_{\epsilon}(s)}\dss}\right)\notag\\[-0.7em]
&\\[-0.7em]
&\rightharpoonup\frac{R(B-A)}{N}\delta_R,\quad\,weakly\,\,in\,\,C([0,R];\mathbb{R}),\notag
\end{align}
and
\begin{itemize}
\item[(I)] When $0<A<B$, there hold 
\begin{align}
\left(\epsilon\,U_\epsilon'\right)^2\rightharpoonup&\frac{2R(B-A)}{qNg(R)}\delta_R,\quad\,weakly\,\,in\,\,C([0,R];\mathbb{R}),\label{add1111-1}\\
e^{-qU_{\epsilon}}-1\rightharpoonup&\frac{R(B-A)}{AN}\delta_R,\quad\,\,\,weakly\,\,in\,\,C([0,R];\mathbb{R}),\label{thm2-0911-3}\\
e^{pU_{\epsilon}}-1\rightarrow&\,0,\quad\quad\quad\quad\quad\quad\,\,strongly\,\,in\,\,L^{\theta}((0,R)),\,\,\forall\,\theta>0.\label{add1111-2}
\end{align}
\item[(II)] When $A>B>0$, there hold 
\begin{align}
\left(\epsilon\,U_\epsilon'\right)^2\rightharpoonup&\frac{2R(A-B)}{pNg(R)}\delta_R,\quad\,weakly\,\,in\,\,C([0,R];\mathbb{R}),\label{add1111-3}\\
e^{pU_{\epsilon}}-1\rightharpoonup&\frac{R(A-B)}{BN}\delta_R,\quad\,\,\,weakly\,\,in\,\,C([0,R];\mathbb{R}),\label{add0103-2017}\\
e^{-qU_{\epsilon}}-1\rightarrow&\,0,\quad\quad\quad\quad\quad\quad\,\,strongly\,\,in\,\,L^{\theta}((0,R)),\,\,\forall\,\theta>0.\label{add1111-4}
\end{align}
\end{itemize}

\begin{itemize}
\item[(III)] For $\widetilde{\theta}\in(0,2)$, $\epsilon\,U_\epsilon'\rightarrow\,0$ 
strongly in $L^{\widetilde{\theta}}((0,R))$ as $\epsilon\downarrow0$,
which is in extreme contrast with the behavior of
 $\left(\epsilon\,U_\epsilon'\right)^2$ presented in (\ref{add1111-1}) and (\ref{add1111-3}).
Moreover, as $0<\epsilon\ll1$, there holds
\begin{align}\label{add1234-5-1119}
||\epsilon\,U_\epsilon'||_{L^{\widetilde{\theta}}((0,R))}\leq\,C_{\widetilde{\theta}}\epsilon^{\min\left\{1,\frac{2}{\widetilde{\theta}}-1\right\}}\left(\log\frac{1}{\epsilon}\right)^{\chi(\widetilde{\theta})},
\end{align}
where $C_{\widetilde{\theta}}$ is a positive constant independent of $\epsilon$,
and $\chi(\widetilde{\theta})=1$ if $\widetilde{\theta}\in(0,1]$;
$\chi(\widetilde{\theta})=0$ if $\widetilde{\theta}\in(1,2)$.
\end{itemize}
\end{theorem}

\begin{remark}
\ \ \ \
\begin{itemize}
\item[(a)] Theorem~\ref{thm3} presents that the asymptotic behaviors of $e^{-qU_{\epsilon}(r)}-1$ and $e^{pU_{\epsilon}(r)}-1$
near boundary points
are totally different.
On the other hand, we stress that 
for any $\widetilde{R}\in(0,R)$ independent of $\epsilon$,
$\left(\epsilon\,U_\epsilon'\right)^2$, $\rho_\epsilon$,
$e^{pU_{\epsilon}}-1$ and
$e^{-qU_{\epsilon}}-1$
tend to zero uniformly in $[0,\widetilde{R}]$
as $\epsilon$ goes to zero (by Lemma~\ref{lem-aug22-2} and (\ref{0909-1})).

\item[(b)] It is worth mentioning that the concentration difference $|A-B|$
plays crucial role in the boundary concentration phenomenon of $U_{\epsilon}$.
\end{itemize}
\end{remark}

The proof of Theorem~\ref{thm3} is stated as follows.
\begin{proof}[Proof of Theorem~\ref{thm3}]
 Assume $0<A<B$. Multiplying (\ref{bigoneadd-as}) by $\epsilon^2$
and using (\ref{holderrd1}), we rewrite the expansion as 
\begin{align}\label{1010-2016-330pm}
\epsilon^2\sqrt{g(t)}U_{\epsilon}'(t)=-\Bigg[&\frac{2R^N\epsilon^2}{N}\left(\frac{A\left(e^{pU_{\epsilon}(t)}-e^{pU_{\epsilon}(\epsilon^\kappa)}\right)}{p\int_0^{R}s^{N-1}e^{pU_{\epsilon}(s)}\dss}
-\frac{Be^{-qU_{\epsilon}(\epsilon^\kappa)}}
{q\int_0^{R}s^{N-1}e^{-qU_{\epsilon}(s)}\dss}\right)\notag\\
&+\epsilon^4g(\epsilon^\kappa)U_{\epsilon}'^2(\epsilon^\kappa)-\epsilon^4\int_{\epsilon^\kappa}^t\left(\frac{2(N-1)g(r)}{r}+g'(r)\right)U_{\epsilon}'^2(r)\dr\\
&+\frac{2R^NB\epsilon^2e^{-qU_{\epsilon}(t)}}
{qN\int_0^{R}s^{N-1}e^{-qU_{\epsilon}(s)}\dss}\Bigg]^{\frac{1}{2}},\notag
\end{align}
for $t\in(0,R]$.
Moreover, by (\ref{780822-1-2016}), (\ref{holderrd2}), (\ref{0823-99}), (\ref{0823-99-eoei}), (\ref{0901-id1}), (\ref{0901-id2}), (\ref{leadingur}), (\ref{0909-1}) and (\ref{0909-6}), we arrive at 
\begin{align}\label{2016-1010-guochin}
\epsilon^2U_{\epsilon}'(t)=-\sqrt{\frac{2A}{qg(t)}}e^{-\frac{q}{2}\left(U_{\epsilon}(t)-\frac{2}{q}\log\epsilon\right)}+\lambda_{\epsilon,\kappa}(t),\quad\mathrm{for}\,\,t\in[\epsilon^\kappa,R],
\end{align}
and 
\begin{align}\label{etae-t}
\max_{[\epsilon^\kappa,R]}|\lambda_{\epsilon,\kappa}(t)|\leq\widetilde{C}_{12}\epsilon^{\kappa}\log\frac{1}{\epsilon},\quad\quad\mathrm{as}\,\,0<\epsilon\ll1,
\end{align} 
due to the following estimates (a)--(c):
\begin{itemize}
\item[\textbf{(a).}] By (\ref{holderrd1}), (\ref{holderrd2}), (\ref{780822-1-2016}), (\ref{leadingur}) and (\ref{0909-1}), for $t\in[\epsilon^{\kappa},R]$
we have 
\begin{align}\label{1010-2016-2pm}
&\epsilon^2\left|\frac{A\left(e^{pU_{\epsilon}(t)}-e^{pU_{\epsilon}(\epsilon^\kappa)}\right)}{p\int_0^{R}s^{N-1}e^{pU_{\epsilon}(s)}\dss}
-\frac{Be^{-qU_{\epsilon}(\epsilon^\kappa)}}
{q\int_0^{R}s^{N-1}e^{-qU_{\epsilon}(s)}\dss}\right|\notag\\
&\quad\quad\quad\quad\quad\leq
\epsilon^2\left(\frac{Ae^{pU_{\epsilon}(\epsilon^\kappa)}}{p\int_0^{R}s^{N-1}e^{pU_{\epsilon}(s)}\dss}
+\frac{Be^{-qU_{\epsilon}(\epsilon^\kappa)}}
{q\int_0^{R}s^{N-1}e^{-qU_{\epsilon}(s)}\dss}\right)\\
&\quad\quad\quad\quad\quad\leq\frac{N}{R^N}\left(\frac{A}{p}+\frac{B}{q}\right)
\epsilon^{2-\widetilde{C}_8\max\{p,q\}\epsilon^{\kappa}},\quad\quad\quad\quad\quad\,\,\,\mathrm{as}\,\,0<\epsilon\ll1.\notag
\end{align}
\item[\textbf{(b).}] By (\ref{0823-99}), (\ref{0823-99-eoei}), (\ref{0901-id1}) and (\ref{0901-id2}),
we have
\begin{align}
\epsilon^4g(\epsilon^\kappa)U_{\epsilon}'^2(\epsilon^\kappa)\leq&\epsilon^4C_1^2\max_{[0,R]}g,\notag\\
\epsilon^4\left|\int_{\epsilon^\kappa}^t\left(\frac{2(N-1)g(r)}{r}+g'(r)\right)U_{\epsilon}'^2(r)\dr\right|\leq&\epsilon\frac{\widetilde{C}_2}{\widetilde{M}}\max_{[0,R]}|g'|,\quad\mathrm{as}\,\,0<\epsilon\ll1.\label{1010-2016-243pm}
\end{align}
\item[\textbf{(c).}] By (\ref{780822-1-2016}) and (\ref{leadingur}),we have
\begin{align}\label{1010-2016-1}
\epsilon^2e^{-qU_{\epsilon}(t)}={e}^{-q\left(U_{\epsilon}(t)-\frac{2}{q}\log\epsilon\right)}\leq{e}^{-q\left(U_{\epsilon}(R)-\frac{2}{q}\log\epsilon\right)}=O(1),\,\,\quad\mathrm{as}\,\,0<\epsilon\ll1.
\end{align}
Along with (\ref{0909-6}) immediately implies
\begin{align}\label{1010-2016-2}
&\frac{2R^NB\epsilon^2e^{-qU_{\epsilon}(t)}}
{qN\int_0^{R}s^{N-1}e^{-qU_{\epsilon}(s)}\dss}\notag\\
&\quad\quad\quad\quad\quad=
\frac{2A}{q\left(1+O(1)\epsilon^{\kappa}\log\frac{1}{\epsilon}\right)}e^{-q\left(U_{\epsilon}(t)-\frac{2}{q}\log\epsilon\right)}\\
&\quad\quad\quad\quad\quad=
\frac{2A}{q}e^{-q\left(U_{\epsilon}(t)-\frac{2}{q}\log\epsilon\right)}+O(1)\epsilon^{\kappa}\log\frac{1}{\epsilon},\quad\quad\quad\quad\,\,\,\mathrm{as}\,\,0<\epsilon\ll1.\notag
\end{align}
\end{itemize}

For the convenience of the statement in proof,
we divide the proof of Theorem~\ref{thm3} into three steps.
Firstly, we state
the proof of (\ref{1123-2016-1113pm}), (\ref{thm2-0911-3}), (\ref{add1111-2}), (\ref{add0103-2017}) and
 (\ref{add1111-4}). Next, we prove
 (\ref{add1111-1}) and (\ref{add1111-3}).
 Finally, we give the proof of Theorem~\ref{thm3}(III).\\


\textbf{Step~1.~Proof of (\ref{1123-2016-1113pm}), (\ref{thm2-0911-3}), (\ref{add1111-2}), (\ref{add0103-2017}) and
 (\ref{add1111-4}).}
We deal with (\ref{thm2-0911-3}) as follows. Assume firstly $h\in{C}^1([0,R];\mathbb{R})$. 
Then, for $\kappa\in(0,1)$,
multiplying $h(r)$ on both sides of (\ref{bigu1})
and integrating the expression over $(\epsilon^{\kappa},R)$, one obtains
\begin{align}\label{bigu-0916}
&\epsilon^2\int_{\epsilon^{\kappa}}^Rh(r)g(r)\left[U_{\epsilon}''(r)+\left(\frac{N-1}{r}
+\frac{g'(r)}{g(r)}\right)U_{\epsilon}'(r)\right]\dr&\notag\\
&\quad\quad\quad=\frac{R^N}{N}\left(\frac{A\int_{\epsilon^{\kappa}}^Rh(r)\left(e^{pU_{\epsilon}(r)}-1\right)\dr}{\int_0^{R}s^{N-1}e^{pU_{\epsilon}(s)}\dss}-\frac{B\int_{\epsilon^{\kappa}}^Rh(r)\left(e^{-qU_{\epsilon}(r)}-1\right)\dr}{\int_0^{R}s^{N-1}e^{-qU_{\epsilon}(s)}\dss}\right)\\
&\quad\quad\quad\quad\quad+\frac{R^N}{N}\left(\frac{A}{\int_0^{R}s^{N-1}e^{pU_{\epsilon}(s)}\dss}-\frac{B}{\int_0^{R}s^{N-1}e^{-qU_{\epsilon}(s)}\dss}\right)\int_{\epsilon^{\kappa}}^Rh(r)\dr.\notag
\end{align}
We shall deal with each term of (\ref{bigu-0916}) for $0<\epsilon\ll1$.
Using integration by parts, we have
\begin{align}\label{id-0916-2}
&\epsilon^2\int_{\epsilon^{\kappa}}^Rh(r)g(r)\left[U_{\epsilon}''(r)+\left(\frac{N-1}{r}
+\frac{g'(r)}{g(r)}\right)U_{\epsilon}'(r)\right]\dr\notag\\
&\quad\quad\quad=\epsilon^2h(R)g(R)U_{\epsilon}'(R)
-\epsilon^2h(\epsilon^{\kappa})g(\epsilon^{\kappa})U_{\epsilon}'(\epsilon^{\kappa})\\
&\quad\quad\quad\quad\quad+\epsilon^2\int_{\epsilon^{\kappa}}^R\left[\frac{N-1}{r}-h'(r)g(r)+\left(\frac{1}{g(r)}+h(r)\right)g'(r)\right]U_{\epsilon}'(r)\dr.\notag
\end{align}
By (\ref{0823-99}) and (\ref{0823-99-eoei}), we immediately get
\begin{align}\label{0916-3}
\epsilon^2|h(\epsilon^{\kappa})g(\epsilon^{\kappa})U_{\epsilon}'(\epsilon^{\kappa})|\leq\epsilon^2{C}_1\left(\max_{[0,R]}|hg|\right){e}^{-\frac{\widetilde{M}}{\epsilon}\min\left\{\frac{\delta}{8},\frac{R-\delta}{4}\right\}}.
\end{align}
On the other hand, (\ref{780822-1-2016}), (\ref{0909-1}) and (\ref{0916-thm2-1}) give
\begin{align}\label{0916-4}
&\epsilon^2\left|\int_{\epsilon^{\kappa}}^R\left[\frac{N-1}{r}-h'(r)g(r)+\left(\frac{1}{g(r)}+h(r)\right)g'(r)\right]U_{\epsilon}'(r)\dr\right|\notag\\
&\quad\quad\quad\leq-\epsilon^2\left[(N-1)\epsilon^{-\kappa}+\max_{[0,R]}\left(|h'g|+\left|\frac{g'}{g}\right|+|hg'|\right)\right]\int_{\epsilon^{\kappa}}^RU_{\epsilon}'(r)\dr\\
&\quad\quad\quad\leq\widetilde{C}_{11}\epsilon^{2-\kappa}\log\frac{1}{\epsilon},\notag
\end{align}
as $0<\epsilon\ll1$. By (\ref{bigu3}) and (\ref{id-0916-2})--(\ref{0916-4}), we conclude that
\begin{align}\label{k-0916}
\lim_{\epsilon\downarrow0}\epsilon^2\int_{\epsilon^{\kappa}}^Rh(r)g(r)\left[U_{\epsilon}''(r)+\left(\frac{N-1}{r}
+\frac{g'(r)}{g(r)}\right)U_{\epsilon}'(r)\right]\dr=\frac{R}{N}(A-B)h(R).
\end{align}
Hence, we obtain the limit of the left-hand side of (\ref{bigu-0916}) as $\epsilon\downarrow0$.

Next, we deal with the right-hand side of (\ref{bigu-0916}). By (\ref{0909-6}) and (\ref{notagad-d}), one immeciately finds
\begin{align}
\lim_{\epsilon\downarrow0}\frac{A\int_{\epsilon^{\kappa}}^Rh(r)\left(e^{pU_{\epsilon}(r)}-1\right)\dr}{\int_0^{R}s^{N-1}e^{pU_{\epsilon}(s)}\dss}=0,\quad&\label{k0916-1}\\
\lim_{\epsilon\downarrow0}\left(\frac{A}{\int_0^{R}s^{N-1}e^{pU_{\epsilon}(s)}\dss}-\frac{B}{\int_0^{R}s^{N-1}e^{-qU_{\epsilon}(s)}\dss}\right)&\int_{\epsilon^{\kappa}}^Rh(r)\dr=0.\label{k0916-2}
\end{align}
As a consequence, by (\ref{0909-6}), (\ref{bigu-0916}) and (\ref{k-0916})--(\ref{k0916-2}), we arrive at
\begin{align}\label{0916ko}
\lim_{\epsilon\downarrow0}\int_{\epsilon^{\kappa}}^Rh(r)\left(e^{-qU_{\epsilon}(r)}-1\right)\dr=\frac{R}{AN}(B-A)h(R).
\end{align}
Along with the fact $\int_0^{\epsilon^{\kappa}}h(r)\left(e^{-qU_{\epsilon}(r)}-1\right)\dr\to0$ 
as $\epsilon\downarrow0$ (by (\ref{0909-1})), we get (\ref{thm2-0911-3}).


Due to the concept of the previous proof for the case $h\in{C}^1([0,R];\mathbb{R})$,
we now deal with (\ref{thm2-0911-3}) for $h\in{C}([0,R];\mathbb{R})$
by estimating $\int_{0}^Rh(r)\left(e^{-qU_{\epsilon}(r)}-1\right)\dr-\frac{R}{AN}(B-A)h(R)$
 (for $0<\epsilon\ll1$)  directly  via (\ref{0828-id1}), (\ref{0909-1}), (\ref{0909-6}), 
(\ref{0910-2}) and (\ref{0917-hae}). 
Note that (\ref{0909-6}) and (\ref{0910-2}) imply
\begin{align}\label{0917-hae}
\lim_{\epsilon\downarrow0}\int_{R-\epsilon^\kappa}^Rr^{N-1}e^{-qU_{\epsilon}(r)}\dr=\frac{R^N}{AN}(B-A).
\end{align}
Using the expression
\begin{align}\label{918wuwu}
&\int_{0}^Rh(r)\left(e^{-qU_{\epsilon}(r)}-1\right)\dr-\frac{R}{AN}(B-A)h(R)\notag\\
&\quad\quad\quad=\int_{0}^{R-\epsilon^{\kappa}}h(r)\left(e^{-qU_{\epsilon}(r)}-1\right)\dr
-\int_{R-\epsilon^{\kappa}}^Rh(r)\dr\notag\\[-0.7em]
&\\[-0.7em]
&\quad\quad\quad\quad+\int_{R-\epsilon^{\kappa}}^R\left(\frac{h(r)}{r^{N-1}}-\frac{h(R)}{R^{N-1}}\right)r^{N-1}e^{-qU_{\epsilon}(r)}\dr\notag\\
&\quad\quad\quad\quad\quad+\frac{h(R)}{R^{N-1}}\left(\int_{R-\epsilon^\kappa}^Rr^{N-1}e^{-qU_{\epsilon}(r)}\dr-\frac{R^N}{AN}(B-A)\right),\notag
\end{align}
one may check by virtue of (\ref{0828-id1}), (\ref{0909-1}) and (\ref{0917-hae}) that
\begin{align}
&\left|\int_{0}^Rh(r)\left(e^{-qU_{\epsilon}(r)}-1\right)\dr-\frac{R}{AN}(B-A)h(R)\right|\notag\\
&\quad\quad\quad\leq\int_{0}^{R-\epsilon^{\kappa}}\left|h(r)\right|\left|e^{-qU_{\epsilon}(r)}-1\right|\dr
+\int_{R-\epsilon^{\kappa}}^R\left|h(r)\right|\dr\notag\\
&\quad\quad\quad\quad+\int_{R-\epsilon^{\kappa}}^R\left|\frac{h(r)}{r^{N-1}}-\frac{h(R)}{R^{N-1}}\right|r^{N-1}e^{-qU_{\epsilon}(r)}\dr\notag\\[-0.7em]
&\\[-0.7em]
&\quad\quad\quad\quad\quad+\frac{\left|h(R)\right|}{R^{N-1}}\left|\int_{R-\epsilon^\kappa}^Rr^{N-1}e^{-qU_{\epsilon}(r)}\dr-\frac{R^N}{AN}(B-A)\right|\notag\\
&\quad\quad\quad\leq\max_{[0,R]}|h|\left[R\left(e^{q\widetilde{C}_{8}\epsilon^{\kappa}\log\frac{1}{\epsilon}}-1\right)+\epsilon^{\kappa}\right]\notag\\
&\quad\quad\quad\quad\quad\quad\quad+\max_{[R-\epsilon^\kappa,R]}\left|\frac{h(r)}{r^{N-1}}-\frac{h(R)}{R^{N-1}}\right|\left(\frac{R^N}{AN}(B-A)
+o_{\epsilon}(1)\right)+o_{\epsilon}(1)\longrightarrow0\quad\mathrm{as}\,\,\epsilon\downarrow0.\notag
\end{align}
Here we have verified $\displaystyle\lim_{\epsilon\downarrow0}\max_{[R-\epsilon^\kappa,R]}$~$\left|\frac{h(r)}{r^{N-1}}-\frac{h(R)}{R^{N-1}}\right|=0$ due to $h\in{C}([0,R];\mathbb{R})$.
Therefore, we complete the proof of (\ref{thm2-0911-3}).

For $\theta>0$,
we may use (\ref{chiun0102}) and follow the same argument as in (\ref{0910-1}) to obtain
\begin{align}\label{notagad-d}
\int_0^R|e^{pU_{\epsilon}}-1|^{\theta}\dr
&=\left\{\int_0^{R-\epsilon^\kappa}+\int_{R-\epsilon^\kappa}^R\right\}\left|e^{pU_{\epsilon}}-1\right|^{\theta}\dr\notag\\[-0.7em]
&\\[-0.7em]
&\leq\left[\left(\frac{B}{A}\right)^{\frac{p}{q}}+1\right]^{\theta}\left[\left(\frac{\widetilde{C}_{8}q}{N\log\frac{B}{A}}\epsilon^{\kappa}\log\frac{1}{\epsilon}\right)^{\theta}+\epsilon^\kappa\right]\notag,
\end{align}
which gives (\ref{add1111-2}).

By (\ref{thm2-0911-3}), (\ref{add1111-2}) and (\ref{0909-6}), we immediately obtain (\ref{1123-2016-1113pm}).
Similarly, for the case $0<B<A$ we can prove (\ref{add0103-2017}), (\ref{add1111-4}) and (\ref{1123-2016-1113pm}).\\

\textbf{Step~2.~Proof of (\ref{add1111-1}) and (\ref{add1111-3}).}
Thanks to (\ref{2016-1010-guochin}) and (\ref{etae-t}), we now state the proof of (\ref{add1111-1}).
For $h\in{C}([0,R];\mathbb{R})$, by (\ref{2016-1010-guochin}) one may check that
\begin{align}\label{1010-2016-812pm}
&\epsilon^2\int_0^Rh(t)U_{\epsilon}'^2(t)\dt\notag\\
&\quad\quad\quad=\epsilon^2\int_0^{\epsilon^{\kappa}}h(t)U_{\epsilon}'^2(t)\dt+
\int_{\epsilon^{\kappa}}^Rh(t)\left(\lambda_{\epsilon,\kappa}(t)-\sqrt{\frac{2A}{qg(t)}}e^{-\frac{q}{2}\left(U_{\epsilon}(t)-\frac{2}{q}\log\epsilon\right)}\right)U_{\epsilon}'(t)\dt\notag\\
&\quad\quad\quad=\epsilon^2\int_0^{\epsilon^{\kappa}}h(t)U_{\epsilon}'^2(t)\dt+\int_{\epsilon^{\kappa}}^Rh(t)\lambda_{\epsilon,\kappa}(t)U_{\epsilon}'(t)\dt\notag\\
&\quad\quad\quad\quad-h(R)\sqrt{\frac{2A}{qg(R)}}\int_{\epsilon^{\kappa}}^Re^{-\frac{q}{2}\left(U_{\epsilon}(t)-\frac{2}{q}\log\epsilon\right)}U_{\epsilon}'(t)\dt\\
&\quad\quad\quad\quad\hspace*{8pt}+\sqrt{\frac{2A}{q}}\int_{\epsilon^{\kappa}}^R\left(\frac{h(R)}{\sqrt{g(R)}}-\frac{h(t)}{\sqrt{g(t)}}\right)e^{-\frac{q}{2}\left(U_{\epsilon}(t)-\frac{2}{q}\log\epsilon\right)}U_{\epsilon}'(t)\dt.\notag
\end{align}

We shall deal with each term in the last line of (\ref{1010-2016-812pm}).
Firstly, we deal with $\epsilon^2\int_0^{\epsilon^{\kappa}}h(t)U_{\epsilon}'^2(t)\dt$.
Letting $t\to+0$ in (\ref{bigoneadd-as}) and using $U_{\epsilon}'(0)=0$,
one may check that
\begin{align}\label{bigoneadd-aspp}
&\frac{R^N}{N}\left(\frac{A\left(e^{pU_{\epsilon}(0)}-e^{pU_{\epsilon}(\epsilon^\kappa)}\right)}{p\int_0^{R}s^{N-1}e^{pU_{\epsilon}(s)}\dss}
+\frac{B\left(e^{-qU_{\epsilon}(0)}-e^{-qU_{\epsilon}(\epsilon^\kappa)}\right)}
{q\int_0^{R}s^{N-1}e^{-qU_{\epsilon}(s)}\dss}\right)\notag\\[-0.7em]
&\\[-0.7em]
&\quad\quad\quad=-\frac{\epsilon^2}{2}
g(\epsilon^\kappa)U_{\epsilon}'^2(\epsilon^\kappa)-\frac{\epsilon^2}{2}\int^{\epsilon^\kappa}_0\left(\frac{2(N-1)g(r)}{r}+g'(r)\right)U_{\epsilon}'^2(r)\dr.\notag
\end{align}
Combining (\ref{0823-99}) (\ref{0909-1}), (\ref{0909-6}) and (\ref{bigoneadd-aspp}),
we get
\begin{align*}
\lim_{\epsilon\downarrow0}\frac{\epsilon^2}{2}\int^{\epsilon^\kappa}_0\left(\frac{2(N-1)g(r)}{r}+g'(r)\right)U_{\epsilon}'^2(r)\dr=0.
\end{align*}
Moreover, for $\epsilon>0$ sufficiently small,
$\frac{2(N-1)g(r)}{r}+g'(r)\geq\frac{2(N-1)}{r}\min_{[0,R]}g-\max_{[0,R]}|g'|\gg1$ for $r\in(0,\epsilon^{\kappa}]$.
As a consequence, $\epsilon^2\int^{\epsilon^\kappa}_0U_{\epsilon}'^2(r)\dr\to0$ as $\epsilon\downarrow0$, and hence
\begin{align}\label{1011-2016-1207am}
\epsilon^2\left|\int_0^{\epsilon^{\kappa}}h(t)U_{\epsilon}'^2(t)\dt\right|\leq\epsilon^2\left(\max_{[0,R]}|h|\right)\int_0^{\epsilon^{\kappa}}U_{\epsilon}'^2(t)\dt\longrightarrow0\quad\mathrm{as}\,\,\epsilon\downarrow0.
\end{align}

By (\ref{780822-1-2016}), (\ref{leadingur}), (\ref{0909-1}), (\ref{etae-t}) and (\ref{1010-2016-1}), we have
\begin{align}\label{1010-2016-824pm}
\left|\int_0^Rh(t)\lambda_{\epsilon,\kappa}(t)U_{\epsilon}'(t)\dt\right|\leq{O}(1)
\left(\max_{[0,R]}|h|\right)\left(U_{\epsilon}(0)-U_{\epsilon}(R)\right)\epsilon^{\kappa}\log\frac{1}{\epsilon}\longrightarrow0\quad\mathrm{as}\,\,\epsilon\downarrow0,
\end{align}
and
\begin{align}\label{1010-2016-859pm}
&\left|\int_0^R\left(\frac{h(R)}{\sqrt{g(R)}}-\frac{h(t)}{\sqrt{g(t)}}\right)e^{-\frac{q}{2}\left(U_{\epsilon}(t)-\frac{2}{q}\log\epsilon\right)}U_{\epsilon}'(t)\dt\right|\notag\\
&\quad\quad\quad=\left|\left\{\int_0^{R-\epsilon^{\kappa}}+\int_{R-\epsilon^{\kappa}}^{R}\right\}\left(\frac{h(R)}{\sqrt{g(R)}}-\frac{h(t)}{\sqrt{g(t)}}\right)e^{-\frac{q}{2}\left(U_{\epsilon}(t)-\frac{2}{q}\log\epsilon\right)}U_{\epsilon}'(t)\dt\right|\notag\\[-0.7em]
&\\[-0.7em]
&\quad\quad\quad\leq\max_{[0,R]}\frac{|h|}{\sqrt{g}}\times\frac{4\epsilon}{q}\left(e^{-\frac{q}{2}U_{\epsilon}(R-\epsilon^{\kappa})}-e^{-\frac{q}{2}U_{\epsilon}(0)}\right)\notag\\
&\quad\quad\quad\quad+\max_{t\in[R-\epsilon^{\kappa},R]}\left|\frac{h(R)}{\sqrt{g(R)}}-\frac{h(t)}{\sqrt{g(t)}}\right|\times\frac{2\epsilon}{q}\left(e^{-\frac{q}{2}U_{\epsilon}(R)}-e^{-\frac{q}{2}U_{\epsilon}(R-\epsilon^{\kappa})}\right)\longrightarrow0\quad\mathrm{as}\,\,\epsilon\downarrow0.\notag
\end{align}
In (\ref{1010-2016-859pm})
 we have used the elementary property of $\frac{h}{\sqrt{g}}\in C([0,R];\mathbb{R})$
and verified that $e^{-\frac{q}{2}U_{\epsilon}(R-\epsilon^{\kappa})}$ and $e^{-\frac{q}{2}U_{\epsilon}(0)}$ 
are uniformly bounded to $\epsilon$, and
\begin{equation*}
\displaystyle\sup_{0<\epsilon\ll1}\frac{2\epsilon}{q}\left(e^{-\frac{q}{2}U_{\epsilon}(R)}-e^{-\frac{q}{2}U_{\epsilon}(R-\epsilon^{\kappa})}\right)=O(1)\quad\mathrm{(by\,\,(\ref{1010-2016-1}))}.
\end{equation*}

For (\ref{1010-2016-812pm}),
it remains to
 deal with $-h(R)\sqrt{\frac{2A}{qg(R)}}\int_0^Re^{-\frac{q}{2}\left(U_{\epsilon}(t)-\frac{2}{q}\log\epsilon\right)}U_{\epsilon}'(t)\dt$.
Notice $0<A<B$ and (\ref{780822-1-2016}),
one may check via (\ref{0909-1}) and (\ref{0916-thm2-1}) that 
\begin{align}\label{1011-2016-1208pm}
-h(R)\sqrt{\frac{2A}{qg(R)}}\int_0^R&e^{-\frac{q}{2}\left(U_{\epsilon}(t)-\frac{2}{q}\log\epsilon\right)}U_{\epsilon}'(t)\dt\notag\\
=&\frac{2\epsilon{h(R)}}{q}\sqrt{\frac{2A}{qg(R)}}\left(e^{-\frac{q}{2}U_{\epsilon}(R)}-e^{-\frac{q}{2}U_{\epsilon}(0)}\right)\\
=&\frac{2R(B-A)}{qNg(R)}h(R)+o_{\epsilon}(1),\quad\quad\quad\quad\quad\quad\quad\quad\quad\quad\mathrm{as}\,\,0<\epsilon\ll1.\notag
\end{align}
By the combination of (\ref{1010-2016-812pm}) and (\ref{1011-2016-1207am})--(\ref{1011-2016-1208pm}),
we conclude 
\begin{align}\label{1011-2016117pm}
\epsilon^2\int_0^Rh(t)U_{\epsilon}'^2(t)\dt=\frac{2R(B-A)}{qNg(R)}h(R)+o_{\epsilon}(1),\quad\quad\mathrm{as}\,\,0<\epsilon\ll1,
\end{align}
which gives (\ref{add1111-1}).
Similarly, we can prove (\ref{add1111-3}).
Therefore, we complete the proof of Theorem~\ref{thm3}(I) and (II).\\
	

\textbf{Step~3.~Proof of Theorem~\ref{thm3}(III).} 
Assume $0<A<B$. Then by (\ref{leadingur}) and (\ref{0909-1}) we have
$\int_0^R|\epsilon{U}_\epsilon'|\dr=-\epsilon\int_0^R{U}_\epsilon'\dr=\epsilon({U}_\epsilon(0)-{U}_\epsilon(R))\leq\widetilde{C}_{14}\epsilon\log\frac{1}{\epsilon}$, where constant~$\widetilde{C}_{14}>\frac{2}{q}$ is independent of $\epsilon$.
Hence, for $\widetilde{\theta}\in(0,1]$, using the H\"{o}lder inequality one gets
\begin{align}\label{theta0103}
\int_0^R|\epsilon{U}_\epsilon'|^{\widetilde{\theta}}\dr\leq{R}^{1-\widetilde{\theta}}\left(\int_0^R|\epsilon{U}_\epsilon'|\dr\right)^{\widetilde{\theta}}\leq\widetilde{C}_{15}\left(\epsilon\log\frac{1}{\epsilon}\right)^{\widetilde{\theta}},\quad\mathrm{for}\,\,\widetilde{\theta}\in(0,1],
\end{align}
where $\widetilde{C}_{15}={R}^{1-\widetilde{\theta}}\widetilde{C}_{14}^{\widetilde{\theta}}$.

For $\widetilde{\theta}\in(1,2)$, by virtue of (\ref{leadingur}), (\ref{0909-1}) and (\ref{2016-1010-guochin})--(\ref{etae-t})
we may follow the similar argument as in (\ref{1010-2016-812pm}) to check that
\begin{align}\label{theta0103new}
\int_0^R|\epsilon{U}_\epsilon'|^{\widetilde{\theta}}\dr=&-\epsilon^{2-\widetilde{\theta}}\int_0^R\left(-\epsilon^2{U}_\epsilon'\right)^{\widetilde{\theta}-1}{U}_\epsilon'\dr\notag\\
=&-\epsilon^{2-\widetilde{\theta}}\int_0^R\left(\sqrt{\frac{2A}{qg(r)}}e^{-\frac{q}{2}\left(U_{\epsilon}(r)-\frac{2}{q}\log\epsilon\right)}-\lambda_{\epsilon,\kappa}(r)\right)^{\widetilde{\theta}-1}{U}_\epsilon'(r)\dr\notag\\
\leq&-\epsilon^{2-\widetilde{\theta}}\int_0^R\left[\left(\epsilon\sqrt{\frac{2A}{q\min_{[0,R]}g}}e^{-\frac{q}{2}U_{\epsilon}(r)}\right)^{\widetilde{\theta}-1}+\left(\widetilde{C}_{12}\epsilon^{\kappa}\log\frac{1}{\epsilon}\right)^{\widetilde{\theta}-1}\right]{U}_\epsilon'(r)\dr\notag\\[-0.7em]
&\\[-0.7em]
=&\frac{2\epsilon}{q(\widetilde{\theta}-1)}\left(\frac{2A}{q\min_{[0,R]}g}\right)^{\frac{\widetilde{\theta}-1}{2}}\left(e^{-\frac{q}{2}(\widetilde{\theta}-1)U_{\epsilon}(R)}-e^{-\frac{q}{2}(\widetilde{\theta}-1)U_{\epsilon}(0)}\right)\notag\\
&\hspace*{8pt}+\epsilon^{2-\widetilde{\theta}}\left(\widetilde{C}_{12}\epsilon^{\kappa}\log\frac{1}{\epsilon}\right)^{\widetilde{\theta}-1}\left(U_{\epsilon}(0)-U_{\epsilon}(R)\right)\notag\\
\leq&\widetilde{C}_{16}\epsilon^{2-\widetilde{\theta}}\left(1+\epsilon^{\kappa(\widetilde{\theta}-1)}\left(\log\frac{1}{\epsilon}\right)^{\widetilde{\theta}}\right),\quad\mathrm{for}\,\,\widetilde{\theta}\in(1,2).\notag
\end{align}
Here we have used ${U}_\epsilon'\leq0$ and the elementary inequality $|a+b|^{\widetilde{\theta}-1}\leq|a|^{\widetilde{\theta}-1}+|b|^{\widetilde{\theta}-1}$
(due to the fact $\widetilde{\theta}-1\in(0,1)$) to get the third line of (\ref{theta0103new}).
As a consequence, (\ref{add1234-5-1119}) follows from (\ref{theta0103}) and (\ref{theta0103new}).

Similarly, for the case of $0<B<A$, we can prove (\ref{add1234-5-1119}). Therefore,
by Steps~1--3, we complete the proof of Theorem~\ref{thm3}.
\end{proof}


\section{Thin boundary layer structures with curvature effects}\label{compl-proof}
Theorem~\ref{thm3} provides a basic understanding 
on the concentration phenomenon of $U_{\epsilon}$ at the boundary point.
However,
it
merely gives a ``one-point jumped" for the limiting behavior of $U_{\epsilon}$ at
the boundary, and the structure of the boundary layer is hidden in the statement.
In order to further analyze the boundary concentrating behavior,
we shall establish the asymptotic expansion for $U_{\epsilon}$ at each point close to the boundary.

Firstly, we show that as $\epsilon\downarrow0$,
 $\displaystyle\max_{r\in[0,R-\epsilon^{\kappa}]}\left(|U_{\epsilon}(r)|+r^{N-1}|U_{\epsilon}'(r)|\right)$
approaches zero,
while $\displaystyle\min_{r\in[R-\epsilon^{\beta},R]}\left|U_{\epsilon}(r)\right|$
asymptotically blows up, where
\begin{align}\label{kabegamma}
\kappa\in(0,1),\quad\beta\in(1,\infty)
\end{align}
are independent of $\epsilon$.
Hence, for the asymptotic behavior of $U_{\epsilon}$,
 there is a great deal of difference in
the regions $[0,R-\epsilon^{\kappa}]$ and $[R-\epsilon^{\beta},R]$
as $0<\epsilon\ll1$.

For each $r_\epsilon\in[R-\epsilon^{\beta},R]$,
we establish the precise first two-order terms for the asymptotic expansions (as $0<\epsilon\ll1$)
of $U_{\epsilon}(r_\epsilon)$ with respect to $\epsilon$ (Note that $\beta>1$ is arbitrary).
Since
the leading-orders of those expansions
have various blow-up rates depending on the position of $r_\epsilon$,
we state those asymptotic expansions as follows
for the sake of clarity.

\begin{itemize}
\item Firstly, we focus on the situation where $r_{\epsilon}\equiv{r}_{\epsilon,\beta}\in(0,R)$
such that $R-r_{\epsilon,\beta}$
has a leading order term~$\gamma_{\beta}\epsilon^{\beta}$
for some $\beta>1$ and $\gamma_{\beta}\in(0,\infty)$, i.e.,
\begin{align}\label{1015-2016-gamma}
\lim_{\epsilon\downarrow0}\frac{R-r_{\epsilon,\beta}}{\epsilon^{\beta}}=\gamma_{\beta}\in(0,\infty).
\end{align}

\item Next, we consider the other situation where
 the leading order term of $R-r_{\epsilon}$
is not of the form $\epsilon^{\beta}$ for any $\beta>1$, i.e., there exists 
a $\beta_0>1$ independent of $\epsilon$ such that
$R-r_\epsilon=\pi_{\beta_0}(\epsilon)$ satisfies
\begin{align}\label{0121-2017-1}
\begin{cases}
\displaystyle\lim_{\epsilon\downarrow0}\frac{\pi_{\beta_0}(\epsilon)}{\epsilon^{\beta'}}=0,\,\,\,\quad&\mathrm{for}\,\,\beta'\in(1,\beta_0),\\
\displaystyle\lim_{\epsilon\downarrow0}\frac{\pi_{\beta_0}(\epsilon)}{\epsilon^{\beta''}}=\infty,\quad&\mathrm{for}\,\,\beta''\in(\beta_0,\infty),\\
\displaystyle\lim_{\epsilon\downarrow0}\frac{\pi_{\beta_0}(\epsilon)}{\epsilon^{\beta_0}}\in\{0,\infty\}.
\end{cases}
\end{align}
An example for (\ref{0121-2017-1}) is $\pi_{\beta_0}(\epsilon)=\gamma_{\beta_0}\epsilon^{\beta_0}\left(\log\frac{1}{\epsilon}\right)^{l}$ for $\gamma_{\beta_0}>0$ and $l\neq0$
independent of $\epsilon$.

We are also interested in some points 
$\displaystyle{r}_{\epsilon}\in\bigcap_{\kappa\in(0,1),\,\beta>1}[R-\epsilon^{\kappa},R-\epsilon^{\beta}]$ as $0<\epsilon\ll1$.
For this, we 
consider the situation where
 $R-r_{\epsilon}=\epsilon^{\beta_i(\epsilon)}$ and $\beta_i(\epsilon)$'s depends on $\epsilon$ 
and satisfy $\displaystyle\lim_{\epsilon\downarrow0}\beta_i(\epsilon)=1$
and $\displaystyle\lim_{\epsilon\downarrow0}\epsilon^{\beta_i(\epsilon)-1}=0$, and
provide novel asymptotic behaviors for $U_{\epsilon}(r_{\epsilon})$, $U_{\epsilon}'(r_{\epsilon})$
and $\rho_{\epsilon}(r_{\epsilon})$. 

We stress that the asymptotic blow-up behavior of $U_{\epsilon}(r_{\epsilon})$
for these cases
is more complicated than that for $r_{\epsilon}$ satisfying~(\ref{1015-2016-gamma}).
\end{itemize}

Taking full advantage of pointwise descriptions,
we exactly describe
effects of the curvature~$\frac{1}{R}$,
the concentration difference~$|A-B|$
and the boundary dielectric~$g(R)$
on $U_{\epsilon}(r_{\epsilon})$. 
Now we are in a position to state these two theorems as follows:

\begin{theorem}\label{thm2}
 Under the same hypotheses as in Theorem~\ref{thm3},
if $0<A<B$ (resp., $0<B<A$), then $U_{\epsilon}$ is monotonically decreasing (resp., monotonically increasing) on $[0,R]$. 
Moreover, we have
\begin{itemize}
\item[(I)] 
As $0<\epsilon\ll1$, 
there exist positive constants $C_{\kappa}$ and
 $M_{\kappa}$ independent of $\epsilon$ such that
\begin{align}\label{thm2-0911-1}
\max_{r\in[0,R-\epsilon^{\kappa}]}\left|U_{\epsilon}(r)\right|\leq{C}_{\kappa}\epsilon^{\kappa}\log\frac{1}{\epsilon}\quad\mathrm{and}\quad\max_{r\in[0,R-\epsilon^{\kappa}]}r^{N-1}\left|U_{\epsilon}'(r)\right|\leq{e}^{-\frac{{M}_{\kappa}}{\epsilon^{1-\kappa}}}.
\end{align}
\item[(II)] If ${r}_{\epsilon,\beta}\in(0,R)$ is sufficiently close to
the boundary such that (\ref{1015-2016-gamma}) holds, 
 then as $0<\epsilon\ll1$,
 $U_{\epsilon}(r_{\epsilon,\beta})$  
asymptotically blows up and admits the following expansions with precise first two order terms:
\begin{align}
U_{\epsilon}(r_{\epsilon,\beta})=-&\frac{2}{q}\min\{1,\beta-1\}\log\frac{1}{\epsilon}\notag\\[-0.7em]
&\label{1105-2016-ab1}\\[-0.7em]
&+\frac{2}{q}\log\left[\sqrt{\frac{A}{2qg(R)}}\left(\gamma_{\beta}q\boldsymbol{\chi}_{1}(\beta)+\frac{2Ng(R)}{R(B-A)}\boldsymbol{\chi}_{2}(\beta)\right)\right]+o_{\epsilon}(1),\,\,for\,\,0<A<B,\notag\\
U_{\epsilon}(r_{\epsilon,\beta})=\quad&\frac{2}{p}\min\{1,\beta-1\}\log\frac{1}{\epsilon}\notag\\[-0.7em]
&\label{1105-2016-ab2}\\[-0.7em]
&-\frac{2}{p}\log\left[\sqrt{\frac{B}{2pg(R)}}\left(\gamma_{\beta}p\boldsymbol{\chi}_{1}(\beta)+\frac{2Ng(R)}{R(A-B)}\boldsymbol{\chi}_{2}(\beta)\right)\right]+o_{\epsilon}(1),\,\,for\,\,0<B<A,\notag
\end{align}
where 
\begin{align}
\begin{cases}
\boldsymbol{\chi}_{1}(\beta)=1\quad{and}\quad\boldsymbol{\chi}_{2}(\beta)=0\quad{for}\quad\beta\in(1,2),\notag\\
 \boldsymbol{\chi}_{1}(\beta)=\boldsymbol{\chi}_{2}(\beta)=1\quad\quad\quad\quad\quad{for}\quad\beta=2,\notag\\
\boldsymbol{\chi}_{1}(\beta)=0\quad{and}\quad\boldsymbol{\chi}_{2}(\beta)=1\quad{for}\quad\beta\in(2,\infty).\notag
\end{cases}
\end{align}

Hence, when $\beta\geq2$, we have $\boldsymbol{\chi}_{2}(\beta)=1$
and the curvature $\frac{1}{R}$ exactly appears in the second order terms 
which are bounded quantities independent of $\epsilon$.
However, when $1<\beta<2$, we have $\boldsymbol{\chi}_{2}(\beta)=0$ and 
the curvature $\frac{1}{R}$ may appear in terms 
tending to zero as $\epsilon$ goes to zero.
Moreover, for $0<A<B$, the exact leading order term of $U_{\epsilon}'(r_{\epsilon,\beta})$ and $\rho_{\epsilon}(r_{\epsilon,\beta})$
are represented as follows:
\begin{itemize}
\item[(a)] When $1<\beta<2$ and $\gamma_{\beta}>0$, we have
\begin{align}
U_{\epsilon}'(r_{\epsilon,\beta})&=-\frac{2}{\gamma_{\beta}q}\epsilon^{-\beta}(1+o_{\epsilon}(1)),\label{1105-162}\\
\rho_{\epsilon}(r_{\epsilon,\beta})&=\quad\frac{2g(R)}{\gamma_{\beta}^2q}\epsilon^{-2(\beta-1)}(1+o_{\epsilon}(1)),\label{0111-2017-1}
\end{align}
as $0<\epsilon\ll1$.

\item[(b)] When $\beta=2$ and $\gamma_{\beta}>0$, we have
\begin{align}
U_{\epsilon}'(r_{\epsilon,\beta})&=-\frac{2}{\gamma_\beta{q}+\frac{2Ng(R)}{R(B-A)}}\epsilon^{-2}(1+o_{\epsilon}(1)),\label{1105-164}\\
\rho_{\epsilon}(r_{\epsilon,\beta})&=\quad\frac{2qg(R)}{\left(\gamma_{\beta}q+\frac{2Ng(R)}{R(B-A)}\right)^{2}}\epsilon^{-2}(1+o_{\epsilon}(1)),\label{0111-2017-2}
\end{align}
as $0<\epsilon\ll1$.

\item[(c)] If $\beta>2$, then for any $\gamma_{\beta}>0$, we have
\begin{align}
U_{\epsilon}'(r_{\epsilon,\beta})&=-\frac{R(B-A)}{Ng(R)}\epsilon^{-2}(1+o_{\epsilon}(1)),\label{nov11-2016}\\
\rho_{\epsilon}(r_{\epsilon,\beta})&=\quad\frac{qR^2(B-A)^2}{2N^2g(R)}\epsilon^{-2}(1+o_{\epsilon}(1)),\label{0111-2017-3}
\end{align}
as $0<\epsilon\ll1$.

Particularly, in this case  both $U_{\epsilon}(r_{\epsilon,\beta})$ and $U_{\epsilon}(R)$
share the same first two order terms, and both  $U_{\epsilon}'(r_{\epsilon,\beta})$ and $U_{\epsilon}'(R)$
have the same leading order term.
\end{itemize}
\end{itemize}
Analogous results of (a)--(c) also hold for the case of $0<B<A$.
\end{theorem}

\begin{theorem}\label{thm3-new1106}
Under the same hypotheses as in Theorem~\ref{thm3},
in addition, we assume $0<A<B$. Then
\begin{itemize}
\item[(I)] Let $r_\epsilon\in(0,R)$ be sufficiently close to the boundary such that
$R-r_\epsilon=\pi_{\beta_0}(\epsilon)$ satisfies
(\ref{0121-2017-1}). Then we have
\begin{itemize}
\item[(a)] In the situation $\beta_0\in(1,2)$
or the situation $\beta_0=2$ and $\displaystyle\lim_{\epsilon\downarrow0}\frac{\pi_2(\epsilon)}{\epsilon^2}=\infty$,
 $U_{\epsilon}(r_{\epsilon})$, $U_{\epsilon}'(r_{\epsilon})$
and $\rho_{\epsilon}(r_{\epsilon})$ admit the following expansions:
\begin{align}
U_{\epsilon}(r_{\epsilon})=&-\frac{2}{q}(\beta_0-1)\log\frac{1}{\epsilon}+\frac{2}{q}\log\frac{\pi_{\beta_0}(\epsilon)}{\epsilon^{\beta_0}}
+\frac{2}{q}\log\sqrt{\frac{qA}{2g(R)}}+o_{\epsilon}(1),\label{20170129-a1}\\
U_{\epsilon}'(r_{\epsilon})=&-\frac{2}{q\pi_{\beta_0}(\epsilon)}(1+o_{\epsilon}(1)),\label{20170129-a2}\\
\rho_{\epsilon}(r_{\epsilon})=&\quad\frac{2g(R)}{q}\left(\frac{\epsilon}{\pi_{\beta_0}(\epsilon)}\right)^2(1+o_{\epsilon}(1)),\label{20170129-a3}
\end{align}
 for $0<\epsilon\ll1$.

We stress that the first two terms of (\ref{20170129-a1})
asymptotically blow up, but $\left|\log\frac{\pi_{\beta_0}(\epsilon)}{\epsilon^{\beta_0}}\right|\ll\log\frac{1}{\epsilon}$ as $0<\epsilon\ll1$ (by (\ref{0121-2017-1})), 
which is different from the asymptotic expansion
 of (\ref{1105-2016-ab1}).
 
\item[(b)] In the situation $\beta_0\in(2,\infty)$
or the situation $\beta_0=2$ and $\displaystyle\lim_{\epsilon\downarrow0}\frac{\pi_2(\epsilon)}{\epsilon^2}=0$,
 for $0<\epsilon\ll1$,
 $U_{\epsilon}(r_{\epsilon})$ has the same asymptotic expansion as (\ref{1105-2016-ab1}) with $\boldsymbol{\chi}_{1}(\beta)=0$
and $\boldsymbol{\chi}_{2}(\beta)=1$;
 $U_{\epsilon}'(r_{\epsilon})$
and $\rho_{\epsilon}(r_{\epsilon})$ have the same asymptotic expansions
as (\ref{nov11-2016}) and (\ref{0111-2017-3}), respectively.
\end{itemize}
\item[(II)] Let $\Theta\in(0,1)$, $\gamma>0$ and $\tau\in\mathbb{R}$ be constants independent of $\epsilon$ and
\begin{align*}
{\kappa({\epsilon})}=&\,1-\left(\log\frac{1}{\epsilon}\right)^{-\Theta},\,\,\beta_1({\epsilon})=1+\gamma\left(\log\frac{1}{\epsilon}\right)^{-\Theta},\\
\beta_2({\epsilon})=&\,1+\left(\log\frac{1}{\epsilon}\right)^{-1}\left(\gamma\log^{(n)}\frac{1}{\epsilon}+\tau\right),
\end{align*}
where $n\geq2$ and $\log^{(n)}$ denotes as the $n$ times self-composite function of the natural logarithmic function. 
Then for $0<\kappa<1$ and $\beta>1$ fixed,
$R-\epsilon^{\kappa({\epsilon})}$ and $R-\epsilon^{\beta_i({\epsilon})}$ are located in $[R-\epsilon^{\kappa},R-\epsilon^{\beta}]$
as $0<\epsilon\ll1$, and there hold
\begin{align}
\lim_{\epsilon\downarrow0}|U_{\epsilon}&(R-{\epsilon}^{\kappa({\epsilon})})|
=\lim_{\epsilon\downarrow0}|U_{\epsilon}'(R-{\epsilon}^{\kappa({\epsilon})})|=0,\label{01032017-1}\\
U_{\epsilon}(R-{\epsilon}^{\beta_1({\epsilon})})=&-\frac{2\gamma}{q}\left(\log\frac{1}{\epsilon}\right)^{1-\Theta}+\frac{1}{q}\log\frac{qA}{2g(R)}+o_{\epsilon}(1),\label{01032017-2}\\
U_{\epsilon}(R-{\epsilon}^{\beta_2({\epsilon})})=&-\frac{2\gamma}{q}\log^{(n)}\frac{1}{\epsilon}+\frac{1}{q}\left(-2\tau+\log\frac{qA}{2g(R)}\right)+o_{\epsilon}(1).\label{01032017-3}
\end{align}
Moreover, we have 
\begin{align}\label{1013-123}
\begin{cases}
U_{\epsilon}'(R-{\epsilon}^{\beta_1({\epsilon})})=&-\displaystyle\frac{2}{q}\epsilon^{-\beta_1(\epsilon)}(1+o_{\epsilon}(1)),\vspace{5pt}\\
\rho_{\epsilon}(R-{\epsilon}^{\beta_1({\epsilon})})=&\displaystyle\frac{2g(R)}{q}\epsilon^{-2(\beta_1(\epsilon)-1)}(1+o_{\epsilon}(1)),
\end{cases}
\end{align}
and for $n=2$ and $\gamma>2$,
\begin{align}\label{1013-556}
\begin{cases}
U_{\epsilon}'(R-{\epsilon}^{\beta_2({\epsilon})})=&-\displaystyle\frac{2}{q}\epsilon^{-\beta_2(\epsilon)}(1+o_{\epsilon}(1)),\vspace{5pt}\\
\rho_{\epsilon}(R-{\epsilon}^{\beta_2({\epsilon})})=&\displaystyle\frac{2g(R)}{q}\epsilon^{-2(\beta_2(\epsilon)-1)}(1+o_{\epsilon}(1)).
\end{cases}
\end{align}
\end{itemize}
\end{theorem}

The proof of Theorems~\ref{thm2} and~\ref{thm3-new1106} are stated in Sections~\ref{sec6thm2}
and~\ref{sec6thm3-new1106}, respectively.

\begin{remark}\label{rkcor1} \ \ \ \ 
\begin{itemize}
\item[(i)] Theorem~\ref{thm2} implies
the coefficient of the leading order term is exactly determined 
by the valence of the cation species (resp., anion species)
when the total concentration of cation species (resp., anion species) is strictly large than that of anion species (resp., cation species). 

\item[(ii)] Theorems~\ref{thm2} and \ref{thm3-new1106} present that as
$r_\epsilon$ satisfies $\lim_{\epsilon\downarrow0}\frac{R-r_\epsilon}{\epsilon^2}>0$
and $\lim_{\epsilon\downarrow0}\frac{R-r_\epsilon}{\epsilon}=0$,
we have
$|U_\epsilon'(r_\epsilon)|\sim\frac{1}{|R-r_\epsilon|}$.
\end{itemize}
\end{remark}

  To prove Theorems~\ref{thm2} and \ref{thm3-new1106}, we shall define an auxiliary function
\begin{align}\label{psi-2016-1}
\psie(t)=e^{\frac{q}{2}\left(U_{\epsilon}(t)-\frac{2}{q}\log\epsilon\right)}
\end{align}
which transforms (\ref{2016-1010-guochin}) into
\begin{align}\label{psi-2016-2}
\epsilon^2\psie'(t)-\frac{q}{2}\lambda_{\epsilon,\kappa}(t)\psie(t)+\sqrt{\frac{qA}{2g(t)}}=0,\quad\mathrm{for}\,\,t\in[\epsilon^{\kappa},R].
\end{align} 
On the other hand, one may check from (\ref{0916-thm2-1}) and (\ref{psi-2016-1}) that
\begin{align}\label{psi-2016-5}
\psie(R)=\frac{N}{R(B-A)}\sqrt{\frac{2Ag(R)}{q}}+o_{\epsilon}(1).
\end{align}

Multiplying (\ref{psi-2016-2}) by $\exp\left(-\frac{q}{2\epsilon^2}\int_{r_{\epsilon}}^t\lambda_{\epsilon,\kappa}(s)\dss\right)$,
integrating the expansion over $[r_{\epsilon},R]$
and  going through simple calculations,
one may check that
\begin{align}\label{psi-2016-3}
\psie(r_{\epsilon})=&\psie(R)e^{-\frac{q}{2\epsilon^2}\int_{r_{\epsilon}}^R\lambda_{\epsilon,\kappa}(s)\dss}\notag\\[-0.7em]
&\\[-0.7em]
&+\frac{1}{\epsilon^2}\sqrt{\frac{qA}{2}}\int_{r_{\epsilon}}^R\frac{1}{\sqrt{g(t)}}e^{-\frac{q}{2\epsilon^2}\int_{r_{\epsilon}}^t\lambda_{\epsilon,\kappa}(s)\dss}\dt.\notag
\end{align}

\subsection{Proof of Theorem~\ref{thm2}}\label{sec6thm2}
The monotonicity of $U_{\epsilon}$ is a direct result of (\ref{780822-1-2016}).
Now we give the proof of Theorem~\ref{thm2}(I).
Assume $0<A<B$. Due to $\kappa\in(0,1)$ and (\ref{0823-99})--(\ref{0823-99-eoei}), 
one immediately obtains
\begin{align}\label{1009-2016-1}
\max_{r\in[\epsilon^{\kappa},R-\epsilon^{\kappa}]}\left|U_{\epsilon}'(r)\right|
\leq{e}^{-\frac{{\widehat{C}}_{2,\kappa}}{\epsilon^{1-\kappa}}},\quad\mathrm{as}\,\,0<\epsilon\ll1,
\end{align}
where ${\widehat{C}}_{2,\kappa}$ is a positive constant independent of $\epsilon$.
On the other hand,  
 the combination of (\ref{1110822-1-2016}), (\ref{780822-1-2016}) and (\ref{1009-2016-1})
implies, for $r\in(0,\epsilon^{\kappa}]$, that
\begin{align}\label{1009-2016-2}
0\geq{r}^{N-1}U_{\epsilon}'(r)\geq&\frac{{g}(\epsilon^{\kappa})}{{g}(r)}\epsilon^{\kappa(N-1)}U_{\epsilon}'(\epsilon^{\kappa})\notag\\
\geq&-\left(\frac{\max_{[0,R]}g}{\min_{[0,R]}g}\right)\epsilon^{\kappa(N-1)}{e}^{-\frac{{\widehat{C}}_{2,\kappa}}{\epsilon^{1-\kappa}}}\\
\geq&-{e}^{-\frac{{\widehat{C}}_{2,\kappa}}{2\epsilon^{1-\kappa}}},\quad\quad\mathrm{as}\,\,0<\epsilon\ll1.\notag 
\end{align}
Hence, for the case $0<A<B$,
(\ref{thm2-0911-1})  immediately follows from (\ref{0823-99}), (\ref{1009-2016-1}) and (\ref{1009-2016-2}).
Using the same argument we can prove (\ref{thm2-0911-1}) for the case $0<B<A$.
This completes the proof of 
Theorem~\ref{thm2}(I).

To prove Theorem~\ref{thm2}(II),
 we consider the case when $r_{\epsilon}=r_{\epsilon,\beta}$ satisfies (\ref{1015-2016-gamma}),
i.e., 
\begin{align}\label{0201-2017-1}
r_{\epsilon}=r_{\epsilon,\beta}=R-(\gamma_{\beta}+o_{\epsilon}(1))\epsilon^{\beta},\quad\mathrm{as}\,\,0<\epsilon\ll1.
\end{align}
We shall deal with the each right-hand term of (\ref{psi-2016-3}).
Note that $\beta>1$. Thus, for any $\kappa\in(0,1)$, we have $[r_{\epsilon,\beta},R]\subset[\epsilon^{\kappa},R]$ as $0<\epsilon\ll1$.
In particular, one may choose   
\begin{align}\label{kappa-1014}
\kappa\in(\max\{0,2-\beta\},1)\subset(0,1)
\end{align}
so that $\kappa+\beta>2$.
Setting $r_{\epsilon}=r_{\epsilon,\beta}$ in (\ref{etae-t}) and using (\ref{0201-2017-1}) and (\ref{kappa-1014}),
one finds
\begin{align}\label{psi-2016-4}
\left|\frac{q}{2\epsilon^2}\int_{r_{\epsilon,\beta}}^R\lambda_{\epsilon,\kappa}(s)\dss\right|\leq\widetilde{C}_{13}\epsilon^{\kappa+\beta-2}\log\frac{1}{\epsilon}\longrightarrow0\quad\mathrm{as}\,\,\epsilon\downarrow0.
\end{align}
Hence, 
applying the Taylor expansions to $\frac{1}{\sqrt{g(t)}}$ and using~(\ref{psi-2016-4})
and $|t-R|\leq(\gamma_{\beta}+o_{\epsilon}(1))\epsilon^{\beta}$, we obtain
\begin{align}\label{psi-2016-6}
&\int_{r_{\epsilon,\beta}}^R\frac{1}{\sqrt{g(t)}}e^{-\frac{q}{2\epsilon^2}\int_{r_{\epsilon,\beta}}^t\lambda_{\epsilon,\kappa}(s)\dss}\dt\notag\\
=&\,\int_{r_{\epsilon,\beta}}^R\frac{1}{\sqrt{g(R)}}\left(1-\frac{g'(R)}{2g(R)}(t-R)+o_{\epsilon}(1)\right)\dt\\
=&\,\gamma_{\beta}\epsilon^{\beta}\left(\frac{1}{\sqrt{g(R)}}+o_{\epsilon}(1)\right).\notag
\end{align}
Combining (\ref{psi-2016-5}), (\ref{psi-2016-3}), (\ref{psi-2016-4}) and (\ref{psi-2016-6}), 
we arrive at
\begin{align}\label{psi-2016-7}
\psie(r_{\epsilon,\beta})=&\left(\frac{N}{R(B-A)}\sqrt{\frac{2Ag(R)}{q}}+o_{\epsilon}(1)\right)\notag\\[-0.7em]
&\\[-0.7em]
&\quad\quad\quad+\epsilon^{\beta-2}\left(\gamma_{\beta}\sqrt{\frac{qA}{2g(R)}}+o_{\epsilon}(1)\right).\notag
\end{align}

Next, we shall deal with (\ref{psi-2016-7}) via the following three cases for $\beta$.

\underline{\textbf{Case~1.}} When $1<\beta<2$ and $\gamma_{\beta}>0$,
using (\ref{psi-2016-1}) and applying the approximation $\log(1+\mu)\approx\mu$ for $|\mu|\ll1$
to (\ref{psi-2016-7}), we arrive at
\begin{align}\label{psi-2016-8}
U_{\epsilon}(r_{\epsilon,\beta})=&\frac{2}{q}\log\epsilon+\frac{2}{q}\log\psie(r_{\epsilon,\beta})\notag\\
=&\frac{2}{q}\log\epsilon+\frac{2}{q}\log\Bigg\{\epsilon^{\beta-2}\left(\gamma_{\beta}\sqrt{\frac{qA}{2g(R)}}+o_{\epsilon}(1)\right)\times\notag\\[-0.7em]
&\\[-0.7em]
&\hspace*{80pt}\left[1+\epsilon^{2-\beta}\left(\frac{2Ng(R)}{\gamma_{\beta}{q}R(B-A)}+o_{\epsilon}(1)\right)\right]\Bigg\}\notag\\
=&\frac{2}{q}(\beta-1)\log\epsilon+\frac{2}{q}\log\left(\gamma_{\beta}\sqrt{\frac{qA}{2g(R)}}\right)+o_{\epsilon}(1),\notag
\end{align}
which gives (\ref{1105-2016-ab1}) with $\boldsymbol{\chi}_{1}(\beta)=1$
and $\boldsymbol{\chi}_{2}(\beta)=0$.

By (\ref{2016-1010-guochin}) and (\ref{psi-2016-8}),
one may check that $\epsilon^2U_{\epsilon}'(r_{\epsilon,\beta})=-\frac{2}{\gamma_{\beta}{q}}\epsilon^{2-\beta}(1+o_{\epsilon}(1))
+\lambda_{\epsilon,\kappa}(r_{\epsilon,\beta})$,
together with (\ref{etae-t}) and (\ref{kappa-1014}) give
\begin{align}\label{psi-2016-9}
\left|\epsilon^{\beta}U_{\epsilon}'(r_{\epsilon,\beta})+\frac{2}{\gamma_{\beta}{q}}\right|\leq&\epsilon^{\beta-2}|\lambda_{\epsilon,\kappa}(r_{\epsilon,\beta})|+o_{\epsilon}(1)\notag\\[-0.7em]
&\\[-0.7em]
\leq&\widetilde{C}_{12}\epsilon^{\kappa+\beta-2}\log\frac{1}{\epsilon}+o_{\epsilon}(1)\longrightarrow0\quad\mathrm{as}\,\,\epsilon\downarrow0.\notag
\end{align}
 (\ref{1105-162}) immediately follows from (\ref{psi-2016-9}).

For $\rho_{\epsilon}(r_{\epsilon,\beta})$,
we can use (\ref{chdensity-1204}), Lemma~\ref{lem-0910-new}, (\ref{0201-2017-1}) and (\ref{psi-2016-8})
to obtain
\begin{align}\label{2017-0202-1}
\rho_{\epsilon}(r_{\epsilon,\beta})=&\left(-A+O(1)\epsilon^{\kappa}\log\frac{1}{\epsilon}\right)\notag\\
&\times\Bigg[\exp\left(\frac{2p}{q}(\beta-1)\log\epsilon+\frac{p}{q}\log\frac{\gamma_{\beta}^2qA}{2g(R)}+o_{\epsilon}(1)\right)\notag\\[-0.7em]\\[-0.7em]
&\quad\quad\quad-\exp\left(-2(\beta-1)\log\epsilon-\log\frac{\gamma_{\beta}^2qA}{2g(R)}+o_{\epsilon}(1)\right)\Bigg]\notag\\
=&\frac{2g(R)}{\gamma_{\beta}^2q}\epsilon^{-2(\beta-1)}(1+o_{\epsilon}(1)).\notag
\end{align}
Therefore, we get (\ref{0111-2017-1}) and complete the proof of Theorem~\ref{thm2}(II-a).\\

\underline{\textbf{Case~2.}} When $\beta=2$ and $\gamma_{\beta}>0$, by (\ref{psi-2016-1}) and (\ref{psi-2016-7}), 
we obtain
\begin{align}\label{psi-2016-10}
U_{\epsilon}(r_{\epsilon,\beta})=&\frac{2}{q}\log{\epsilon}+
\frac{2}{q}\left[\log\sqrt{\frac{A}{2qg(R)}}+\log\left(\gamma_{\beta}q+\frac{2Ng(R)}{R(B-A)}\right)\right]+o_{\epsilon}(1).
\end{align}
Therefore, for $\beta=2$,
 we arrive at  (\ref{1105-2016-ab1}) with $\boldsymbol{\chi}_{1}(\beta)=\boldsymbol{\chi}_{2}(\beta)=1$.
Furthermore, both (\ref{2016-1010-guochin}) and (\ref{psi-2016-10}) give
\begin{align}\label{psi-2016-11}
\epsilon^2U_{\epsilon}'(r_{\epsilon,\beta})=&-\frac{\sqrt{\frac{2A}{qg(R-\gamma_{\beta}\epsilon^2)}}}{\sqrt{\frac{qA}{2g(R)}}\left(\gamma_{\beta}+\frac{2Ng(R)}{qR(B-A)}+o_{\epsilon}(1)\right)}+\lambda_{\epsilon,\kappa}(r_{\epsilon,\beta})\notag\\[-0.7em]
&\\[-0.7em]
=&-\frac{2}{\gamma_{\beta}{q}+\frac{2Ng(R)}{R(B-A)}}+o_{\epsilon}(1),\quad\quad\quad\quad\mathrm{as}\,\,0<\epsilon\ll1,\notag
\end{align}
which gives (\ref{1105-164}).
Here we have used (\ref{etae-t}) to get the last estimate.

Using (\ref{psi-2016-10}) and following the same argument 
as (\ref{2017-0202-1}), we get (\ref{0111-2017-2})
and complete the proof of Theorem~\ref{thm2}(II-b).\\

\underline{\textbf{Case~3.}} When $\beta>2$ and $\gamma_{\beta}\geq0$, (\ref{psi-2016-1}) and (\ref{psi-2016-7}) immediately imply
\begin{align}\label{psi-2016-12}
U_{\epsilon}(r_{\epsilon,\beta})=\frac{2}{q}\log\epsilon+\frac{2}{q}\log\left(\frac{N}{R(B-A)}\sqrt{\frac{2Ag(R)}{q}}\right)+o_{\epsilon}(1).
\end{align}
Consequently, we get (\ref{1105-2016-ab1}) with $\boldsymbol{\chi}_{1}(\beta)=0$ and $\boldsymbol{\chi}_{2}(\beta)=1$.
In particular, 
(\ref{psi-2016-12}) is independent of $\gamma_{\beta}$
and has the same first two order terms as $U_{\epsilon}(R)$ as $0<\epsilon\ll1$.

(\ref{nov11-2016}) follows from the calculation
\begin{align}\label{psi-2016-13}
\epsilon^2U_{\epsilon}'(r_{\epsilon,\beta})=&-\sqrt{\frac{2A}{qg(r_{\epsilon,\beta})}}\left(\frac{R(B-A)}{N}\sqrt{\frac{q}{2Ag(R)}}+o_{\epsilon}(1)\right)+\lambda_{\epsilon,\kappa}(r_{\epsilon,\beta})\notag\\[-0.7em]
&\\[-0.7em]
=&-\frac{R(B-A)}{Ng(R)}+o_{\epsilon}(1)\quad\quad\mathrm{(by\,\,(\ref{2016-1010-guochin})\,\,and\,\,(\ref{psi-2016-12}))}.\notag
\end{align}
In this case,
the leading order term of $U_{\epsilon}'(r_{\epsilon,\beta})$
has the same form as $U_{\epsilon}'(R)$.

Finally, using (\ref{psi-2016-12}) and following the same argument 
as (\ref{2017-0202-1}), we get (\ref{0111-2017-3}).
Therefore, we complete the proof of Theorem~\ref{thm2}(II-c).\\

By Cases~1--3, we prove (\ref{1105-2016-ab1}). Similarly, we can prove (\ref{1105-2016-ab2}).
This completes the proof of Theorem~\ref{thm2}.

\subsection{Proof of Theorem~\ref{thm3-new1106}}\label{sec6thm3-new1106}
Let $r_\epsilon\in(0,R)$ and
$R-r_\epsilon=\pi_{\beta_0}(\epsilon)$ satisfies
(\ref{0121-2017-1}). Since $\beta_0>1$, we can choose $\beta'\in(1,\beta_0)$ and $\kappa'\in(0,1)$
such that $\kappa'+\beta'>2$.
Note that $\displaystyle\lim_{\epsilon\downarrow0}\frac{\pi_{\beta_0}(\epsilon)}{\epsilon^{\beta'}}=0$.
Following the same argument as (\ref{psi-2016-4})--(\ref{psi-2016-7}), 
for (\ref{psi-2016-3}) we have the following estimates:

\begin{align}\label{psi-2017-4}
\left|\frac{q}{2\epsilon^2}\int_{r_{\epsilon}}^R\lambda_{\epsilon,\kappa}(s)\dss\right|\leq\widetilde{C}_{13}\epsilon^{\kappa'+\beta'-2}\log\frac{1}{\epsilon}\longrightarrow0\quad\mathrm{as}\,\,\epsilon\downarrow0,
\end{align}
\begin{align}\label{psi-2017-6}
&\int_{r_{\epsilon}}^R\frac{1}{\sqrt{g(t)}}e^{-\frac{q}{2\epsilon^2}\int_{r_{\epsilon}}^t\lambda_{\epsilon,\kappa}(s)\dss}\dt=\pi_{\beta_0}(\epsilon)\left(\frac{1}{\sqrt{g(R)}}+o_{\epsilon}(1)\right),
\end{align}
\begin{align}\label{psi-2017-7}
\psie(r_{\epsilon})=\left(\frac{N}{R(B-A)}\sqrt{\frac{2Ag(R)}{q}}+o_{\epsilon}(1)\right)+\frac{\pi_{\beta_0}(\epsilon)}{\epsilon^{2}}\left(\sqrt{\frac{qA}{2g(R)}}+o_{\epsilon}(1)\right).
\end{align}

When $\beta_0$ satisfies conditions stated in Theorem~\ref{thm3-new1106}(I-a),
we have $\displaystyle\lim_{\epsilon\downarrow0}\frac{\pi_{\beta_0}(\epsilon)}{\epsilon^{2}}=\infty$.
Hence, by (\ref{psi-2017-7}) one finds
\begin{equation}\label{0721-2018}
\psie(r_{\epsilon})=\frac{\pi_{\beta_0}(\epsilon)}{\epsilon^{2}}\left(\sqrt{\frac{qA}{2g(R)}}+o_{\epsilon}(1)\right)
\end{equation}
due to the fact that the first term of (\ref{psi-2017-7}) is a bounded quantity which is far smaller than the second order term
as $0<\epsilon\ll1$.
 Combining (\ref{psi-2016-1}) with (\ref{0721-2018}) immediately obtains~(\ref{20170129-a1}).
Moreover, following same arguments as (\ref{psi-2016-9}) and (\ref{2017-0202-1}),
we can prove (\ref{20170129-a2}) and (\ref{20170129-a3}).
This completes the proof of Theorem~\ref{thm3-new1106}(I-a).

On the other hand, when $\beta_0$ satisfies conditions stated in Theorem~\ref{thm3-new1106}(I-b),
we have $\displaystyle\lim_{\epsilon\downarrow0}\frac{\pi_{\beta_0}(\epsilon)}{\epsilon^{2}}=0$.
Hence, $\psie(r_{\epsilon})=\frac{N}{R(B-A)}\sqrt{\frac{2Ag(R)}{q}}+o_{\epsilon}(1)$.
Along with~(\ref{psi-2016-1}) yields 
$$U_{\epsilon}(r_{\epsilon})=\frac{2}{q}\log\epsilon+\frac{2}{q}\log\left(\frac{N}{R(B-A)}
\sqrt{\frac{2Ag(R)}{q}}\right)+o_{\epsilon}(1)$$
having the same form as (\ref{psi-2016-12}).
Hence, following the similar argument,
we can obtain the asymptotic expansions of $U_{\epsilon}'(r_{\epsilon})$
and $\rho_{\epsilon}(r_{\epsilon})$
which have the same forms as  (\ref{nov11-2016}) and (\ref{0111-2017-3}).
Therefore, Theorem~\ref{thm3-new1106}(I-b) follows.

It remains to prove Theorem~\ref{thm3-new1106}(II).
 For $\gamma>0$, $\tau\in\mathbb{R}$, $n\geq2$ and $0<\Theta<1$, 
we set 
\begin{align}\label{1020-2016}
\kappa_1(\epsilon)=1-\frac{\gamma}{2}\left(\log\frac{1}{\epsilon}\right)^{-\Theta}\quad\mathrm{and}\quad
\kappa_2(\epsilon)=1-\left(\log\frac{1}{\epsilon}\right)^{-1}\left(\frac{\gamma}{2}\log^{(n)}\frac{1}{\epsilon}+\tau\right).
\end{align}
As for (\ref{thm2-0911-1}),
 one may use Lemma~\ref{lem-aug22-2} and the same argument of (\ref{0830-cchaha})--(\ref{0909-1})   
with $\kappa=\kappa_i(\epsilon)$ ($i=1,2$) 
to check that
\begin{align}
\max_{[\delta,R-\epsilon^{\kappa_i(\epsilon)}]}\left|U_{\epsilon}'\right|\leq\widetilde{C}_{6}e^{-\frac{\widetilde{M}}{4\epsilon^{1-\kappa_i(\epsilon)}}}&\ll1\label{1020-2016-00},\\
\max_{[0,R-\epsilon^{\kappa_i(\epsilon)}]}\left|U_{\epsilon}\right|\leq\widetilde{C}_{8}\epsilon^{\kappa_i(\epsilon)}\log\frac{1}{\epsilon}&\ll1,\quad\mathrm{as}\,\,0<\epsilon\ll1,\label{1020-2016-01}
\end{align} 
where $\delta$, $\widetilde{C}_{6}$, $\widetilde{C}_{8}$ and $\widetilde{M}$ are positive constants defined in
Lemma~\ref{lem-aug22-2}, (\ref{0830-cchaha}) and (\ref{0909-1}). Here we have verified
that as $\epsilon$ goes to zero,
\begin{align}
\frac{\widetilde{M}}{\epsilon^{1-\kappa_1(\epsilon)}}=e^{\frac{\widetilde{M}\gamma}{2}\left(\log\frac{1}{\epsilon}\right)^{1-\Theta}}\longrightarrow\infty,\quad\epsilon^{\kappa_1(\epsilon)}\log\frac{1}{\epsilon}={e}^{\frac{\gamma}{2}\left(\log\frac{1}{\epsilon}\right)^{1-\Theta}}\epsilon\log\frac{1}{\epsilon}\longrightarrow 0,\label{1020-2016-02}\\
\frac{\widetilde{M}}{\epsilon^{1-\kappa_2(\epsilon)}}=e^{\widetilde{M}\left(\frac{\gamma}{2}\log^{(n)}\frac{1}{\epsilon}+\tau\right)}\longrightarrow\infty,\quad\epsilon^{\kappa_2(\epsilon)}\log\frac{1}{\epsilon}={e}^{\frac{\gamma}{2}\log^{(n)}\frac{1}{\epsilon}+\tau}\epsilon\log\frac{1}{\epsilon}\longrightarrow 0.\label{1020-2016-02}
\end{align}
Therefore, (\ref{01032017-1}) follows.

To estimate $U_{\epsilon}(R-\epsilon^{\beta_i(\epsilon)})$
and $U_{\epsilon}'(R-\epsilon^{\beta_i(\epsilon)})$, we shall deal with each term of the right-hand side of (\ref{psi-2016-3}) and re-establish the related estimates from (\ref{psi-2016-7}).

\underline{\textbf{Case {1\'{}}.}} For $\beta_1(\epsilon)=1+\gamma\left(\log\frac{1}{\epsilon}\right)^{-\Theta}$ with $\gamma>0$
and $0<\Theta<1$,
we have 
\begin{align*}
\kappa_1(\epsilon)+\beta_1(\epsilon)-2=\frac{\gamma}{2}\left(\log\frac{1}{\epsilon}\right)^{-\Theta}\longrightarrow0
\end{align*}
 as $\epsilon\downarrow0$. Moreover,
for $\widehat{r}_{\epsilon}=R-\epsilon^{\beta_1(\epsilon)}$,
one may use (\ref{1020-2016-00})--(\ref{1020-2016-01}) and follow the same argument of (\ref{1010-2016-2pm})--(\ref{1010-2016-2}) to check that
\begin{align}\label{1020-2016-03}
\left|\frac{q}{2\epsilon^2}\int_{\widehat{r}_{\epsilon}}^R\lambda_{\epsilon,\kappa}(s)\dss\right|&\leq\widetilde{C}_{13}\epsilon^{\kappa_1(\epsilon)+\beta_1(\epsilon)-2}\log\frac{1}{\epsilon}\notag\\[-0.7em]
&\\[-0.7em]
&=\widetilde{C}_{13}e^{\left(\log\frac{1}{\epsilon}\right)^{1-\Theta}\left[-\frac{\gamma}{2}+\left(\log\frac{1}{\epsilon}\right)^{\Theta-1}\log^{(2)}\frac{1}{\epsilon}\right]}\longrightarrow0\quad\mathrm{as}\,\,\epsilon\downarrow0,\notag
\end{align}
and
\begin{align}\label{psi-2016-1020-1}
\int_{\widehat{r}_{\epsilon}}^R\frac{1}{\sqrt{g(t)}}e^{-\frac{q}{2\epsilon^2}\int_{\widehat{r}_{\epsilon}}^t\lambda_{\epsilon,\kappa}(s)\dss}\dt
=&\epsilon^{\beta_1(\epsilon)}\left(\frac{1}{\sqrt{g(R)}}+o_{\epsilon}(1)\right)\quad\mathrm{as}\,\,0<\epsilon\ll1.
\end{align}
Note that $1<\beta_1(\epsilon)<2$ as $0<\epsilon\ll1$.
Thus, by (\ref{psi-2016-3}), (\ref{psi-2016-5}), (\ref{1020-2016-03}) and (\ref{psi-2016-1020-1}) we obtain
\begin{align*}
\psie(R-\epsilon^{\beta_1(\epsilon)})=&\left(\frac{N}{R(B-A)}\sqrt{\frac{2Ag(R)}{q}}+o_{\epsilon}(1)\right)\\
&\quad\quad\quad+\epsilon^{\beta_1(\epsilon)-2}\left(\sqrt{\frac{qA}{2g(R)}}+o_{\epsilon}(1)\right).
\end{align*}
Along with (\ref{psi-2016-1}) gives
\begin{align}\label{psi-2016-8-1020-842am}
U_{\epsilon}(R-\epsilon^{\beta_1(\epsilon)})=&\frac{2}{q}\left(\log\epsilon+\log\psie(R-\epsilon^{\beta_1(\epsilon)})\right)\notag\\
=&\frac{2}{q}\Bigg[\log\epsilon+\log\left(\epsilon^{\beta_1(\epsilon)-2}\left(\sqrt{\frac{qA}{2g(R)}}+o_{\epsilon}(1)\right)\right)\notag\\
&\quad\quad\quad\,\,+\log\left(1+\frac{2\epsilon^{2-\beta_1(\epsilon)}Ng(R)}{qR(B-A)}\right)\Bigg]\\
=&\frac{2}{q}(\beta_1(\epsilon)-1)\log\epsilon+\frac{1}{q}\log\frac{qA}{2g(R)}+o_{\epsilon}(1)\notag\\
=&-\frac{2\gamma}{q}\left(\log\frac{1}{\epsilon}\right)^{1-\Theta}+\frac{1}{q}\log\frac{qA}{2g(R)}+o_{\epsilon}(1).\notag
\end{align}
Hence, we prove (\ref{01032017-2}).
Here we have used $\beta_1(\epsilon)\approx1$ as $0<\epsilon\ll1$ and $\log(1+\mu)\approx\mu$ for $|\mu|\ll1$
to get $\log\left(1+\frac{2\epsilon^{2-\beta_1(\epsilon)}Ng(R)}{qR(B-A)}\right)=o_{\epsilon}(1)$,
which gives the third equality of (\ref{psi-2016-8-1020-842am}).

By means of (\ref{2016-1010-guochin}), (\ref{etae-t}) with $\kappa=\kappa_1(\epsilon)$ and (\ref{psi-2016-8-1020-842am}) and passing through simple calculations directly, we conclude that
\begin{align}
&\left|\epsilon^{\beta_1(\epsilon)}U_{\epsilon}'(R-\epsilon^{\beta_1(\epsilon)})+\frac{2}{q}\right|\notag\\
=&
\left|\epsilon{e}^{-\gamma\left(\log\frac{1}{\epsilon}\right)^{1-\Theta}}U_{\epsilon}'(R-\epsilon^{\beta_1(\epsilon)})+\frac{2}{q}\right|\notag\\
\leq&\widetilde{C}_{12}e^{\left(\log\frac{1}{\epsilon}\right)^{1-\Theta}\left[-\frac{\gamma}{2}+\left(\log\frac{1}{\epsilon}\right)^{\Theta-1}\log^{(2)}\frac{1}{\epsilon}\right]}+o_{\epsilon}(1)\longrightarrow0\quad\mathrm{as}\,\,\epsilon\downarrow0.\notag
\end{align}
Hence, (\ref{1013-123}) immediately follows from (\ref{psi-2016-8-1020-842am}) and the above estimate.

\underline{\textbf{Case 2\'{}.}} For $\beta_2(\epsilon)=1+\left(\log\frac{1}{\epsilon}\right)^{-1}\left(\gamma\log^{(n)}\frac{1}{\epsilon}+\tau\right)$ with $\gamma>0$, $n\geq2$ 
and $\tau\in\mathbb{R}$, by (\ref{1020-2016-02})
we have 
\begin{equation*}
\kappa_2(\epsilon)+\beta_2(\epsilon)-2=\frac{\gamma}{2}\left(\log\frac{1}{\epsilon}\right)^{-1}\log^{(n)}\frac{1}{\epsilon}\longrightarrow0
\end{equation*}
 as $\epsilon\downarrow0$.
Note that $\beta_2(\epsilon)\approx1$ as $0<\epsilon\ll1$. Thus,
we may follow the same method as Case~1\'{} to get
\begin{align}\label{2016-1020-naf}
U_{\epsilon}(R-&\,\epsilon^{\beta_2(\epsilon)})\notag\\
=&\frac{2}{q}(\beta_2(\epsilon)-1)\log\epsilon+\frac{1}{q}\log\frac{qA}{2g(R)}+o_{\epsilon}(1)\\
=&-\frac{2\gamma}{q}\log^{(n)}\frac{1}{\epsilon}+\frac{1}{q}\left(-2\tau+\log\frac{qA}{2g(R)}\right)+o_{\epsilon}(1).\notag
\end{align} 
Hence, (\ref{01032017-3}) follows.
In addition, for $n=2$ and $\gamma>2$, by (\ref{2016-1010-guochin}), (\ref{etae-t}) with $\kappa=\kappa_2(\epsilon)$ and (\ref{2016-1020-naf}), we have 
\begin{align}
&\left|\epsilon^{\beta(\epsilon)}U_{\epsilon}'(R-\epsilon^{\beta_2(\epsilon)})+\frac{2}{q}\right|\notag\\
&\quad\quad\leq\widetilde{C}_{12}\left(\log^{(n-1)}\frac{1}{\epsilon}\right)^{-\frac{\gamma}{2}}\log\frac{1}{\epsilon}+o_{\epsilon}(1)\notag\\
&\quad\quad=\widetilde{C}_{12}\left(\log\frac{1}{\epsilon}\right)^{1-\frac{\gamma}{2}}+o_{\epsilon}(1)\longrightarrow0\quad\mathrm{as}\,\,\epsilon\downarrow0.\notag
\end{align}

Using (\ref{01032017-2}) and (\ref{01032017-3}) and following the same argument 
as (\ref{2017-0202-1}), we obtain the leading order terms of $\rho_{\epsilon}(R-\epsilon^{\beta_i(\epsilon)})$
and finish the proof of (\ref{1013-123}) and (\ref{1013-556}). Therefore,
 the proof of Theorem~\ref{thm3-new1106} is completed.

\section{A connection to the EDL capacitance}\label{sec-phys-app}
  Using the pointwise descriptions established in Theorems~\ref{thm2}--\ref{thm3-new1106},
we derive a formula for the quantity $\mathscr{C}^+(U_{\epsilon};[r_{\epsilon},R])$ (related to the EDL capacitance; see (\ref{capa-formula-new})). We show that the value
$\mathscr{C}^+(U_{\epsilon};[r_{\epsilon},R])\ll1$ 
if the thickness of the region~$[r_{\epsilon},R]\gg\epsilon^2$, and 
$\mathscr{C}^+(U_{\epsilon};[r_{\epsilon},R])\sim{O}(1)$ if the thickness of the region~$[r_{\epsilon},R]\sim\epsilon^2$  
as $0<\epsilon\ll1$. Moreover, 
when the thickness~$\gamma\epsilon^2$ of~$[r_{\epsilon}^{\gamma},R]$ becomes thinner,
 the value $\mathscr{C}^+(U_{\epsilon};[r_{\epsilon}^{\gamma},R])$
becomes higher and has an upper bounded for $0<\gamma\ll1$. Such results are stated
as follows:

\begin{theorem}\label{cor-20170206}
Let $0<A<B$.
Under the same hypotheses as in Theorem~\ref{thm3},
 we have the following properties for $\mathscr{C}^+(U_{\epsilon};[r_{\epsilon},R])$
as $0<\epsilon\ll1$:
\begin{itemize}
\item[(I)] If $r_{\epsilon}^{\infty}\in(0,R)$ satisfies $\displaystyle\lim_{\epsilon\downarrow0}\frac{R-r_{\epsilon}^{\infty}}{\epsilon^2}=\infty$,
then
\begin{equation*}
\lim_{\epsilon\downarrow0}\mathscr{C}^+(U_{\epsilon};[r_{\epsilon}^{\infty},R])=0.
\end{equation*}
\item[(II)] If $r_{\epsilon}^{\gamma}\in(0,R)$ satisfies $\displaystyle\lim_{\epsilon\downarrow0}\frac{R-r_{\epsilon}^{\gamma}}{\epsilon^2}=\gamma\in(0,\infty)$, then 
\begin{align}\label{20170207-1}
\displaystyle\lim_{\epsilon\downarrow0}\mathscr{C}^+(U_{\epsilon};[r_{\epsilon}^{\gamma},R])=
\frac{R^N(B-A)q}{2N\left(1+\frac{2Ng(R)}{R(B-A)\gamma{q}}\right)\log\left(1+\frac{R(B-A)\gamma{q}}{2Ng(R)}\right)}.
\end{align}
Moreover, we have the following comparisons:
\begin{itemize}
\item[(i)] $\mathscr{C}^+(U_{\epsilon};[r_{\epsilon}^{\gamma},R])$ is strictly decreasing to $\gamma$ in the sense that
for $0<\gamma^1<\gamma^2$, 
\begin{align}\label{1131-0414}
\mathscr{C}^+(U_{\epsilon};[r_{\epsilon}^{\gamma^2},R])<\mathscr{C}^+(U_{\epsilon};[r_{\epsilon}^{\gamma^1},R])<\frac{R^N(B-A)q}{2N}\quad{as}\,\,0<\epsilon\ll1.
\end{align}
\item[(ii)] $\mathscr{C}^+(U_{\epsilon};[r_{\epsilon}^{\gamma},R])$ is strictly increasing to $R$ in the sense that for $0<R_1<R_2$, there holds
\begin{equation*} 
\mathscr{C}^+(U_{\epsilon};[r_{\epsilon}^{\gamma},R_1])<\mathscr{C}^+(U_{\epsilon};[r_{\epsilon}^{\gamma},R_2]) 
\end{equation*}
as $0<\epsilon\ll1$.
\end{itemize}
\end{itemize}
For the case of $0<B<A$, similar results also hold.
 \end{theorem}
\begin{remark}\label{rk0207}
Note that by (\ref{ab0412}), the dimensionless quantity~$C^b:=R^N(B-A)$ is related to the total charge over the whole physical domain~$\mathbb{B}_R$.
Along with (\ref{20170207-1}),
we have the following approximation:
\begin{align}\label{20170413-1}
\mathscr{C}^+(U_{\epsilon};[r_{\epsilon}^{\gamma},R])=
\frac{C^bq}{2N\left(1+\frac{2NR^{N-1}g(R)}{C^b\gamma{q}}\right)\log\left(1+\frac{C^b\gamma{q}}{2NR^{N-1}g(R)}\right)}+o_{\epsilon}(1).
\end{align}
This formula is similar as the 
specific capacitance of the
Helmholtz double layer for the cylindrical electrode of radius~$R$ (cf. \cite{WP,WP2}). 
Moreover, assume $0<\epsilon\ll1$ and $0<\gamma\ll1$; namely, the region $[r_{\epsilon}^{\gamma},R]\approx[R-\gamma\epsilon^2,R]$
sufficiently attaching to the charged surface is quite thin. Then applying the approximation~$\log(1+s)\approx{s}$ for $|s|\ll1$ to (\ref{20170413-1}) gives
\begin{align}\label{11280413}
\frac{1}{\mathscr{C}^+(U_{\epsilon};[R-\gamma{\epsilon}^{2},R])}\approx\frac{\gamma}{R^{N-1}g(R)}+\frac{2N}{C^bq},
\end{align}
which presents the effects of the curvature~$\frac{1}{R}$, the surface dielectric constant~$g(R)$,
the thickness associated with $\gamma$, and the total charge~$C^bq$. 
\end{remark}


 Having Theorems~\ref{thm2}--\ref{thm3-new1106} at hand, we are now in a position to prove Theorem~\ref{cor-20170206}.
\begin{proof}[Proof of Theorem~\ref{cor-20170206}]
 Assume $0<A<B$. Multiplying (\ref{bigu1}) by $r^{N-1}$, 
	integrating the expression over $[r_{\epsilon},R]$ and using (\ref{bigu3}) and (\ref{1123-2016-1113pm}), one may check that
$\mathscr{C}^+(U_{\epsilon};[r_{\epsilon},R])$ obeys 
\begin{align}\label{capa-formula-new-1}
\mathscr{C}^+(U_{\epsilon};[r_{\epsilon},R])=\left|\frac{\frac{R^N(A-B)}{N}-\epsilon^2g(r_{\epsilon})r_{\epsilon}^{N-1}U_{\epsilon}'(r_{\epsilon})}{U_{\epsilon}(R)-U_{\epsilon}(r_{\epsilon})}\right|.
\end{align}

For the case $\displaystyle\lim_{\epsilon\downarrow0}\frac{R-r_{\epsilon}^{\infty}}{\epsilon^2}=\infty$,
 $r_{\epsilon}^{\infty}$ has to satisfy one of the following:
\begin{itemize}
\item[\textbf{(i).}] There exists $\widehat{\beta}\in[0,2)$ such that $\limsup_{\epsilon\downarrow0}\frac{R-r_{\epsilon}^{\infty}}{\epsilon^{\widehat{\beta}}}<\infty$ but $\lim_{\epsilon\downarrow0}\frac{R-r_{\epsilon}^{\infty}}{\epsilon^{\beta''}}=\infty$ for $\beta''\in(\widehat{\beta},2)$;
\item[\textbf{(ii).}] $R-r_{\epsilon}^{\infty}=\pi_2(\epsilon)$ satisfies (\ref{0121-2017-1}) with $\beta_0=2$.
\end{itemize}
Hence, by Theorems~\ref{thm2}(II) and \ref{thm3-new1106}(I-a)
and the monotonicity of $U_{\epsilon}$, one immediately finds
\begin{align}\label{today-2017-1}
|U_{\epsilon}(R)-U_{\epsilon}(r_{\epsilon}^\infty)|\geq
\begin{cases}
\displaystyle\frac{2}{q}(2-\beta'')\log\frac{1}{\epsilon}+O(1),\quad\mathrm{for\,\,case~(i)},\vspace{5pt}\\
\displaystyle\frac{2}{q}\log\frac{\pi_2(\epsilon)}{\epsilon^2}+O(1),\quad\mathrm{for\,\,case~(ii)},
\end{cases}
\end{align}
for any $\beta''\in(\max\{1,\widehat{\beta}\},2)$.

On the other hand, by (\ref{thm2-0911-1}), Theorems~\ref{thm2}(II-a) and \ref{thm3-new1106}(I-a), one may check that
\begin{align}\label{today-2017-2}
0\leq-\epsilon^2g(r_{\epsilon}^\infty)\cdot\left(r_{\epsilon}^{\infty}\right)^{N-1}U_{\epsilon}'(r_{\epsilon}^\infty)\leq
\frac{2g(R)R^{N-1}\epsilon^2}{q\pi_2(\epsilon)}(1+o_{\epsilon}(1)),\quad\mathrm{as}\,\,0<\epsilon\ll1.
\end{align}
Note that $\displaystyle\lim_{\epsilon\downarrow0}\frac{\epsilon^2}{\pi_2(\epsilon)}=0$. Thus, by 
(\ref{capa-formula-new-1})--(\ref{today-2017-2}) we find
\begin{align}\label{capa-formula-new-2}
\mathscr{C}^+(U_{\epsilon};[r_{\epsilon}^\infty,R])=\left|\frac{\frac{R^N(A-B)}{N}-\epsilon^2g(r_{\epsilon}^\infty)\cdot(r_{\epsilon}^\infty)^{N-1}U_{\epsilon}'(r_{\epsilon}^\infty)}{U_{\epsilon}(R)-U_{\epsilon}(r_{\epsilon}^\infty)}\right|\longrightarrow0\quad\mathrm{as}\,\,\epsilon\downarrow0.
\end{align}
Therefore, Theorem~\ref{cor-20170206}(I) follows.

Now we give the proof of Theorem~\ref{cor-20170206}(II).
Assume $\displaystyle\lim_{\epsilon\downarrow0}\frac{R-r_{\epsilon}^{\gamma}}{\epsilon^2}=\gamma\in(0,\infty)$. 
Then setting $r_{\epsilon,\beta}=r_{\epsilon}^{\gamma}$
and $\gamma_{\beta}=\gamma$ in (\ref{1105-2016-ab1}) with $\beta=2$ and in Theorem~\ref{thm2}(II-b),
and putting these expansions into (\ref{capa-formula-new-1}),
one may check that 
\begin{align}\label{20170207-af}
\mathscr{C}^+(U_{\epsilon};[r_{\epsilon}^\gamma,R])=&\left|\frac{\frac{R^N(A-B)}{N}-\epsilon^2g(r_{\epsilon}^\gamma)\cdot(r_{\epsilon}^\gamma)^{N-1}U_{\epsilon}'(r_{\epsilon}^\gamma)}{U_{\epsilon}(R)-U_{\epsilon}(r_{\epsilon}^\gamma)}\right|\notag\\
=&\frac{\left|\frac{R^N(A-B)}{N}+\frac{2g(R)R^{N-1}}{\gamma{q}+\frac{2Ng(R)}{R(B-A)}}+o_{\epsilon}(1)\right|}
{\frac{2}{q}\left[\log\left(\gamma{q}+\frac{2Ng(R)}{R(B-A)}\right)-\log\frac{2Ng(R)}{R(B-A)}\right]+o_{\epsilon}(1)}\\
=&\frac{R^N(B-A)q}{2N\left(1+\frac{2Ng(R)}{R(B-A)\gamma{q}}\right)\log\left(1+\frac{R(B-A)\gamma{q}}{2Ng(R)}\right)}+o_{\epsilon}(1),\notag
\end{align}
which gives (\ref{20170207-1}).
Here we have used $\lim_{\epsilon\downarrow0}g(r_{\epsilon}^\gamma)=0$ and the asymptotic expansion of $U_{\epsilon}(R)$ 
from (\ref{1105-2016-ab1}) with $\boldsymbol{\chi}_{1}(\beta)=0$ 
and $\boldsymbol{\chi}_{2}(\beta)=1$  (see also, (\ref{0916-thm2-1}))
to get the second line of (\ref{20170207-af}).

Now we deal with the monotonicity of $\mathscr{C}^+(U_{\epsilon};[r_{\epsilon}^\gamma,R])$ 
with respect to $\gamma$. Let
\begin{equation*}
 \widetilde{f}(\gamma)=\left(1+\frac{C}{\gamma}\right)\log\left(1+\frac{\gamma}{C}\right)
\end{equation*}
be a function of $\gamma$,
where $C=\frac{2Ng(R)}{R(B-A){q}}>0$.
One may check that $\widetilde{f}'(\gamma)=-\frac{C}{\gamma^2}\log\left(1+\frac{\gamma}{C}\right)+\frac{1}{\gamma}>0$
for $\gamma>0$,
which implies that $\widetilde{f}(\gamma)$ is strictly increasing to $\gamma$,
and satisfies
\begin{align}\label{ccc-14-2017}
\widetilde{f}(\gamma)>\lim_{t\to0+}\left(1+\frac{C}{t}\right)\log\left(1+\frac{t}{C}\right)=1,\quad\mathrm{for}\,\,\gamma>0.
\end{align}
By (\ref{20170207-af}) and (\ref{ccc-14-2017}) with $C=\frac{2Ng(R)}{R(B-A){q}}$,
we obtain~(\ref{1131-0414}).
 This completes the proof of Theorem~\ref{cor-20170206}(II-i).

It remains to prove Theorem~\ref{cor-20170206}(II-ii).
Let 
\begin{equation*}
\widehat{f}(R)=\frac{R^N(B-A)q}{2N\left(1+\frac{2Ng(R)}{R(B-A)\gamma{q}}\right)\log\left(1+\frac{R(B-A)\gamma{q}}{2Ng(R)}\right)}
\end{equation*}
be a function of $R$. By a simple calculation, one may obtain
\begin{align}\label{0412-n}
\widehat{f}'(R)=\frac{\widehat{f}(R)}{R+\frac{2Ng(R)}{(B-A)\gamma{q}}}\left[(N+1)\left(1+\frac{2Ng(R)}{R(B-A)\gamma{q}}\right)-1-\frac{1}{\log\left(1+\frac{R(B-A)\gamma{q}}{2Ng(R)}\right)}\right].
\end{align}
On the other hand, it is easy to check that for any $a_0>\frac{3}{2}$, there holds
\begin{align}\label{0412-t}
\log\left(1+\frac{1}{t}\right)>\frac{1}{a_0(1+t)-1},\quad\forall\,t>0.
\end{align}
In particular, choosing $a_0=N+1$ and
setting $t=\frac{2Ng(R)}{R(B-A)\gamma{q}}$ in (\ref{0412-t}),
one immediately gets 
\begin{equation*}
(N+1)\left(1+\frac{2Ng(R)}{R(B-A)\gamma{q}}\right)>1+\frac{1}{\log\left(1+\frac{R(B-A)\gamma{q}}{2Ng(R)}\right)}.
\end{equation*}
(Note that $0<A<B$.)
Along with (\ref{0412-n}), we arrive at $\widehat{f}'(R)>0$ and obtain $\mathscr{C}^+(U_{\epsilon};[r_{\epsilon}^{\gamma},R_1])<\mathscr{C}^+(U_{\epsilon};[r_{\epsilon}^{\gamma},R_2])$ as $0<\epsilon\ll1$.

Therefore, we prove Theorem~\ref{cor-20170206}(II-ii)
and complete the proof of Theorem~\ref{cor-20170206}.
\end{proof}


\textbf{Acknowledgments.} This work was partially supported by the MOST grant 106-2115-M-007-007 of Taiwan. The author is much obliged to Professors Ping Sheng and Chien-Cheng Chang for their valuable knowledge on the physical background of the EDL and supercapacitor, and to the referee for the invaluable comments that improved the manuscript. His special thank also goes to Professors Tai-Chia Lin and Chun Liu for pointing out the importance of this non-local model.

\footnotesize


\end{document}